\documentclass[12pt,reqno]{amsart}

\usepackage{amsmath}
\usepackage{amsthm,amssymb,color,enumerate,csquotes,placeins,mathtools}
\usepackage[hyperindex,pagebackref=false, citecolor=cyan,linkcolor=cyan]{hyperref}
\usepackage{tikz-cd}
\usetikzlibrary{patterns}
\pgfdeclarepatternformonly{mydots}{\pgfqpoint{-1pt}{-1pt}}{\pgfqpoint{1pt}{1pt}}{\pgfqpoint{3pt}{3pt}}
{%
  \pgfpathcircle{\pgfqpoint{0pt}{0pt}}{.2pt}%
  \pgfusepath{fill}%
}%

\usepackage{enumitem}
\usepackage{mathrsfs} 
\usepackage{mathtools}
\mathtoolsset{showonlyrefs} 
\usepackage[margin=1in]{geometry}

\newcommand{\G}{\mathcal G}
\newcommand{\chp}{\mathrm{Ch}}
\newcommand{\Pic}{\operatorname{Pic}} 
\newcommand{\norm}[1]{{|{#1}|}}
\newcommand{\Div}{\operatorname{Div}} 
\newcommand{\lb}{\mathcal L}
\newcommand{\dlb}{\lb_G}
\newcommand{\charfn}{1} 
\newcommand{\ccl}{N} 
\newcommand{\cl}{M} 
\newcommand{\ccs}{\mathfrak t} 
\newcommand{\cs}{\ccs^\vee} 
\newcommand{\lt}{\ell} 
\newcommand{\Lt}{\ell} 
\newcommand{\MaximalCompact}[1]{K^{\mathrm{max}}_{#1}} 
\newcommand{\IntegralCompact}[1]{K^{D}_{#1}} 
\newcommand{\Places}[1]{P_{#1}} 
\newcommand{\ToricIdeals}{\mathscr N}
\newcommand{\ToricPrinIdeals}{\mathscr P}
\newcommand{\ToricFracIdeals}{\mathscr Q}
\newcommand{\ToricUnits}{\Upsilon}
\newcommand{\ToricROU}{\mu}
\newcommand{\itres}{\operatorname{itres}}
\newcommand{\res}{\operatorname{res}}

\newcommand{\defn}[1]{\emph{#1}}

\newcommand{\ZZ}{\mathbb Z}
\newcommand{\QQ}{\mathbb Q}
\newcommand{\RR}{\mathbb R}
\newcommand{\CC}{\mathbb C}
\renewcommand{\AA}{\mathbb A} 
\newcommand{\PP}{\mathbb P}

\newcommand{\GG}{\mathbb G}

\newcommand{\mcO}{\mathcal O} 

\newcommand{\Gal}{\operatorname{Gal}} 
\newcommand{\Hom}{\operatorname{Hom}} 
\newcommand{\im}{\operatorname{im}} 

\newcommand{\Spec}{\operatorname{Spec}\,}
\renewcommand{\H}{\operatorname{H}}

\newcommand{\GL}{\operatorname{GL}}

\newcommand{\N}{\mathcal N}

\DeclareMathOperator{\lcm}{lcm}
\newcommand{\Mod}[1]{\ (\mathrm{mod}\ #1)}

\newcommand{\tik}{\[\begin{tikzcd}}
\newcommand{\zcd}{\end{tikzcd}\]}

\newtheorem{theorem}{Theorem}
\newtheorem*{theorem*}{Theorem}

\newtheorem{lemma}{Lemma}
\newtheorem{proposition}{Proposition}
\newtheorem*{proposition*}{Proposition}
\newtheorem{corollary}{Corollary}

\newtheorem*{prop*}{Proposition}

\theoremstyle{definition} 
\newtheorem{definition}{Definition}

\newtheorem*{example*}{Example}
\newtheorem{example}{Example}

\newtheorem{warning}{Warning}
\newtheorem*{warning*}{Warning}

\theoremstyle{remark} 

\newtheorem{remark}{Remark}

\DeclareFontFamily{U}{wncy}{}
\DeclareFontShape{U}{wncy}{m}{n}{<->wncyr10}{}
\DeclareSymbolFont{mcy}{U}{wncy}{m}{n}
\DeclareMathSymbol{\Sh}{\mathord}{mcy}{"58} 

\usetikzlibrary{decorations.markings}
\tikzset{double line with arrow/.style args={#1,#2}{decorate,decoration={markings,%
mark=at position 0 with {\coordinate (ta-base-1) at (0,1pt);
\coordinate (ta-base-2) at (0,-1pt);},
mark=at position 1 with {\draw[#1] (ta-base-1) -- (0,1pt);
\draw[#2] (ta-base-2) -- (0,-1pt);
}}}}

\makeatletter
\newtheorem*{rep@theorem}{\rep@title}
\newcommand{\newreptheorem}[2]{%
\newenvironment{rep#1}[1]{%
 \def\rep@title{#2 \ref*{##1}}%
 \begin{rep@theorem}}%
 {\end{rep@theorem}}}
\makeatother
\newreptheorem{theorem}{Theorem}

\title{Integral points on singular toric varieties 
and cyclic normal polynomials}

\author{Andrew O'Desky}
\address{
\parbox{0.7\linewidth}{
    Department of Mathematics\\
    Princeton University\\
    Princeton, NJ 08544-1000, USA\\[.1em]
}}
\email{andy.odesky@gmail.com}

\date{\today. Research supported by NSF grant DMS-2103361.}

\begin{document}

\begin{abstract}
We establish a formula for the height zeta function 
    for integral points on a 
    class of projective toric varieties. 
Our method builds on the harmonic analysis 
    approach of Batyrev--Tschinkel for rational points 
    and is applicable even when the toric variety 
    has cyclic quotient singularities. 
As an application, we determine the leading term 
    in the asymptotic number 
    of monic integral polynomials 
    of bounded height 
    with linearly independent roots 
    and a given cyclic Galois group. 
\end{abstract}


\maketitle

\setcounter{tocdepth}{1}
\tableofcontents

\section{Introduction}  


Let $Y$ be a projective normal toric variety defined over $\QQ$ 
    with at worst cyclic quotient singularities. 
Let $T \subset Y$ be the open torus of $Y$. 
The problem of computing the asymptotic number 
    of rational points on $T$ 
    of bounded height was taken up in \cite{IMRN}. 
Their strategy was to extend the height function 
    to the space of adelic points $T(\AA)$ 
    and exploit the spectral decomposition of 
    the $T(\QQ)$-invariant function 
\begin{equation}\label{eqn:TQInvariantFn} 
    g \mapsto 
    \sum_{t \in T(\QQ)}
    \frac{1}{H(gt,s)} 
\end{equation} 
using the Poisson summation formula. 
The quantity $s$ denotes an element 
    of the complex vector space $\Pic(Y) \otimes \CC$. 
If the real part of $s$ is sufficiently positive 
    with respect to the effective cone 
    then the sum 
    converges absolutely,  
    and the Poisson formula implies that 
\begin{equation} 
    \sum_{t \in T(\QQ)}
    \frac{1}{H(gt,s)} 
    =
    \int_{(T(\AA)/T(\QQ))^\vee}
    \widehat{H}(\chi,-s)
    \chi(g)
    \,d\chi
\end{equation} 
where $\widehat{H}(\chi,-s)$ is the Fourier transform 
    of $H(\cdot,-s) = H(\cdot,s)^{-1}\in L^2(T(\AA))$. 
When $g = 1$ and $s$ is proportional to 
    the class $[L]$ of a big line bundle, 
    this multiple Dirichlet series specializes to 
    the single Dirichlet series (height zeta function) 
\begin{equation} 
    Z(s)=
    \sum_{t \in T(\QQ)}
    \frac{1}{H(t,s)} 
\end{equation} 
for rational points on $T$ of bounded $L$-height. 
When $T(\AA)/T(\QQ)$ is compact and $Y$ is 
    smooth---as assumed in \cite{IMRN}---the 
    terms on the spectral side 
    of the Poisson formula are easily isolated 
    and expressed using Hecke $L$-functions. 
From the principal terms one can deduce 
    the singularities of the height zeta function 
    using analytic properties of 
    Dedekind zeta functions. 

Difficulties arise 
    when $T(\AA)/T(\QQ)$ is not compact or $Y$ is not smooth. 
In the non-compact case, 
    the spectral side of the Poisson formula 
    involves an integral over the continuous part of 
    the automorphic spectrum of $T$, 
    and it is difficult to determine 
    where $Z(s)$ has singularities from 
    the Fourier transform $\widehat{H}(\chi,-s)$. 
In their follow-up papers \cite{HZF}, \cite{zbMATH01353487}  
    treating the smooth non-compact case, 
    this integral over the continuous spectrum was 
    approximated by relating 
    $\widehat{H}(\chi,-s)$ to simpler functions 
    constructed using certain cones 
    (``$\mathcal X$-functions''). 
By their method they inferred 
    the abscissa of convergence for $Z(s)$, 
    however their method has not been successfully applied 
    to compute either secondary terms 
    for the distribution of rational points, 
    or leading terms 
    for the distribution of integral points. 

Let $D$ be a $T$-stable divisor of $Y-T$ 
and let $U$ be an integral model of $Y-D$. 
A rational point of $T$ is called $D$-integral 
    if it extends to an integral point of $U$. 
The present paper is concerned with the height zeta function 
in the integral singular non-compact setting: 
\begin{equation} 
    Z(s) = 
    \sum_{\substack{t \in T(\QQ)\\\text{$D$-integral}}}
    \frac{1}{H(t,s)}. 
\end{equation} 
The Poisson strategy was applied in 
    the integral smooth non-compact setting 
    to study $Z(s)$ in \cite{integral-chambert-tschinkel}. 
Unfortunately their application of the Poisson strategy is flawed 
    and has a published counterexample \cite{zbMATH07814403} 
    proven by a different method. 
According to \cite[Remark~3.2.2]{zbMATH07814403}, 
    there is a gap in the proof of 
    \cite[Lemma~3.11.4]{integral-chambert-tschinkel} 
    where it is assumed that certain $\mathcal X$-functions 
    are non-vanishing. 
    

In this paper we do not use $\mathcal X$-functions 
    and thus do not follow the strategy of 
    \cite{HZF}, 
    \cite{zbMATH01353487}, 
    \cite{integral-chambert-tschinkel} 
    in treating the integral over the continuous spectrum 
    of $T(\AA)/T(\QQ)$. 
The form to be integrated is $\widehat{H}(\chi,-s)\,d\chi$ 
    which has singularities along affine hyperplanes 
    depending on the fan of $Y$. 
This is a strong constraint, 
    and our strategy is to use 
    the multidimensional residue formula 
    from \cite{residue} 
    which applies to such integrals. 
Poles of meromorphic forms with hyperplane singularities 
    are indexed by \emph{flags}, 
    and our residue formula 
    shows that only a small subset of flags 
    contribute a residue to the integral 
    which does not cancel identically with 
    residues from other flags. 
Since this residue formula does not involve approximation, 
    we obtain an {exact} formula for $Z(s)$ 
    in terms of simpler multiple Dirichlet series 
    and rational functions that can be computed 
    from the fan of $Y$. 
Our residue formula requires 
    that the fan of $Y$ satisfies 
    a certain compatibility condition 
    (see \S\ref{sec:methodIterRes}). 

\begin{theorem}\label{thm:hzfformulaintro} 
Assume the residue method is applicable for the fan $\Sigma$, 
    i.e.~each $\Sigma$-convex singular flag is compatible 
    with the polyhedron $\Pi_\sigma$ 
    for every maximal cone $\sigma \in \Sigma_D(r)$. 
Then the height zeta function $Z(s)$ 
    for $D$-integral rational points 
    on the torus 
is equal to 
\begin{align}
    Z(s) =
    \frac{|\ToricROU|}{|C||R|P}
    \sum_{\xi\in R^\vee}
    \sum_{\psi\in C^\vee}
    \sum_{\tau \in \Sigma_D}
    \sum_{\gamma \in Z_{\sigma_\tau}}
    L(s+2\pi im_{\gamma \xi},(\xi\psi)^{-1},-\tau^\circ,D)
    R_\gamma(s+2\pi im_\xi)
\end{align} 
    in the region of $s\in\Pic(Y) \otimes \CC$ 
    where both sides converge. 
\end{theorem} 

The arithmetic and combinatorial invariants 
    appearing in the formula can be described as follows: 
\begin{itemize} 
    \item[--] $\mu$ is the torsion subgroup 
            of $T(\QQ) \cap \left(T(\RR) \times \prod_{v\text{ finite}} K_v\right)$ where $K_v$ is the open compact subgroup of 
            $T(\QQ_v)$ stabilizing $U(\ZZ_v)$ 
            (\S\ref{sec:regulatorGroup}), 
    \item[--] $|R|$ is the volume of 
        the level $K$ regulator group $R$ of the torus $T$, 
        i.e.~the identity component of 
        $T(\QQ)\backslash T(\AA)^1/K$ 
        where $T(\AA)^1$ is the norm-one subgroup of $T(\AA)$ 
        and $K = \MaximalCompact{\infty} \times 
        \prod_{v\text{ finite}}K_v$ 
            (\S\ref{sec:normOne}), 
    \item[--] $C$ is the level $K$ class group of the torus $T$, 
        i.e.~the quotient of $T(\QQ)\backslash T(\AA)^1/K$ 
        by $R$ (\S\ref{sec:normOne}), 
    \item[--] $(-)^\vee$ denotes Pontryagin dual, 
    \item[--] $\Sigma_D$ is a subfan of $\Sigma_\Gamma$ 
        determined by $T, D$ and $\Sigma$ 
        (\S\ref{sec:theMainTheorem}) where 
    \item[--] $\Sigma_\Gamma$ is the restricted fan of $\Sigma$ 
        with respect to the action of the Galois group 
        $\Gamma$ of the splitting field of $T$ 
        (\S\ref{sec:restrictedFan}), 
    \item[--] $\sigma_\tau$ is any maximal cone of $\Sigma_D$ 
        containing $\tau$, 
    \item[--] $Z_{\sigma}$ is 
            the subset of flags terminating in $\Pi_\sigma$ 
            cut out by a collection 
            of $\Pi_\sigma$-stable singular hyperplanes 
            of the meromorphic form 
            $\widehat{H_\infty}(\chi_m,-s)\,dm$
            for a cone $\sigma$ in $\Sigma_\Gamma(r)$ 
            (\S\ref{sec:methodIterRes}) where 
    \item[--] $\Pi_{\sigma} = \cs + i\sigma^\vee 
    \subset \cs \otimes \CC$  
    is a polyhedron with boundary $\cs$ 
        (\S\ref{sec:polyhedra}) 
    where 
    \item[--] $\cs$ is the Galois invariant subspace 
        of the real vector space $\cs_E= \RR \otimes \ccl_E$ 
        where $E$ is the splitting field of $T$ and 
        $\ccl_E$ is the cocharacter lattice 
        of $T_E$ (\S\ref{sec:preliminaries}), 
    \item[--]$m_{\gamma\xi}$ is the terminal point 
        of the flag $\gamma$ of 
        $\widehat{H_\infty}(\chi_m,-(s+2\pi im_\xi))\,dm$ 
        (\S\ref{sec:ConesForms}) 
        and $m_\xi$ is the local exponent of $\xi$, 
    \item[--] $\tau^\circ=\tau - \cup_{\tau' \subsetneq \tau} \tau'$ 
        is the interior of a cone $\tau$, 
    \item[--] $\Sigma(1)$ is the set of generators of $\Sigma$, 
    \item[--] $L(s,\xi,\sigma,D)$ 
            is a certain multiple Dirichlet series 
            in $s \in \CC^{\Sigma(1)}$ 
            (\S\ref{sec:MultipleDirSeries}) 
    \item[--] $R_\gamma(s)$ is a certain 
        rational function on $\Pic(Y)$ 
        defined by Definition~\ref{defn:ResidualFunction} 
        and given explicitly by 
        \eqref{eqn:ResidualFunctionGeneral} 
        (\S\ref{sec:residual}), 
    \item[--] and $P$ is a period determined 
        by the fan given by 
\begin{equation} 
    P=\sum_{\gamma \in Z_{\sigma_1}}
    \int_{\ccs_0^\vee}
    R_\gamma(s_{AC}+2\pi im)\,dm
\end{equation} 
where $s_{AC}$ is the element of $\ZZ^{\Sigma(1)}$ 
corresponding to the sum of the toric divisors of $Y$, 
$\sigma_1$ is an arbitrary maximal cone of $\Sigma_D$, 
and $\ccs_0$ is the trace-zero subspace of $\ccs$ 
            (\S\ref{sec:regulatorGroup}).  
\end{itemize} 

Our work is incomplete in the sense that 
    we do not make contact with Peyre's formula 
    for the leading constant in the rational case, 
    nor Batyrev--Manin's conjecture 
    on the shape of the leading term. 
Both conjectures have recently been formulated 
    in the integral and smooth setting 
    with a proof using universal torsors 
    for toric varieties 
    in a recent preprint \cite{santens}. 
One would like to verify that 
    the leading term predicted by our formula 
    in the case that $Y$ is smooth 
    agrees with \cite{santens} 
    though we have yet to carry out this important verification. 

\subsubsection{Explicit formulas for summatory functions} 

In favorable cases, 
    applying Perron's formula to our formula for $Z(s)$ 
    results in explicit formulas for 
    the sum of Dirichlet coefficients 
    in terms of automorphic data. 
These formulas are quite complicated 
    but can sometimes be analyzed to extract main terms 
    and secondary terms. 
The idea is that one expects 
    the ``conical'' multiple Dirichlet series 
    $L(s+2\pi im_{\gamma \xi},(\xi\psi)^{-1},-\tau^\circ,D)$ 
    in the formula for $Z(s)$ 
    to be nearly automorphic 
    and the rational functions $R_\gamma$ 
    to be nonvanishing. 
(We have been able to prove that $R_\gamma$ 
    is nonvanishing under a strong convexity assumption 
    (Proposition~\ref{prop:nonvanishingR}) 
    though we expect it to hold generally.) 

We illustrate our formula with 
    the real quadratic Severi--Brauer surface 
    associated to a real quadratic field $E$.\footnote{For imaginary quadratic Severi--Brauer surfaces the height zeta function is a specialization of a real analytic Eisenstein series for $\GL_2$ and the rational functions $R_\gamma$ all have degree zero, so it is ``not interesting'' as far as illustrating our formula. For quasi-split tori the function $R_\gamma$ is the reciprocal of a polynomial whose degree is the rank of the unit group, and real quadratic fields are the first interesting examples.} 
    See Figure~\ref{fig:quadraticSB} for the fan. 
For an appropriate boundary divisor, 
    the integral Diophantine problem has the nonzero 
    algebraic integers in $E$ as its solutions, 
    and the height zeta function is, by definition, 
\begin{equation} 
    Z(s) = 
    \sum_{\substack{a \in O_E\\ a \neq 0}}
    \frac{1}{\max(|a|,|a'|)^{s}}
\end{equation} 
where $a'$ is the conjugate of $a$. 
The regulator group $R$ is $\RR/\ZZ\log u$ 
where $u$ is a positive generator of $O_E^\times$. 
The regulator characters are $\xi^k$ 
for integers $k$ 
where $\xi(y) = |y|^{2\pi i/\log u}$. 
The fan $\Sigma_D$ consists of a single cone $\sigma$, 
and $Z_\sigma$ consists of a single flag. 
Our formula is 
\begin{equation} 
    Z(s)=
    \frac{4}{hR}
    \sum_{k \in \ZZ}
    \sum_{\psi \in C^\vee}
    L(\tfrac{s}{2},\psi\xi^k)
    \frac{s}{s^2+\left(\frac{2\pi k}{R}\right)^2}
\end{equation} 
where $h$ denotes the class number, $R$ the regulator,  
and $L(s,\chi)=\sum_{I \subset O_E}\chi(I)N(I)^{-s}=\sum_N b_N(\chi) N^{-s}$ is the Hecke $L$-function of $\chi$. 
Since $Y$ has Picard number one, 
the conical Dirichlet series reduce to 
one-variable Dirichlet series. 
This is a useful formula since we have related the distribution 
of integral points of bounded height---a global problem---to 
the distribution of ideals of bounded norm, 
and the latter can be described with Hecke characters. 
Now we apply Perron's formula 
to obtain an explicit formula for 
the number of algebraic integers in $E$ 
of bounded height in terms of automorphic data: 
\begin{align} 
    \#\{a \in O_E : 0< |a|,|a'|<X\}
    =
    &\frac{4}{hR}
    \sum_{\psi \in C^\vee}
    \sum_{N<X^2}
        b_N(\psi)
        \log(X/\sqrt{N})\\
    +&
    \frac{8\pi}{hR^2}
    \sum_{\substack{k \in \ZZ\\ k \neq 0}}
    \sum_{\psi \in C^\vee}
    \sum_{N<X^2}
        k
        b_N(\psi\xi^k)
        \sin\left(\frac{2\pi k}{R}\log(\sqrt{N}/X)\right).
\end{align} 
With the class number formula 
    we can verify that the first term is the principal term: 
\begin{equation} 
    \frac{4}{hR}
    \sum_{\psi \in C^\vee}
    \sum_{N<X^2}
        b_N(\psi)
        \log(X/\sqrt{N})
        =\frac{2}{hR} \res[\zeta_E] \cdot X^2
        +O(X^{2-\delta})
        =\frac{4}{\sqrt{d_E}} \cdot X^2
        +O(X^{2-\delta}),
\end{equation} 
which is the expected number of points 
of a lattice with covolume $\sqrt{d_E}$ 
in a box of sidelength $2X$. 
The contribution of the rational functions 
$R_\gamma=\frac{s}{s^2+\left(\frac{2\pi k}{R}\right)^2}$ 
    to this formula 
    are the $\log(X/\sqrt{N})$ 
    and $\sin\left(\frac{2\pi k}{R}\log(\sqrt{N}/X)\right)$ 
    terms (obtained from well-known inverse Mellin transforms). 

Other explicit formulas related to cubic abelian fields 
    have been derived in 
    \cite{CUBIC}, \cite{CF} by this method. 
In \cite{CUBIC} this was used 
    to obtain leading and secondary terms 
    for the number of integral monic trace-one 
    cubic abelian polynomials of bounded root height.\footnote{The \emph{root height} of a polynomial 
$f=t^n+a_1t^{n-1}+\cdots+a_n \in \ZZ[t]$ 
is $H(f) = \max(|a_1|,|a_2|^{1/2},\ldots,|a_n|^{1/n})$.} 



\subsection{Cyclic normal polynomials} 

We apply our formula to a problem 
    in arithmetic statistics: 
    counting monic integral polynomials 
    with a given cyclic Galois group 
    and linearly independent roots in their splitting field. 
We call an irreducible polynomial \defn{normal} 
    if its roots are $\QQ$-linearly independent 
    in the splitting field. 
Normal polynomials 
    with a given Galois group 
    admit an orbit parametrization \cite{moge}. 
Let $G$ be a finite group. 
The orbit parametrization from \cite{moge} 
    is similar to a prehomogeneous space 
    in that it is a projective variety $Y=Y_G$ 
    with an action of a reductive group $\G$ 
    (the unit group of the group algebra of $G$) 
    and containing a dense open $\G$-orbit $T \subset Y$, 
    however the orbit $T$ is not generally birational 
    to projective space. 
The rational points of $T$ parametrize 
    data $(K/\QQ,x)$ where $K/\QQ$ is 
    an \'etale $\QQ$-algebra with Galois group $G$ 
    and $x$ is a normal element of $K$. 
(An element of $K$ is \defn{normal} if its $G$-conjugates 
    are $\QQ$-linearly independent.) 
We recall this construction in \S\ref{sec:orbitparam}. 
When $G$ is abelian this orbit parametrization $Y$ 
    has a natural toric structure, 
    and we may apply the methods of this paper 
    to study the distribution of such polynomials. 

\begin{theorem}\label{thm:mainthm} 
Let $n \geq 5$. 
Let $N(C_n,H)$ be the number of normal monic degree $n$ 
    integral polynomials with Galois group $C_n$ 
    and root height less than $H$. 
Then 
\begin{equation} 
    N(C_n,H) = 
    \varphi(n)^{-1}
    \frac{|\ToricROU||C \cap \ker \gamma^\ast|}{|C||R|P}
    \kappa_0
    \left(
    \sum_{d|n}
    \kappa_d
    \right)
    \tfrac{1}{n}
    H^{n}(1+o(1))
\end{equation} 
where 
\begin{equation} 
    \kappa_0=
    \left(\prod_{d|n}\res[\zeta_{\QQ(\zeta_d)}]\right)
    \prod_p
    Q_p(p^{-\frac1{e_p}s_{AC}})
\end{equation} 
and 
\begin{equation} 
    \kappa_d=\left(\frac{\varphi(d)}{2n}\right)^{\frac{1}{2}(1-\varphi(n))}
    \sum_{\tau \in \Sigma_{\Gamma_\infty}(r_\infty)}
    \binom{\tau}{\sigma_{(d)}}
    \frac{[\ccl_\infty:\ZZ\tau(1)]}{[\ccl:\ZZ\sigma_{(d)}(1)]}.
\end{equation} 
\end{theorem} 

The factors of the leading constant may be interpreted as follows: 
\begin{itemize} 
    \item[--] $\varphi(n)$ is Euler's totient function, 
        and also the number of twists of a rational point 
        $(K/\QQ,x)$ for a $C_n$-field $K$ 
        by an outer automorphism of $C_n$; 
        there are $\varphi(n)$ such rational points 
        of $T$ giving rise to any $C_n$-polynomial 
        so we divide by this factor 
        (cf. Corollary~\ref{cor:cyclicTwists} 
        or the examples in \cite[\S2]{CUBIC}), 
    \item[--] the open torus $T\subset Y$ is 
        the quotient torus $T=\mathcal G/C_n$ 
        where $\mathcal G$ is the group of units 
        in the group algebra of $C_n$, 
    \item[--]$\frac{|\ToricROU|}{|C||R|P}$ 
        comes from our formula for $Z(s)$, 
    \item[--] $\gamma \colon P=\GG_m \times \mathcal G 
        \to \mathcal G \to T$ where  $\mathcal G \to T$ 
        is the natural quotient homomorphism 
        and $\gamma^\ast \colon T(\AA)^\vee \to P(\AA)^\vee$ 
        is the pullback homomorphism on characters 
        (\S\ref{sec:weak}), 
    \item[--]$|C \cap \ker \gamma^\ast|$ is the number of 
        Hecke characters on $R \times C$ which contribute to 
        the principal terms, 
    \item[--] $\kappa_0$ is part of the leading constant 
        of the principal conical Dirichlet series,  
    \item[--] $Q_p$ is a polynomial 
        determined by the fan 
        which measures the singularities of $Y$;  
        it only depends on the decomposition group at $p$ 
        in the splitting field $E/\QQ$ 
        (Proposition~\ref{prop:localEulerFactor}), 
    \item[--] $\sum_d\kappa_d$ is the 
        sum of the values of the rational functions $R_\gamma$ 
        at the abscissa of convergence, 
    \item[--] $N = N_E^\Gamma$ and $N_\infty = N_E^{\Gamma_\infty}$ 
        where $\Gamma_\infty$ is 
        the decomposition group at infinity, 
    \item[--] $r_\infty=\tfrac12(1+\sum_{d|n}\varphi(d))$ 
        is the rank of $\prod_{d | n}\ZZ[\zeta_d]^\times$, 
    \item[--] $\Sigma$ is the fan for $Y$ 
        constructed in \S\ref{sec:DiophProblem}, 
    \item[--] $\sigma_{(d)}$ is the cone of $\Sigma_D$ 
        corresponding to the unique Galois orbit 
        in $C_n^\vee \subset \Sigma(1)$ of size $\varphi(d)$ 
        (\S\ref{sec:SubfanDescrip}), 
    \item[--] $\binom{\tau}{\sigma}$ is 
    the number of unordered $r$-subsets 
    $E \subset \tau(1)$ such that $\pi E = \sigma(1)$ 
        where $\pi \colon \ccl_E \to \ccl$ denotes 
        restriction to $\cl_E^\Gamma$ and $\cl_E=\ccl_E^\vee$. 
\end{itemize} 

Our proofs do not make use of 
    any known cases of Malle's conjecture 
    nor class field theory, 
    though we do make use of known zero-free regions in 
    the critical strip 
    for $L$-functions of Hecke characters 
    on cyclotomic fields. 
We expect that the same asymptotic formula holds 
    when counting all $C_n$-polynomials 
    (see Remark~\ref{rmk:GeneralPolynms}).


\subsubsection{Malle's conjecture}\label{sec:malle} 

Our theorem on abelian polynomials 
    may be compared with 
    another asymptotic formula for $N(C_n,H)$ 
    implied by Malle's conjecture. 
Let $G$ be a regular permutation group of order $n\geq 4$, 
and let $D_G(s)$ be the ``discriminant series'' 
\begin{equation} 
    D_G(s) = 
    \sum_{\text{$K/\QQ$ $G$-field}} 
        |d_K|^{-s}.
\end{equation} 
Malle's conjecture predicts there is a positive constant 
    $a<\tfrac12$ 
such that the number $F(G,H)$ of $G$-fields $K/\QQ$ 
    with $|d_K| < X$ satisfies 
    $X^{a-\varepsilon} \ll_\varepsilon F(G,H)
    \ll_\varepsilon X^{a+\varepsilon}$ 
    for any $\varepsilon >0$ as $X$ goes to infinity; 
    in particular, $D_G(1/2)$ is finite. 
Let $f_x\in\ZZ[t]$ denote the characteristic polynomial 
of an element $x \in K$. 
As $H$ goes to infinity, we have that 
\begin{equation} 
    N(G,H)
    =
    \sum_{\text{$K/\QQ$ $G$-field}} 
        \tfrac1n
        \#\{x \in O_K : K=\QQ(x), H(f_x)<H\}
    =
    c_G
    D_G(1/2)
    H^n
        (1+o(1))
\end{equation} 
where $c_G$ is the scaling ratio 
    between root height and Minkowski norm. 
It would be interesting to compare 
    the formulas for $D_{C_p}(1/2)$ 
    when $p$ is prime 
    provided by our theorem 
    and \cite{zbMATH01765119}.\\ \par





We are very grateful to 
Anshul Adve, 
Tim Browning, 
    Vesselin Dimitrov, 
    Peter Sarnak, 
    Victor Wang, and Shou-Wu Zhang for 
    helpful discussions. 
In particular, we would like to 
    thank Manjul Bhargava for many inspiring discussions 
    and also for showing us the argument in \S\ref{sec:malle}. 


\subsubsection*{Notation} 

The places of a number field $E$ are denoted $\Places{E}$ 
and the completion of $E$ at a place $w$ is denoted $E_w$. 
For a finite place $w$ of $E$ we set 
    $\norm{ x }_{w}
    \coloneqq  
    q_w^{-\mathrm{ord}_{w}(x)}$ 
where $\mathrm{ord}_w \colon E^\times \twoheadrightarrow \ZZ$ 
is the normalized additive valuation associated to $w$ 
and $q_w$ is the size of the residue field of 
the ring of integers of $E_w$. 
The norm $\norm{x}_w$ is also the multiplier for 
    the additive Haar measure on $E_w$ under $y \mapsto yx$. 
For an infinite place $w$ 
corresponding to an embedding 
$i \colon E \hookrightarrow \CC$, 
let 
    $\norm{ x }_{w}
    \coloneqq  
    |i x|_\CC^{d_{w}}$ 
where $d_{w} = 1$ (resp. $d_w=2$) if $w$ is real 
(resp. if $w$ is complex), 
and we set 
$\mathrm{ord}_w(x) = -\frac{\log \norm{x}_w}{\log q_w}$ 
where $q_w \coloneqq e^{d_w}$. 
With these norms the product formula takes the form 
$\prod_{w\in \Places{E}} \norm{x}_{w} = 1$ for any nonzero $x$. 
If $E$ is Galois over $\QQ$ and $v$ is a place of $\QQ$ 
we let $e_v$ (resp. $f_v$) 
denote the ramification (resp. inertia) index of 
any place of $E$ lying over $v$; 
by convention, $e_v=1$ if $v = \infty$. 
We set $e(x) = e^{2\pi ix}$. 
For a fan $\Sigma$ 
    we let $|\Sigma|$ denote the union of its cones 
    in the ambient space $\ccs$. 
A subset of $\ccs$ is called \defn{$\Sigma$-convex} 
    if it is contained in a cone of $\Sigma$. 



\section{Integral points on toric varieties} 

\subsection{Preliminaries on algebraic tori}\label{sec:preliminaries} 

Let $T$ be an $n$-dimensional algebraic torus defined over $\QQ$. 
Let $E/\QQ$ be the minimal Galois extension 
which splits $T$ and let $\Gamma = G(E/\QQ)$ be its Galois group. 
For a field $F$ containing $\QQ$ 
write $T_F = T\otimes F$ for the base-change of $T$ to $F$. 
Let $X^\ast T = \Hom(T,\GG_{m,\QQ})$ 
be the set of rational characters of $T$ 
and let $X_\ast T = \Hom(\GG_{m,\QQ},T)$ 
be the set of rational cocharacters of $T$; 
we will also need $X^\ast T_E = \Hom(T,\GG_{m,E})$ 
and $X_\ast T_E = \Hom(\GG_{m,E},T_E)$. 
These sets are finite free abelian groups 
and there is a canonical pairing 
$X^\ast T \times X_\ast T \to \ZZ$ 
given by $(m,n) \mapsto \deg m(n)$. 
While basic examples show this pairing is 
generally {not} perfect, 
it does identify $X_\ast T$ with 
a finite index sublattice of the dual lattice to $X^\ast T$ 
so this pairing induces a canonical isomorphism 
on the level of real vector spaces: 
$$\RR \otimes X_\ast T \cong \RR \otimes X^\ast T^\vee.$$ 
The character lattice $X^\ast T_E$ has 
    a natural Galois action by post-composition 
    and we equip $X_\ast T_E$ with 
    the corresponding dual representation. 
It is however more convenient to use additive notation 
    for these lattices. 
We denote these with $M$ and $N$: 
\begin{equation} 
    \cl_E \cong X^\ast T_E : u \leftrightarrow \chi^u, \quad
    \cl = (\cl_E)^\Gamma = X^\ast T, \quad
    \ccl_E = X_\ast T_E, \quad
    \ccl = \ccl_E^\Gamma = X_\ast T. 
\end{equation} 
The associated real vector spaces will be denoted 
\begin{equation} 
    \cs_E = \RR \otimes \cl_E, \quad
    \cs = (\cs_E)^\Gamma, \quad
    \ccs_E = \RR \otimes \ccl_E = \mathrm{LieAlg}(T_E), \quad
    \ccs = \ccs_E^\Gamma 
\end{equation} 


\begin{definition}
Let $w$ be a place of $E$. 
For any $t \in T(E_w)$ 
    the function 
    $\chi \mapsto \mathrm{ord}_w(\chi(t))$ 
    on characters $\chi \in X^\ast T_{E_w}$ 
    determines an element of 
    $\RR \otimes X_\ast T_{E_w}$. 
Let $$\lt_w(t) \in \ccs_E$$ 
    be the corresponding cocharacter 
    under the canonical isomorphism 
    $\RR \otimes X_\ast T_{E_w} \cong \RR \otimes X_\ast T_{E}$ 
    induced by base change of the split torus $T_E$ 
    along the inclusion $E \to E_w$. 
\end{definition}

Let $w$ be a place of $E$ and let $v$ be its restriction to $\QQ$. 
Let 
\begin{equation} 
\ccl_w = \ccl_{E}^{\Gamma_w},\quad
\ccs_w = \ccs_{E}^{\Gamma_w}
\end{equation} 
where $\Gamma_w$ is the decomposition group at $w$. 
Let $\MaximalCompact{v} \subset T(\QQ_v)$ 
be the maximal compact subgroup. 

\begin{proposition}[{\cite[(1.3), p.~449]{MR291099}}]\label{prop:basicfacts}\, 
Let $w$ be a place of $E$ and let $v$ be its restriction to $\QQ$. 
There is an exact sequence 
\begin{equation} 
    1 \longrightarrow \MaximalCompact{v} 
    \longrightarrow T(\QQ_v) 
    \xlongrightarrow{\lt_w}
    \ccs_w.
\end{equation} 
If $w$ is infinite then $\lt_w$ is surjective, 
    if $w$ is finite and unramified 
    then the image of $\lt_w$ is the lattice 
    $\ccl_w$, 
    and if $v$ is ramified then 
    the image of $\lt_w$ is 
    a finite index subgroup of $\ccl_w$. 
\end{proposition} 




We recall the definition of the toric height 
    on $Y$ following \cite{IMRN}. 

\begin{definition}
    Let $\varphi \colon \ccs_E \to \CC$ be a function. 
    For $t = (t_w)_w \in T(\AA_E)$ 
    let 
\begin{equation} 
    H(t,\varphi) 
    =\prod_{w \in \Places{E}}
    \left(q_w^{-\varphi(\lt_w(t_w))}\right)^{\frac{1}{[E:\QQ]}}
\end{equation} 
where $\varphi(\lt_w(t_w))$ is evaluated using 
    the canonical isomorphism $X_\ast T_E \cong X_\ast T_{E_w}$. 
\end{definition}

The reader may easily verify the following basic properties: 
\begin{enumerate} 
\item $H(-,\varphi)$ is invariant under 
    the action of the maximal compact subgroup of $T(\AA_E)$. 
\item $H(t,\varphi) = 1$ if $t \in T(\QQ)$ and 
    $\varphi$ is linear. 
\item If $t \in T(\QQ_v)$ where $w$ lies over $v$ 
    and $\varphi$ is $\Gamma$-invariant, then 
    $\varphi(\lt_w(t))$ is independent of $w$, 
        even though $\lt_w(t)$ depends on 
        the choice of $w$ over $v$ 
        in general. 
\item If $t = (t_w)_w \in T(\AA)$ 
    embedded diagonally in $T(\AA_E)$
    and $\varphi$ is $\Gamma$-invariant, then 
\begin{equation}\label{eqn:heightforrationaladeles} 
        H(t,\varphi)
        = \prod_{v \in \Places{\QQ}} 
        q_v^{-\frac{1}{e_v}\varphi(\lt_w(t_w))}.
\end{equation} 
\end{enumerate} 

\begin{remark}
The toric height defined in \cite{IMRN} 
    differs slightly from ours. 
We have introduced normalization factors 
    that are needed for (2) to hold, 
    which is itself needed in order for 
    the height zeta function to descend 
    to a function on $\Pic Y$ 
    and not just $\Pic^T Y$. 
\end{remark}

\subsection{Preliminaries on toric varieties} 

Let $\Sigma$ denote the fan of $Y_E$ in $\ccs_E$. 
We do not assume $Y$ is smooth and 
    permit cyclic quotient singularities. 
This means the fan $\Sigma$ is simplicial: 
    all the cones are simplices. 
Let $\Div_T Y$ denote the group of $T$-Cartier divisors on $Y$. 
See \cite{MR1234037} for more details. 
We let $\Sigma(k)$ denote 
the set of $k$-dimensional cones in $\Sigma$ 
and identify one-dimensional cones 
with their minimal lattice generators. 
The next proposition is well-known 
    \cite[\S3.4]{MR1234037}.

\begin{proposition}\label{prop:supportFunctions} 
There is a commutative diagram with exact rows: 
\begin{equation} 
\begin{tikzcd}
0 \ar[r] &\cl_E \ar[-,double line with arrow={-,-}]{d} 
            \arrow{r}{\mathrm{div}} 
    & \Div_T Y_{E} \ar[r] \ar[hookrightarrow]{d} 
    &\Pic (Y_{E}) \ar[r] \ar[hookrightarrow]{d} & 0\\
0 \ar[r] &\cl_E \ar[r] 
    & \ZZ^{\Sigma(1)} \cong\oplus_{e\in {\Sigma}(1)}\ZZ D_e \ar[r] 
        &A_{n-1}(Y_{E}) \ar[r] &  0.
\end{tikzcd}
\end{equation} 
The two vertical inclusions 
    become isomorphisms 
    after tensoring with $\QQ$. 
The canonical $\Gamma$-action on $\ccl_E$ 
    preserves the fan of $Y_E$. 
The group $\Div_T Y_{E}$ 
    is naturally isomorphic with 
    the group of $\Sigma$-linear\footnote{
    i.e.~continuous functions 
    $\varphi \colon \ccs_E \to \RR$ 
    such that the restriction of $\varphi$ 
    to each cone $\sigma \in \Sigma$ is equal to 
    the restriction of some linear form $u(\sigma) \in \cl_E$.} 
            functions 
            $\varphi \colon \ccs_E \to \RR$ 
            which are $\ZZ$-valued on $\ccl_E$. 
The $\Sigma$-linear function $\varphi$ 
    corresponding to the line bundle 
    $\mathcal O(\sum_{e\in{\Sigma}(1)}s_eD_e)$ 
    satisfies $\varphi(e) = -s_e$ for each $e \in {\Sigma}(1)$. 
A $T$-Cartier divisor on $Y_E$ 
descends to a $T$-Cartier divisor on $Y$ if and only if 
    its associated $\Sigma$-linear function 
    is Galois invariant. 

\end{proposition} 

To each $s \in \CC^{\Sigma(1)}$ we associate 
    the unique $\Sigma$-linear function $\varphi$ 
    satisfying $$\varphi(e) = -s_e$$ for all $e \in \Sigma(1)$,  
and set 
\begin{equation} 
    H(t,s) \coloneqq H(t,\varphi). 
\end{equation} 

\begin{remark}
Our sign convention on $\varphi$ follows 
    \cite{MR1234037} 
    and not \cite{IMRN}. 
In particular, the linear function 
    $\varphi_{\mathrm{div}(m)}\colon \ccs_E \to \RR$ 
    is equal to $-m$ and not $m$. 
\end{remark}

\subsubsection{Integral points} 

Let $D$ be an effective $T$-stable Weil divisor of $Y$ 
(``boundary divisor'') 
    and let $U$ be an integral model for $Y-D$. 
An adelic point $t = (t_v)_v$ of $T$ 
    is \defn{$D$-integral} 
    if $t_v \in U(\ZZ_v)$ for all finite places $v$. 
The indicator function on $D$-integral points in $T(\AA)$ 
    is a locally constant factorizable function 
    $$\charfn^D = \prod_{v\in \Places{\QQ}} \charfn^{D}_v$$ 
    where $\charfn^D_\infty$ 
    is identically one on $T(\RR)$. 
For finite places $v$ 
    let $\IntegralCompact{v} \subset T(\QQ_v)$ 
    be the (open compact) subgroup 
    under which $\charfn^{D}_v$ is invariant, 
    and let $\IntegralCompact{\infty}$ 
    denote the maximal compact subgroup of $T(\RR)$. 
Set $\IntegralCompact{} = \prod_v \IntegralCompact{v}$. 

\begin{warning}
The toric height $t\mapsto H(t,s)$ is 
    invariant under translation in $x$ 
    by elements of the maximal compact subgroup 
    $\MaximalCompact{}$ of $T(\AA)$, 
    but to count $D$-integral points we must use 
    $H(t,s)\charfn^D(t)$ 
    instead of $H(t,s)$ 
and the indicator function $1_D$ is 
    {generally not invariant under 
    the maximal compact subgroup $\MaximalCompact{}$}. 
(E.g. the subgroup $\IntegralCompact{} \subset K$ has index $6$ 
    for the Diophantine problem in \cite{CUBIC}, 
    see \cite[Lemma~2]{CUBIC}.) 
This means that in the Poisson-style approach 
    to counting integral points, 
    the integral \eqref{eqn:poissonae} 
    over the automorphic spectrum of the torus $T$ 
    is complicated by characters 
    with non-maximal level. 
\end{warning}

\begin{remark}[Cartier divisors]\label{rmk:CartierIntegrality} 
When the boundary divisor $D$ is 
    Cartier there is a canonical integral model 
    for $Y-D$ defined as follows. 
Let $\varphi$ be the $\Sigma$-linear function 
    corresponding to the $T$-Cartier divisor $D$. 
On each cone $\sigma \in \Sigma$ 
    the restriction of $\varphi$ to $\sigma$ 
    is equal to the restriction of some linear form 
    $u(\sigma) \in \cl_E$. 
Let $U_\sigma = \Spec \QQ[\sigma^\vee \cap \cl_E] \subset Y_E$ 
    denote the affine open subvariety 
    corresponding to a cone $\sigma$. 
The ideal $I \subset \mathcal O(U_\sigma)$ 
    corresponding to $D \cap U_\sigma$ is generated 
    by the character corresponding to $u(\sigma)$, i.e. 
\begin{equation} 
    I=\QQ[\sigma^\vee \cap \cl_E]\chi^{u(\sigma)}.
\end{equation} 
From this description it is clear that 
$U_\sigma-D \cap U_\sigma$ has the integral model 
\begin{equation} 
    \Spec \ZZ[\sigma^\vee \cap \cl_E][\chi^{-u(\sigma)}].
\end{equation} 
The affine opens $\{U_\sigma\}$ form an open cover of $Y_E$ 
and glue to an integral model $U$ for $Y - D$. 
We have the following integrality criterion: 
\begin{equation}\label{eqn:CartierIntegralityCondition} 
t \in T(\QQ_v) \cap U(\ZZ_v)
\iff
    \chi^{u(\sigma)}(t)  \in \ZZ_v^\times 
    \quad\text{whenever $t \in U_\sigma$.} 
\end{equation} 
Nonetheless, in applications 
    it is important to consider integral points 
    with respect to non-Cartier divisors, 
    e.g.~the boundary divisor in \cite{CUBIC} is not Cartier 
    and its notion of integrality corresponds to 
    the natural condition of being monic 
    for integral cubic abelian polynomials. 
\end{remark} 



\subsubsection{The height zeta function} 

Let $\Phi=\prod_{v \in \Places{\QQ}}\Phi_v \colon T(\AA) \to \CC$ be 
    a locally constant factorizable function 
    for which $\Phi_\infty$ 
    is identically one on $T(\RR)$. 

\begin{definition}
Define the multiple Dirichlet series 
\begin{equation} 
    Z(s,\Phi) = 
    \sum_{t \in T(\QQ)}
    \frac{\Phi(t)}{H(t,s)}.
\end{equation} 
The \defn{height zeta function 
(for $D$-integral points in $T$)} is 
\begin{equation} 
    Z(s) = Z(s,\charfn_D)=
    \sum_{\substack{t \in T(\QQ)\\\text{$D$-integral}}}
    \frac{1}{H(t,s)}.
\end{equation} 
This converges to a holomorphic function 
    of several complex variables on 
    the region of $\Pic(Y) \otimes \CC$. 
    whose real part is 
    sufficiently positive with respect to the effective cone. 
\end{definition}



\subsubsection{Restricted fans}\label{sec:restrictedFan} 

We fix once and for all a choice of place $w$ of $E$ 
    lying over each place $v$ of $\QQ$. 
If $\Gamma_w \subset \Gamma$ is the decomposition group 
    at the chosen place $w$ of $E$ over $v$ 
    then we write $\Gamma_v$ for $\Gamma_w$. 

\begin{definition}[restricted fan]
For any subgroup $\Gamma_v$ of $\Gamma$ 
    let $\Sigma_{\Gamma_v}$ denote 
    the fan in $\ccs_E^{\Gamma_v}$ 
    consisting of 
    the cones $\sigma \cap \ccs_E^{\Gamma_v}$ 
    for each cone $\sigma$ in $\Sigma$. 
\end{definition}

The next lemma describes the restricted fan 
    and its generators 
    for a simplicial fan $\Sigma$ 
    which is not necessarily regular; 
cf.~\cite[Prop.~1.3.2]{IMRN} for the regular case. 

\begin{lemma}\label{lemma:subfan}\, 
    \begin{enumerate}
        \item For each $\sigma \in \Sigma_\Gamma$, 
    there is a unique cone $\tau_\sigma \in \Sigma$ 
    which is minimal among cones in $\Sigma$ containing $\sigma$. 
The cone $\tau_\sigma$ 
    will be fixed by $\Gamma$ 
    (i.e.~$\gamma(\tau_\sigma) = \tau_\sigma$ 
    for all $\gamma\in\Gamma$) 
            and satisfies $\tau_\sigma \cap \ccs= \sigma$. 
The map $\sigma \to \tau_\sigma$ determines a bijection 
    between $\Sigma_\Gamma$ 
    and the subset of cones of $\Sigma$ which are fixed by $\Gamma$. 
\item 
Let $e_0\in \Sigma_{\Gamma}(1)$ be a lattice generator 
    for a ray $\sigma$ in $\Sigma_\Gamma$. 
Then $\Gamma$ acts transitively on $\tau_\sigma(1)$ and 
\begin{equation}\label{eqn:rayGeneratorsDescription} 
    e_0=\frac{1}{k}\sum_{e \in \tau_\sigma(1)}e
\end{equation} 
    for a positive integer determined by $\sigma$.
If $\tau_\sigma$ is regular, then $k=1$. 
The bijection $\sigma \leftrightarrow \tau_\sigma$ 
    determines a bijection between 
    $\Sigma_\Gamma(1)$ 
    and the set of $\Gamma$-orbits in $\Sigma(1)$ 
    equal to $\tau(1)$ 
    for some cone $\tau \in \Sigma$. 
\item 
The fan $\Sigma_\Gamma$ is simplicial. 
If $\tau \in \Sigma$ is regular 
    then $\tau \cap \ccs \in \Sigma_\Gamma$ 
    is regular. 
    \end{enumerate}
\end{lemma} 

\begin{proof} 
Since $\Sigma$ is closed under intersection, 
    $\tau_\sigma$ is simply the intersection of all cones 
    in $\Sigma$ containing $\sigma$; 
if $\tau_\sigma$ were not fixed by $\Gamma$, 
    then it would not be minimal. 
By definition of $\Sigma_\Gamma$ there is 
    a cone $\tau \in \Sigma$ for which $\tau \cap \ccs = \sigma$. 
Since $\tau$ contains $\tau_\sigma$ 
    we see that $\tau_\sigma \cap \ccs = \sigma$. 
To see this is a bijection, 
    we will show that any $\Gamma$-fixed cone $\tau$ of $\Sigma$ 
    satisfies $\tau = \tau_\sigma$ 
    with $\sigma=\tau \cap \ccs$. 
Clearly $\tau$ contains $\tau_\sigma$, 
    so it suffices to show that if $\tau' \subset \tau$ is 
    another $\Gamma$-fixed cone 
    satisfying $\tau' \cap\ccs = \sigma$, 
    then $\tau' = \tau$. 
To the contrary, if $\tau' \neq \tau$ then there is some 
    $e \in \tau(1) \backslash \tau'(1)$.  
Since $\tau'$ is $\Gamma$-fixed, 
    $\tau'(1)$ is a proper $\Gamma$-stable subset of $\tau(1)$, 
    so the $\Gamma$-orbit of $e$ has 
    no intersection with $\tau'(1)$. 
This shows that 
    $\sum_\gamma \gamma (e) \not \in \tau'$, 
for otherwise we would have a nontrivial relation 
    among the linearly independent generators of $\tau$; 
thus $\tau' \cap \ccs \neq \sigma$ 
    since it fails to contain the point 
    $\sum_\gamma \gamma (e) \in \sigma$. 

Let $e_0\in \Sigma_{\Gamma}(1)$ be a lattice generator 
    for a ray $\sigma$ in $\Sigma_\Gamma$. 
Then $\sum_{e \in \tau_\sigma(1)}e \in \sigma \cap \ccl$ 
    so it is equal to $ke_0$ for some integer $k \in \ZZ_{\geq 1}$. 
Since $\tau_\sigma$ is minimal the group $\Gamma$ 
    must act transitively on $\tau_\sigma(1)$. 
If $\tau_\sigma$ is regular, then 
    $\frac 1k\sum_{e \in \tau_\sigma(1)}e$ can only be 
    a lattice point if $k = 1$. 

Let $\tau \in \Sigma$ be $\Gamma$-fixed 
and let $\sigma=\tau \cap \ccs$. 
Each generator of $\sigma$ 
    corresponds via \eqref{eqn:rayGeneratorsDescription} 
    to a $\Gamma$-orbit in $\tau(1)$. 
Since $\tau(1)$ extends to a $\QQ$-basis of $\QQ \ccl_E$, 
    the generators of $\sigma$ 
    given by \eqref{eqn:rayGeneratorsDescription} 
    extend to a $\QQ$-basis of $\QQ \ccl$. 
For the implication that $\tau$ regular $\implies$ 
    $\tau \cap \ccs$ regular 
    see \cite[Prop.~1.3.2]{IMRN}. 
\end{proof} 

\begin{remark}
Example~\ref{example:antenna} 
    gives a regular cone $\sigma \in \Sigma_\Gamma$ 
    for which $\tau_\sigma \in \Sigma$ is not regular.  
\end{remark}

\subsection{Automorphic characters on the torus} 

Now we describe automorphic characters of the torus. 
We fix an arbitrary open compact subgroup $K_{\mathrm{fin}}$ 
    of $T(\AA_{\mathrm{fin}})$ and set 
    $K = K_{\mathrm{fin}} \times \MaximalCompact{\infty}$. 

\subsubsection{The norm-one subgroup of $T(\AA)$}\label{sec:normOne} 

We recall a logarithm map from the construction 
    of Langlands--Eisenstein series, cf.~\cite[\S6]{shahidi}: 
the homomorphism $\Lt \colon T(\AA) \to \ccs$ 
    defined by the formula 
    $$\Lt(t)(m) = \log |\chi^m(t)|_\AA^{-1}
    \qquad(m \in \cl).
    $$
The \defn{norm-one subgroup} $T(\AA)^1$ is the kernel of $\Lt$. 
A construction from \cite{ono} shows that $\Lt$ admits a splitting. 

\begin{lemma}[{\cite[Lemma~3.8.1, p.~132]{ono}}]\label{lemma:ctssplittings}
The logarithm map $\Lt$ admits a continuous splitting 
    $s \colon \ccs \to T(\AA)$ which takes values in $T(\RR)$. 
In particular, $\Lt$ is surjective. 
\end{lemma}

Fix such a splitting $s \colon \ccs \to T(\AA)$. 
This determines isomorphisms in the usual manner 
\begin{equation}\label{eqn:trivialsplitting} 
T(\AA) \cong_s \ccs  \times T(\AA)^1 \colon 
    t\mapsto (\Lt(t),ts(\Lt(t))^{-1}),
    \qquad
T(\QQ)\backslash T(\AA)/K
    \cong_s
    \ccs
    \times 
    T(\QQ)\backslash T(\AA)^1/K.
\end{equation} 
In particular, any character on 
$T(\QQ)\backslash T(\AA)^1/K$ extends to a character on 
$T(\QQ)\backslash T(\AA)/K$. 

\begin{example}[restriction tori]\label{example:Restriction1} 
Take $T = R^E_\QQ \GG_m$ where $E$ is 
    a finite Galois extension of $\QQ$. 
There are $n$ characters $\chi_1,\ldots,\chi_n$ 
    generating $X^\ast T_E$ 
    which realize the $n$ embeddings 
    $T(\QQ)=E^\times \to \CC$ on the $\QQ$-points of $T$. 
The rational character lattice $\cl \cong X^\ast T$ 
    is generated by the norm character 
    $m_\mathrm{N} \leftrightarrow \mathrm{N}=\prod \chi_i$ 
    where 
    $\mathrm{N} \colon T(\AA)=\AA_E^\times 
    \to \GG_m(\AA)=\AA^\times$. 
If $(e_i)_i \subset \ccl_E$ is the dual basis 
    to $\chi_1,\ldots,\chi_n$ and 
    $e_\infty \coloneqq-\sum e_i$, then 
\begin{equation}\label{eqn:LtRestrTori} 
    \Lt(x) = 
    \frac{1}{n}\log |\mathrm{N}(x)|_{\AA}e_\infty.
\end{equation} 
A natural choice of splitting is given by 
\begin{align} 
    s \colon \ccs = \RR e_\infty &\to 
    T(\AA) = T(\RR) \times T(\AA_{\mathrm{fin}})\\
    e_\infty &\mapsto s(e_\infty) = (q_\infty1_\infty\,;\,1_{\mathrm{fin}})
\end{align} 
    where $1_\infty$ (resp.~$1_{\mathrm{fin}}$) is 
    the identity of the group 
    $T(\RR) = E_\RR^\times$ (resp.~$T(\AA_{\mathrm{fin}})$) and 
    $q_\infty=e^{d_\infty} = e$ (resp.~$e^2$) 
    if $E$ is totally real (resp.~not totally real). 
\end{example} 


Next we turn to characters of $T(\QQ)\backslash T(\AA)^1/K$. 
We will use non-standard but equivalent definitions 
    for the regulator and class groups 
    that are more convenient for the Poisson formula. 

\begin{definition}
Let $R$ denote the identity component 
    of $T(\QQ)\backslash T(\AA)^1/K$ 
    and let $C$ denote 
    the quotient of $T(\QQ)\backslash T(\AA)^1/K$ by $R$. 
We call $R$ the \defn{(level $K$) regulator group} 
    and $C$ the \defn{(level $K$) class group}. 
\end{definition}

It is a fundamental fact 
    \cite[Theorem 3.1.1]{ono} 
    that $R$ is compact and $C$ is finite. 
The terminology reflects the fact that 
    if $T = R^E_\QQ \GG_m$ and $K$ is maximal, then $R$ is 
    the quotient in the trace-zero hyperplane of $E_\RR$ 
    by the lattice of logarithms of units 
    and $C$ is the ideal class group of $E$. 


    %
\begin{lemma}\label{lemma:normonesubgroup} 
Assume that $K_\infty = \MaximalCompact{\infty}$. 
The class group $C$ is isomorphic to 
    $T(\QQ)\backslash T(\AA)/(T(\RR)K)$ 
    (so our definition agrees with the standard definition 
    of the class group of an algebraic group). 
If $\psi \in (T(\QQ)\backslash T(\AA)^1/K)^\vee$ 
    factors through $C$, 
    then the local archimedean component $\psi_\infty$ of $\psi$, 
    given by 
    $T(\RR) \to T(\QQ)\backslash T(\AA) 
    \to_s T(\QQ)\backslash T(\AA)^1/K \to C \to \CC^\times$, 
     is trivial. 
\end{lemma} 


\begin{proof} 
Let $C' = T(\QQ)\backslash T(\AA)/(T(\RR)K)$. 
The kernel of $T(\QQ)\backslash T(\AA)/K \to C'$ 
    is $T(\QQ)\backslash (T(\QQ)T(\RR)K)/K$. 
Observe that $T(\RR)K/K = T(\RR)/K_\infty$ is connected 
    (Proposition~\ref{prop:basicfacts}), 
    and since the natural map 
    $$T(\RR)K/K \to T(\QQ)\backslash (T(\QQ)T(\RR)K)/K$$ 
    is clearly surjective, 
    we see that $T(\QQ)\backslash (T(\QQ)T(\RR)K)/K$ 
    is also connected. 
It is also closed and open since it is the preimage of 
    the identity in $C'$, so the above kernel 
    is the identity component of $T(\QQ)\backslash T(\AA)/K$. 
The isomorphism \eqref{eqn:trivialsplitting} 
    identifies the identity component of 
    $T(\QQ)\backslash T(\AA)/K$ with $\ccs \times R$ 
    so we have the exact sequence 
\begin{equation}
    1 \longrightarrow \ccs \times R
    \longrightarrow 
        \ccs \times T(\QQ)\backslash T(\AA)^1/K
    \longrightarrow_s C'
    \longrightarrow 1.
\end{equation} 
Restricting this short exact sequence to 
    the subgroup $\{0\} \times T(\QQ)\backslash T(\AA)^1/K$ 
    shows that $C\cong_s C'$. 

Let $w$ be an infinite place of $E$. 
The local component $\psi_\infty$ 
    factors through the map 
    $$T(\RR) \to T(\QQ)\backslash T(\AA) 
    \to_s T(\QQ)\backslash T(\AA)^1/K \to C$$ 
    which itself factors through 
    $T(\RR) \xrightarrow{\lt_w} \ccs_w$ 
    (Proposition~\ref{prop:basicfacts}), 
    but every homomorphism $\ccs_w \to C$ 
    is constant so $\psi_\infty$ is trivial. 
%
\end{proof} 

\subsubsection{The regulator group of level $K$}\label{sec:regulatorGroup} 

Let $r_\infty$ be the rank of $\ccl_w$ 
for any infinite place $w$ of $E$ 
and let $r$ be the rank of $\ccl$. 

\begin{definition}
The subspace $\ccs \subset \ccs_w$ 
    of $\Gamma$-invariants has a canonical subspace complement 
    $$\ccs_0 = \bigoplus_{\psi \neq 1}((\ccs_E)_\psi)^{\Gamma_w}$$ 
    coming from the $\Gamma$-isotypic decomposition of $\ccs_E$. 
    We call $\ccs_0$ the 
    \defn{trace-zero subspace} of $\ccs_w$. 
\end{definition}


The abelian subgroup 
\begin{equation} 
    T(\QQ) \cap 
    \left(T(\RR) \times 
    \prod_{v\text{ finite}} K_v\right) 
\end{equation} 
of $T(\AA)$ is finitely generated with rank 
    $r_\infty-r$ \cite[Thm.~4, p.~285]{MR114817} 
and canonically splits as 
\begin{equation} 
\ToricUnits \times \ToricROU = 
    T(\QQ) \cap 
    \left(T(\RR) \times 
    \prod_{v\text{ finite}} K_v\right) 
\end{equation} 
where $\ToricROU = T(\QQ) \cap K$ is its torsion subgroup. 
It is well-known that the natural map 
    $$\ToricUnits \to T(\RR) \xrightarrow{\lt{w}} \ccs_w$$ 
    embeds $\ToricUnits$ as a (full-rank) lattice of 
    the trace-zero subspace $\ccs_0$. 

\begin{remark}
If $D$ is Cartier and the canonical integral model 
    is used to define $D$-integrality 
    (cf.~Remark~\ref{rmk:CartierIntegrality}), 
    then $\mu$ is independent of $D$ 
    and equal to $T(\QQ) \cap \MaximalCompact{}$. 
\end{remark}





\subsection{Fourier transform of the height zeta function} 

If $t \in T(\QQ_v)$ and $s \in (\CC^{\Sigma(1)})^\Gamma$, 
    then we may write 
    $$H_v(t,s) 
    = q_v^{-\frac{1}{e_v}\varphi(\lt_w(t))}$$ 
    where $w$ is any place of $E$ lying over $v$ 
    (cf.~\eqref{eqn:heightforrationaladeles}); 
although $\lt_w(t)$ may depend on the choice of $w$ over $v$, 
    the quantity $H_v(t,s)$ only depends on $v$. 

\subsubsection{Haar measures} 

For each place $v$ of $\QQ$ let a place $w$ of $E$ lying over $v$ 
    be chosen and fixed. 
For each finite place $v$ 
    let $d\mu_v$ be the Haar measure on $T(\QQ_v)$ 
    for which the maximal compact subgroup 
    has measure one. 
At the infinite place choose 
    the Haar measure $d\mu_\infty$ on $T(\RR)$ 
    for which $\ccl_w \subset \ccs_w$ is a unimodular lattice 
    for the pushforward of $d\mu_\infty$ 
    along $\lt_w \colon T(\RR) \to \ccs_w$. 
The product measure $d\mu = \prod_v d\mu_v$ 
    is a Haar measure on the restricted product $T(\AA)$ 
    and the Fourier transform of 
    any factorizable integrable function 
    $f = \otimes_v f_v \in L^1(T(\AA))$ is 
    defined by 
\begin{equation} 
    \widehat{f}(\chi) = 
    \int_{T(\AA)}
    f(x)\chi(x)^{-1}\,d\mu(x)
    =
    \prod_v
    \int_{T(\QQ_v)}
    f_v(x)\chi_v(x)^{-1}\,d\mu_v(x).
\end{equation} 

\subsubsection{Local Fourier transforms: nonarchimedean case}\label{sec:fourierna} 
Let $v$ be a finite place of $\QQ$. 
Write $$H_v(t,s,D)=H_v(t,s)\charfn^D_v(t).$$ 
Let $\IntegralCompact{v}$ denote the subgroup of 
    $\MaximalCompact{v}$ 
    fixing the indicator function $\charfn^D_v$. 

\begin{lemma}\label{lemma:finiteheightfouriertrfm}
Let $s \in (\CC^{\Sigma(1)})^{\Gamma}$
    and let $\varphi \colon \ccs \to \CC$ 
    be the corresponding $\Sigma$-linear function 
    (Proposition~\ref{prop:supportFunctions}). 
Let $w$ be any fixed place of $E$ lying over $v$ 
    and set $\Sigma_v = \Sigma_{\Gamma_w}$. 
For any character $\chi\in T(\QQ_v)^\vee$ 
    the integral 
    $$
    \widehat{H_v}(\chi,-s,D)
    =
    \int_{T(\QQ_v)}
    H_v(x,-s,D)\chi(x)^{-1}\, d\mu_v (x)
    $$ 
    converges absolutely to a holomorphic function of $s$ 
    in the region 
    $$\{s \in (\CC^{\Sigma(1)})^\Gamma : \mathrm{Re}(s_e)>0 
    \text{ for every $e \in \Sigma(1)$}\}.$$ 
This holomorphic function admits 
    a meromorphic continuation to a rational function 
    in the variables 
    $\{q_v^{\varphi(e)}:e \in \Sigma_v(1)\}$. 
Furthermore $\widehat{H_v}(\chi,-s,D)$ 
    is identically zero 
    if $\chi$ is nontrivial on $\IntegralCompact{v}$, 
    and otherwise is given by 
\begin{equation} 
    \widehat{H_v}(\chi,-s,D)
    =
    \sum_{n \in T(\QQ_v)/\IntegralCompact{v}} 
    \chi(n)^{-1} \charfn^D_v(n)
    q_v^{\frac{1}{e_v}\varphi(\lt_w(n))} . 
\end{equation} 
\end{lemma}

\begin{remark} 
The local Fourier transforms 
    must be computed \emph{before} restricting to 
    the line in $(\Pic^T Y) \otimes \CC$ 
    spanned by the line bundle determining the height of interest 
    since $x \mapsto H_v(x,-s,D)$ will generally not 
    be integrable once some of 
    the coordinates of $s$ 
    vanish, 
    no matter how large and positive $\mathrm{Re}(s_e)$ is 
    for the other coordinates. 
\end{remark} 

\begin{remark}
    One aspect where the Poisson-style 
    analysis of integral points differs from 
    the analysis of rational points \cite{IMRN} 
    is that $\charfn^D_v$ and thus 
    ${H_v}(\chi,-s,D)$  
    may have compact support, 
    cf.~\cite[Prop.~4]{CUBIC}. 
\end{remark}

\begin{proof} 
If $\chi$ is nontrivial on $\IntegralCompact{v}$  
    then the Fourier transform of 
    $H_v(x,s)^{-1}\charfn^D_v(x)$ 
    vanishes by Schur's lemma. 
Now suppose $\chi|_{\IntegralCompact{v}} = 1$. 
Then the integrand is $\IntegralCompact{v}$-invariant and 
\begin{equation}\label{eqn:localfourierassumovern}
\int_{T(\QQ_v)} H_v(x,-s,D)\chi(x)^{-1}\,d\mu_v(x) 
=
\sum_{n \in T(\QQ_v)/\IntegralCompact{v}}
    q_v^{\frac{1}{e_v}\varphi(\lt_w(n))}
\charfn^D_v(n) 
    \chi(n)^{-1}.
\end{equation}
The number of $n$ in the sum which are equivalent 
    modulo $\MaximalCompact{v}$ is bounded by 
    the index of $\IntegralCompact{v}$ in $\MaximalCompact{v}$ 
    which is finite, 
    so the sum is absolutely convergent 
    since $\mathrm{Re}(s_e)>0$ 
    for all $e$. 
\end{proof} 


\subsubsection{Toric ideals}\label{sec:ToricIdeals} 

We amplify the last computation by considering 
    all finite places at once. 
For this purpose we introduce some notation. 
Let $K_{\mathrm{fin}}$ be an arbitrary open compact subgroup  
    of $T(\AA_{\mathrm{fin}})$ and set 
    $K = K_{\mathrm{fin}} \times \MaximalCompact{\infty}$. 

\begin{definition}[toric ideals]
The \defn{group of fractional ideals of $T$ (of 
    level $K$)} is 
\begin{equation} 
    \ToricFracIdeals = 
    T(\AA_{\mathrm{fin}})/K_{\mathrm{fin}}.
\end{equation} 
Now assume $K$ contains $\IntegralCompact{}$. 
The \defn{set of $D$-integral ideals of $T$ (of 
    level $K$)} is 
\begin{equation} 
    \ToricIdeals = 
    \{n \in \ToricFracIdeals : \charfn^D(n) = 1\}. 
\end{equation} 
A $D$-integral ideal $n \in \ToricIdeals$ is \defn{principal} 
    if $\psi(n) = 1$ 
    for any class group character $\psi \in C^\vee$. 
The set of $D$-integral principal ideals is denoted 
    $$
    \ToricPrinIdeals=
    \{n \in \ToricIdeals : \psi(n) = 1 
    \text{ for all $\psi \in C^\vee$}\}. 
    $$ 
For any $n \in \ToricFracIdeals$ and 
$\Gamma$-invariant $\Sigma$-linear function $\varphi 
    \leftrightarrow s\in (\CC^{\Sigma(1)})^{\Gamma}$ 
    we set 
\begin{equation} 
    n^{-s}\coloneqq 
    H(n,-s)=
    \prod_{p \in \Places{\QQ}^{\mathrm{fin}}}
        p^{\frac{1}{e_p}\varphi(\lt_w(n_p))} 
        \in \CC^\times.
\end{equation} 
\end{definition}


\begin{corollary}\label{cor:FiniteHeights}
Let $\chi \in (T(\AA)/\IntegralCompact{})^\vee$ be a character. 
The finite part 
    $\widehat{H}_{\mathrm{fin}}(\chi,-s,D)$ of $\widehat{H}$ 
    is equal to 
\begin{equation}\label{eqn:HFin} 
    \widehat{H}_{\mathrm{fin}}(\chi,-s,D) 
    =
    \sum_{n \in \ToricIdeals}
    \chi(n)^{-1}
    n^{-s}
\end{equation} 
    where $\ToricIdeals$ is the set of $D$-integral ideals of $T$ of 
    level $\IntegralCompact{}$. 
There is a constant $\sigma>0$ such that the right-hand side 
    converges absolutely to a holomorphic function of $s$ 
    in the region 
    $\{s \in (\CC^{\Sigma(1)})^\Gamma : \mathrm{Re}(s_e)>\sigma 
    \text{ for every $e \in \Sigma(1)$}\}$. 
\end{corollary}

\begin{proof} 
For each place $v$ choose a place $w$ of $E$ lying over $v$. 
The finite part of the Fourier transform of the height is 
an infinite product 
\begin{equation} 
    \prod_{p \in \Places{\QQ}^{\mathrm{fin}}}
    \sum_{n \in T(\QQ_p)/\IntegralCompact{p}}
    \chi_p(n)^{-1} 
    \charfn^D_v(n)
    p^{\frac{1}{e_p}\varphi(\lt_w(n))} .
\end{equation} 
For all but finitely many $p$ 
    we have $\IntegralCompact{p} = \MaximalCompact{p}$, 
    $e_p = 1$, and $p$-factor 
    (Proposition~\ref{prop:basicfacts}) 
\begin{equation} 
    \sum_{n \in \ccl_w}
    \chi_p(n)^{-1} 
    \charfn^D_p(n)
    p^{\varphi(n)} .
\end{equation} 
Let $\norm{\cdot}$ be any norm on $\ccs_E$. 
Then 
$$
    \left|
    \chi_p(n)^{-1}
    \charfn^D_p(n)
    p^{\varphi(n)}
    \right| 
\leq 
    p^{-\rho t\norm{n}}
$$ 
where $t = \min\{\mathrm{Re}(s_e) : e \in \Sigma(1)\}$ 
    and $\rho>0$ is some constant independent of $p$ and $n$. 
Then 
\begin{equation}\label{eqn:EstimateForLocalFactor} 
    \sum_{n \in \ccl_w}
    \chi_p(n)^{-1} 
    \charfn^D_p(n)
    p^{\varphi(n)} 
    =
    1+O(p^{-\rho t|n_0|})
\end{equation} 
for a minimally chosen $n_0$ and 
    an implied constant independent of $p$. 
This shows that the infinite product defining 
    $\widehat{H}_{\mathrm{fin}}(\chi,-s,D)$ 
    converges to a holomorphic function of $s$ 
    and may be multiplied out into 
    an absolutely convergent infinite sum, 
    for some constant $\sigma>0$. 
\end{proof} 

\begin{remark}\label{rmk:temptingIdentification}
It is tempting to think of $n^{-s}$ as $e^{\varphi(\Lt(n))}$ 
    but in general these are different quantities 
    due to nonlinearity of $\varphi$. 
However if $\varphi$ is linear on the cone generated by 
    every element 
    $$
    \lt_w(n_p)\log q_w 
    $$
    where $n_p \in T(\QQ_p)$ is any element 
    such that 
    $\lt_w(n_p)$ is $D$-integral and 
    $w \in \Places{E}^{\mathrm{fin}}$ lies over $v$, 
    then $n^{-s}=e^{\varphi(\Lt(n))}$. 
This occurs notably when $\varphi$ is linear 
    or when the $D$-integral region of $\ccs_E$ is 
    $\Sigma$-convex. 
\end{remark}

\begin{example}[toric ideals for restriction tori]\label{example:Restriction2} 
We verify that our notions on toric ideals 
    agree with the usual ones for restriction tori. 
Take $T = R^E_\QQ \GG_m$ where $E$ is 
    a finite Galois extension of $\QQ$. 
There is a short exact sequence with finite kernel 
\begin{equation} 
    1 \longrightarrow
    \MaximalCompact{}/(\ToricROU \cdot \IntegralCompact{}) 
    \longrightarrow 
    \ToricFracIdeals
    \longrightarrow
    \{\text{fractional ideals of $E$}\}
    \longrightarrow 1.
\end{equation} 
We may compactify $T$ according to 
    the usual fan of projective space $\PP_n$ 
    to obtain a Severi--Brauer compactification\footnote{ 
This compactification is isomorphic to 
    $\PP_n$ over $\QQ$ as a variety 
    but not as a toric variety: any such isomorphism 
    will not be equivariant for the toric actions.} 
    of $T$. 
Now take the boundary divisor to be $D=D_\infty$ 
    corresponding to the unique Galois-fixed generator 
    $e_\infty$ of the fan. 
This divisor is Cartier and 
    we use the canonical integral model 
    for the complement of $D_\infty$ 
    (Remark~\ref{rmk:CartierIntegrality}). 
Then $\IntegralCompact{}=\MaximalCompact{}$ 
    and the $D_\infty$-integral ideals of $T$ 
    of level $K = \MaximalCompact{}$ 
    are precisely the (nonzero) ideals of $O_E$. 
An element $n \in \ToricIdeals$ is principal by our definition 
    if $\psi(n) = 1$ for every class group character $\psi$. 
This is equivalent to requiring that the image $[n]$ of $n$ under 
    $\ToricIdeals \to T(\AA_{\mathrm{fin}}) \to T(\AA) \to_s T(\AA)^1 \to T(\QQ)\backslash T(\AA)^1/K$ is in the identity component $R$. 
With the help of Example~\ref{example:Restriction1} 
    one finds that 
\begin{equation} 
    [n] = [N(I)^{\frac{d_\infty}{[E:\QQ]}}1_{\infty}\,; 
    \,t_{\mathrm{fin}}]
\end{equation} 
where $I$ is the ideal of $O_E$ corresponding to $n$ 
    and $t_{\mathrm{fin}} \in T(\AA_{\mathrm{fin}})$ is 
    any element mapping to $n$. 
Regarding $[n]$ as an element of $T(\QQ) \backslash T(\AA)/K$ 
    via the natural inclusion $T(\AA)^1 \subset T(\AA)$ 
    and using the description of the identity component of 
    $T(\QQ) \backslash T(\AA)/K$ 
    from the proof of Lemma~\ref{lemma:normonesubgroup} 
    shows that 
\begin{equation} 
[N(I)^{\frac{d_\infty}{[E:\QQ]}}1_{\infty}\,; 
    \,t_{\mathrm{fin}}]
\in 
    T(\QQ)\backslash (T(\QQ)T(\RR)K)/K 
    \subset 
    T(\QQ) \backslash T(\AA)/K
\end{equation} 
which is equivalent to 
    $[N(I)^{\frac{d_\infty}{[E:\QQ]}}1_{\infty}\,; 
    \,t_{\mathrm{fin}}]
    \in 
    T(\QQ)T(\RR)K$. 
This shows that 
    $t_{\mathrm{fin}} \in T(\QQ)_{\mathrm{fin}}K_{\mathrm{fin}}$ 
    which is the condition that $I$ is principal 
    in the usual sense. 
We may also evaluate $n^{-s}$ for $n \in \ToricIdeals$. 
The $D$-integral region is $\Sigma$-convex 
    so $n^{-s} = e^{\varphi(\Lt(n))}$ 
    (Remark~\ref{rmk:temptingIdentification}). 
Any $s \in (\CC^{\Sigma(1)})^{\Gamma}$ 
    satisfies $s_1 = \cdots = s_{[E:\QQ]}$ and so 
    $n^{-s} = N(I)^{-s_1}$ by integrality of $I$ 
    and \eqref{eqn:LtRestrTori}. 
The general quasi-split case is similar: 
if $T = R^{E_1}_\QQ \GG_m \times \cdots \times R^{E_r}_\QQ \GG_m$ 
then the elements $n$ of $\ToricIdeals$ 
    correspond to ideals 
    $I_1\times \cdots \times I_r$ 
    of $O_{E_1}\times \cdots \times O_{E_r}$ and 
$n^{-s} = N(I_1)^{-s_{O_1}}
\cdots
N(I_r)^{-s_{O_r}}$ 
where $s_{O_j}$ is the common value of $s_e$ 
on the $j$th Galois orbit $O_j$ of $\Sigma(1)$ 
corresponding to $E_j$. 
\end{example} 

\subsubsection{Local Fourier transforms: archimedean case}\label{sec:fouriera} 

In the nonarchimedean setting, the local Fourier transform 
    has contributions from cones of every dimension 
whereas in the archimedean setting, the image of $\lt_w$ is $\ccs_w$ 
and only cones of maximal dimension have positive measure. 
Let $w$ be an infinite place of $E$. 
Let $r_\infty$ be the rank of $\ccs_w$. 
Write $\Sigma_\infty$ for the restricted fan 
    $\Sigma_{\Gamma_w}$. 

\begin{proposition}\label{prop:infiniteheightfouriertrfm}
Let $s \in (\CC^{\Sigma(1)})^{\Gamma_w}$
    and let $\varphi \colon \ccs_w \to \CC$ 
    be the corresponding $\Sigma$-linear function 
    (Proposition~\ref{prop:supportFunctions}). 
For any $e \in \Sigma_\infty(1)$ let $s_e=-\varphi(e)$. 
Assume $\mathrm{Re}(s_e)>0$ for all $e \in \Sigma_\infty(1)$. 
Let $\chi\in T(\RR)^\vee$ be a character. 
If $\chi$ is $K_\infty$-ramified then 
    $\widehat{H_\infty}(\chi,-s,D)$ is identically zero. 
If $\chi$ is $K_\infty$-unramified, then 
there is a unique $m \in \cs_w$ 
    such that $\chi(t)= e(\langle \lt_w(t) ,m\rangle)$ 
    for all $t \in T(\RR)$, 
    and we have that 
    \begin{equation}\label{eqn:nonarchfouriertransform} 
    \widehat{H_\infty}(\chi,-s,D) = 
        \left(\frac{1}{2\pi i}\right)^{r_\infty}
        \sum_{\tau \in \Sigma_\infty(r_\infty)}
        [\ccl_w:\ZZ \tau(1)] 
    \prod_{e \in \tau(1)}
        \left(m(e) + \frac{s_e}{2\pi i}\right)^{-1}. 
\end{equation} 
\end{proposition}



\begin{proof} 
Note that $\charfn^D_\infty$ is identically one on $T(\RR)$ 
    and $H_\infty$ is $K_\infty=\MaximalCompact{\infty}$-invariant. 
If $\chi$ is ramified, 
then $\widehat{H_\infty}(\chi,-s,D)$ vanishes by Schur's lemma 
    for $\chi$ on $\MaximalCompact{\infty}$. 
If $\chi$ is unramified, then 
\begin{equation}\label{eqn:intermediateFormulaH} 
\widehat{H_\infty}(\chi,-s,D)
=
\int_{\ccs_w}
    H_\infty(y,-s) e(-\langle y,m\rangle)
    \,d\mu_\infty(y)
=
\int_{\ccs_w}
    e^{\varphi(y)} e(-\langle y,m\rangle)
    \,d\mu_\infty(y)
\end{equation} 
Since $Y$ is projective, 
the restricted fan $\Sigma_\infty$ is complete so 
\begin{equation} 
    \widehat{H_\infty}(\chi,-s,D)
=
\sum_{\tau \in \Sigma_\infty(r_\infty)}
    \int_{\tau} 
    e^{\varphi(y)} e(-\langle y,m\rangle)
    \,d\mu_\infty(y).
\end{equation} 
Since $\Sigma_\infty$ is simplicial, 
    the generators 
    $\tau(1)=\{e_{j_1},\ldots,e_{j_{r_\infty}}\}$ of $\tau$ 
    form a basis of $\ccs_w$. 
    Write $y = x_1e_{j_1}+\cdots+x_{r_\infty} e_{j_{r_\infty}}$. 
Under the isomorphism $\RR^{r_\infty} \to \ccs_w$ 
    determined by $\tau(1)$, 
    the measure $\mu_\infty$ pulls back to 
    $[\ccl_w:\ZZ\tau(1)]\mu_{\text{Lebesgue}}$. 
Then 
\begin{multline} 
\int_{\tau} 
    e^{\varphi(y)} e(-\langle y,m\rangle)
    \,d\mu_\infty(y)
    = 
    [\ccl_w:\ZZ\tau(1)] 
    \int_{\RR_{\geq 0}^{r_\infty}}
    e^{-x_1s_{j_1}-\cdots-x_{r_\infty} s_{j_{r_\infty}}} 
    e(-\langle y,m\rangle) 
    \,
    dx_1 \wedge \cdots \wedge dx_{r_\infty}\\
    = 
    [\ccl_w:\ZZ\tau(1)]
\prod_{e \in \tau(1)}
\int_{0}^\infty
    e^{-x (s_e+2\pi i \langle e, m\rangle ) }
    \, dx
    =
    [\ccl_w:\ZZ\tau(1)]
    \prod_{e \in \tau(1)}
    (s_e + 2 \pi i \langle e,m \rangle )^{-1}.
\end{multline} 
The integral converges 
if and only if 
$\mathrm{Re}(s_e)>0$ for all $e \in \Sigma_\infty(1)$. 
\end{proof} 




\subsection{Characters contributing to the expected principal term}\label{sec:weak} 

Let $P$ denote the quasisplit torus 
    whose character lattice is $\Pic^T(Y_E)\cong\ZZ^{\Sigma(1)}$. 
The natural map $\cl_E \to \ZZ^{\Sigma(1)}$ 
    (Proposition~\ref{prop:supportFunctions}) 
    induces an isogeny 
    $$\gamma \colon P \to T$$ 
    of algebraic tori. 
The summands of our formula for $Z(s)$ 
    are indexed by automorphic characters of $T$ 
    on $R \times C$, 
    and the summands which contribute to 
    the expected principal term are 
    those characters which 
    become trivial after pulling back by $\gamma^\ast$ to $P$. 
The following essentially trivial lemma 
    will ensure that every such character 
    is trivial on the regulator group $R$. 

\begin{lemma}
    Let $\gamma \colon T' \to T$ 
    be a surjective homomorphism of algebraic tori 
    and let $\gamma^\ast \colon T(\AA)^\vee \to T'(\AA)^\vee$ 
    denote the pullback homomorphism. 
    Then there is a positive integer $C(\gamma)$ annihilating 
    the kernel of $\gamma^\ast$. 
\end{lemma}

\begin{proof} 
Recall the category of diagonalizable algebraic groups 
    which split over $E/\QQ$ 
    is equivalent to the category of finite-type 
    $\ZZ\Gamma$-modules. 
Thus the cokernel of $\gamma$ is representable 
    as a finite group scheme. 
If $\chi$ is a unitary character of $T(\AA)$ 
    for which $\chi \circ \gamma = 1$,  
    then $\chi$ factors through the cokernel. 
Thus $\chi$ is annihilated by 
    $C(\gamma)=|\mathrm{coker}(\gamma)(\overline{\QQ})|$. 
\end{proof} 

\begin{remark} 
\cite[Theorem 3.1.1]{IMRN}, 
\cite[Proposition~6.4]{batyrev1995maninsconjecturetoricvarieties}, 
and \cite[p.~19]{HZF} 
    assert that 
    $\ker \gamma^\ast$ is dual 
    to the obstruction group to weak approximation\footnote{See \cite[Theorem~6]{zbMATH03401075} for a definition 
    of this group.} of $T$ and cite \cite{MR291099}. 
In the later paper \cite{integral-chambert-tschinkel} 
    this assertion appears as Lemma 3.8.1, 
    and their explanation 
    cites the proof of \cite[Theorem~6]{zbMATH03401075}. 
We were unable to see how their assertion follows from 
    either \cite{MR291099} or 
    the proof of \cite[Theorem~6]{zbMATH03401075}, 
    so we will not make use of 
    this description of $\ker \gamma^\ast$. 
For our purposes, it is enough to know that 
    $\ker \gamma^\ast$ is annihilated by $C(\gamma)$. 
\end{remark} 


%
%
%
%
%


\section{Iterated residues of multiple Dirichlet series} 

%

\subsection{The Poisson summation formula}\label{sec:Poisson} 

Let $B$ be a locally compact abelian group with Haar measure $db$. 
Let $f \in L^1(B)$. 
The Fourier transform of $f$ given by 
\begin{equation} 
    \widehat{f}(\chi) = \int_B f(b) \chi(b)^{-1} \, db
\end{equation} 
    converges and defines a continuous function on $B^\vee$. 

Now let $A$ be a closed subgroup of $B$ 
and let $A^\perp \subset B^\vee$ denote 
the subgroup of characters on $B$ that are trivial on $A$. 
The general Poisson summation formula --- 
whose proof is essentially the classical proof for $\ZZ \subset \RR$ --- 
says that if $\widehat{f}|_{A^\perp} \in L^1(A^\perp)$ then 
\begin{equation}\label{eqn:poissonae} 
    \int_A f(ab)\, da = 
    \int_{A^\perp}
    \widehat{f}(\chi)
    \chi(b)
    \,d\chi 
\end{equation} 
for a suitably normalized Haar measure on $A^\perp$ 
and all $b \in B$ away from a subset of measure zero 
\cite[Theorem~4.4.2, p.~105]{MR1397028}. 

We will apply the Poisson formula for 
    the discrete subgroup of rational points 
    in the adelic space of the torus: 
\begin{equation} 
    A = T(\QQ)  \subset B = T(\AA).
\end{equation} 
This leads to a difficult integral over 
the automorphic spectrum 
$T(\QQ)^\perp = (T(\AA)/T(\QQ))^\vee$ of $T$ 
that was first studied in general in \cite{zbMATH01353487}. 
Whenever the torus $T$ admits nonconstant characters 
    defined over $\QQ$, 
    the domain of integration is noncompact; 
    this has an interesting effect 
    on the integral which is not limited to 
    expected difficulties concerning convergence. 
We will apply the method in \cite{residue} 
    to obtain an exact expression for 
    this integral in terms of higher-dimensional residues. 
This improves on 
    the analysis in \cite{zbMATH01353487} 
    which only approximated this integral. 

\begin{remark} 
This use of Poisson summation is different from Tate's thesis 
    where Poisson summation is applied 
    for the subgroup $E \subset \AA_E$ 
    in the \emph{additive} adeles; 
unlike the additive group $\AA_E$, 
    the multiplicative group $T(\AA)$ is not self-dual. 
On the spectral side, 
    the Poisson formula in the multiplicative setting 
    involves an integral over 
    Hecke characters on $T$ of general level. 
It is better regarded as a special case of 
    the Arthur--Selberg trace formula. 
\end{remark} 


\subsection{The method of iterated residues}\label{sec:methodIterRes} 

The Poisson formula will lead to the following integral: 
\begin{equation}\label{eqn:MotivatingIntegral} 
    \int_{\cs}
    \chi_m(n)^{-1}
    \widehat{H_\infty}(\xi \chi_m,-s)
    \,dm
\end{equation} 
where $n \in \ToricIdeals$ is a toric ideal, 
    $\chi_m \colon \cs \to \CC^\times$ 
    is the unitary character $\chi_m(n) = e(m(n))$, 
    and $\xi \in (T(\QQ)\backslash T(\AA)/K)^\vee$ 
    is an automorphic character of level $K$. 
We will evaluate this integral 
    using the residue formula from 
    \cite[Theorem~1]{residue}. 
This formula proceeds by choosing a polyhedron $\Pi$ 
    and iteratively applying Cauchy's residue formula 
    by deforming along the rays of 
    the imaginary part of $\Pi$ one at a time. 
This expresses the integral as a sum over 
    \emph{iterated residues} 
    which are indexed by \emph{affine flags} 
    cut out by the affine hyperplanes where 
    $\widehat{H_\infty}(\chi_m,-s)$ is singular, 
    and the problem is to determine 
    which flags contribute a residue. 
This is answered in \cite[Theorem~1]{residue} 
    using sign conditions on the minors of 
    the defining matrices for each flag. 

\subsubsection{Minors} 

For an $r\times r$ real matrix $J = (a_{ij})$, 
we consider the $k$th leading principal minor 
\begin{equation} 
   p_k= 
   \det
    \begin{bmatrix}
        a_{11}&\cdots&a_{1k}\\ 
        \vdots& \ddots&\vdots\\
        a_{k1}&\cdots&a_{kk}\\ 
    \end{bmatrix}
    \qquad(k \in \{1,\ldots,r\}), 
\end{equation} 
the $k\times k$ minor 
\begin{equation} 
    q_{k\ell}=\det 
    \begin{bmatrix}
        a_{11}&\cdots&a_{1,k-1}&a_{1\ell}\\ 
        \vdots& \ddots&\vdots&\vdots\\
        a_{k1}&\cdots&a_{k,k-1}&a_{k\ell}\\ 
    \end{bmatrix}
    \qquad(\ell \in\{ k+1,\ldots,r\}), 
\end{equation} 
and the $(k-1)\times(k-1)$ minor 
\begin{equation} 
    r_{jk}= \det 
    \begin{bmatrix}
        a_{11}&\cdots&a_{1,k-1}\\
        \vdots&\ddots&\vdots\\
        a_{j-1,1}&\cdots&a_{j-1,k-1}\\
        a_{j+1,1}&\cdots&a_{j+1,k-1}\\
        \vdots&\ddots&\vdots\\
        a_{k1}&\cdots&a_{k,k-1}\\ 
    \end{bmatrix}
    \qquad
    (j \in \{1,\ldots,k-1\}). 
\end{equation} 
Let $J$ be a $k \times r$ real matrix. 
We say $J$ is \defn{stable} if 
    $$\text{$p_1,\ldots,p_k>0$ 
    and $(-1)^{\ell-j}r_{j\ell} \geq 0$ 
    for all $1\leq j < \ell \leq k$.}$$ 
We say $J$ is \defn{compatible} if 
    either it is not stable or 
\begin{equation} 
    q_{j\ell} \leq 0
        \,\,\,\text{for all $1 \leq j < \ell \leq k$.}
\end{equation} 


\subsubsection{Polyhedra}\label{sec:polyhedra} 

By a \defn{polyhedron (with boundary $\cs$)} we mean 
    a closed subset $\Pi \subset \cs_\CC$ and 
    a linear isomorphism $\cs \xrightarrow{\sim} \RR^r$ 
    such that $\Pi$ is identified with 
    the product of the closed upper half-planes in $\CC^r$ 
    under the induced linear isomorphism $\cs_\CC \xrightarrow{\sim} \CC^r$. 
We consider 
    the defining map $\cs \xrightarrow{\sim} \RR^r$ 
    to be part of the data of $\Pi$. 
Let $\Pi =\cs + i\Theta$ be a polyhedron. 
Its imaginary cone 
    $\Theta=\RR_{\geq 0}\langle v_1,\ldots,v_r\rangle$ 
    is generated by 
    the distinguished basis $(v_1,\ldots,v_r)$ of $\cs$ 
    determined by the defining map.     
Let $z$ be the dual basis to $(v_1,\ldots,v_r)$. 
Let
\begin{equation} 
\gamma \,\,\,: \,\,\,
\cs_\CC \supset 
H_{1} \supset 
H_{1} \cap H_{2} \supset \cdots \supset
H_{1} \cap \cdots \cap H_{r}
\end{equation} 
    be a flag determined by a collection 
    $H=(H_1,\ldots,H_r)$ 
    of linearly independent singular hyperplanes 
    with respective defining forms $e_1,\ldots,e_r \in E$. 
Let $J=J_{H,\Pi}$ be the Jacobian matrix 
\begin{equation} 
    J = \frac{\partial e}{\partial z}= \begin{bmatrix}
        e_1(v_1)&\cdots&e_1(v_r)\\ 
        \vdots& \ddots&\vdots\\
        e_r(v_1)&\cdots&e_r(v_r)\\ 
    \end{bmatrix}.
\end{equation} 
We say $H$ is \defn{$\Pi$-stable} 
    (resp.~\defn{$\Pi$-compatible}) 
    if $J$ is stable (resp.~compatible). 


\subsubsection{The residue formula} 

Let 
$\gamma : 
\cs_\CC \supset H_{1} \supset 
H_{1} \cap H_{2} \supset \cdots \supset
H_{1} \cap \cdots \cap H_{k}$ 
be a flag. 
We write $\itres_z[\omega,\gamma]$ for the iterated residue 
    of the form $\omega$ along the poles 
    prescribed by the flag $\gamma$, cf.~\cite[\S1]{residue}. 
This is a meromorphic top-degree form on $\gamma(k)$ 
defined as follows. 
We write 
\begin{equation} 
    \res_{z_1}[g\,dz_1,\gamma(1)]
    =
    \res_{z_1}[g\,dz_1,\gamma(1)](z_2,\ldots,z_r)
\end{equation} 
for the residue of $g\,dz_1$ with respect to $z_1$ if 
there exists $z_1^\ast \in \CC$ 
such that the vector $v \in V_\CC$ 
with coordinates $z(v) =(z_1^\ast,z_2(v),\ldots,z_r(v))$ 
is in $\gamma(1)$. 
If no such $z_1$ exists, 
then we set this equal to zero. 
This recursively defines the 
\defn{iterated residue of $\omega$ along $\gamma$}: 
$$\itres_z[\omega,\gamma](z_{k+1},\ldots,z_r)
=
\res_{z_k}[\cdots\res_{z_2}[\res_{z_1}[g\,dz_1,\gamma(1)]\,dz_2,\gamma(2)]\cdots 
\,dz_{k}, \gamma(k)],
$$ 
Let $\omega$ be a meromorphic $r$-form on $\cs_\CC$ 
    which is regular on the complement of 
    an affine hyperplane arrangement 
    $\cs_\CC - \cup_{e \in E} H_e$ 
    where $E\subset \cs$ 
    is a finite set of nonzero real linear forms and 
    $H_e = \{m \in \cs_\CC : e(m) + \frac{s_e}{2\pi i} = 0\} 
    \subset \cs_\CC$ 
for some parameters $s_e$ with positive real part. 

\begin{theorem}[{\cite[Theorem~1]{residue}}]\label{thm:ResidueFormula}
    If $\omega$ admits an iterated residue expansion 
    along the polyhedron $\Pi$ 
    and every singular flag of $\omega$ is $\Pi$-compatible, then 
\begin{equation} 
    \frac{1}{(2\pi i)^{r}}
    \int_{V} \frac{h(z)\,dz}{g_1(z)\cdots g_R(z)}
    = 
    \sum_{\gamma \in Z_\ast}
        \itres_z[\omega,\gamma] 
\end{equation}   
where $Z_\ast \subset \{\gamma \in Z: \gamma(r) \in \Pi\}$ 
    is the subset of flags cut out by 
    a collection of $\Pi$-stable singular hyperplanes. 
\end{theorem}

\subsubsection{Transfer maps and poles}\label{sec:ConesForms} 

Let $\pi \colon \ccs_E \to \ccs$ denote 
    the linear projection map 
    given by restriction to $\cs$. 
Recall that $\ccs_w$ has a canonical subspace decomposition 
\begin{equation} 
    \ccs_w = \ccs \oplus \ccs_0.
\end{equation} 
If $E \subset \ccs_w$ is an $r$-subset 
    whose projection $\pi E$ is a basis of $\ccs$,  
    then it can be used to 
    transfer linear forms on $\ccs_w$ 
    to linear forms on $\ccs$. 
More precisely, for such a subset $E$ we define a linear map 
\begin{equation} 
    \mu_E \colon \cs_w \to \cs
\end{equation} 
by requiring that 
\begin{equation} 
    (\mu_Em)(\pi e) = m(e)\quad
    \text{for all $e \in E$}. 
\end{equation} 
Equivalently, the projection $\pi$ 
    restricts to a bijection of cones 
\begin{equation} 
    \ccs_w\supset
    \RR \langle E\rangle \xrightarrow{\pi} \RR \langle\pi E\rangle
    \subset \ccs,
\end{equation} 
and $\mu_E m$ is the unique linear map on $\ccs$ 
    corresponding to $m|_{\RR\langle E \rangle}$ 
    under this identification.

\begin{remark}\label{rmk:TransferRestriction}
The transfer map $\mu_E$ equals 
    the natural restriction map 
    $\cs_w \to \cs$ 
    if and only if 
    $E \subset \ccs$. 
\end{remark}


Transfer maps help with analyzing 
    the poles of $\widehat{H_\infty}$. 
Let $s \in (\CC^{\Sigma(1)})^\Gamma$, 
    and let $\xi \in R^\vee$. 
Let $\varphi$ be the $\Sigma$-linear map corresponding to $s$. 
Let the exponent $m_\xi \in \cs_0$ 
    be determined by the identity 
    $\xi(t_\infty) = e(\langle m_\xi, \lt_w(t_\infty) \rangle)$ 
    for $t_\infty \in T(\RR) \subset T(\AA)$. 
The next lemma determines the poles and their variation with $\xi$. 
We omit the easy proof. 

\begin{lemma}\label{lemma:VariationOfPoles}
Let $E \subset \Sigma_w(1)$ be any $\Sigma$-convex $r$-subset 
    whose projection $\pi E$ is a basis of $\ccs$. 
Say $\tau \in \Sigma_w(r_w)$ is a cone containing $E$. 
The form $\widehat{H_\infty}(\xi\chi_m,-s)\,dm$ 
    is singular along the hyperplanes 
    $H_e(\xi)=
    \{m \in \cs_\CC : e(m) + e(m_\xi) + \frac{s_e}{2\pi i} = 0\}$ 
    for $e \in E$ and we let 
    $m_{E\xi} = \cap_{e \in E} H_e(\xi)$ 
    denote their intersection. 
Then $m_{E1} \in \cs_\CC$ is equal to 
\begin{equation} 
    m_{E1}=(2\pi i)^{-1}\mu_E(\varphi|_\tau)
\end{equation} 
and 
\begin{equation} 
    m_{E\xi} = m_1-\mu_E(m_\xi)
\end{equation} 
    where $\varphi|_\tau \in \cs_w\otimes \CC$ 
    is the unique linear map 
    agreeing with $\varphi$ on $\tau$. 
\end{lemma}



\subsection{Residual functions}\label{sec:residual} 

Fix an archimedean place $w$ of $E$. 
Let $dm$ be a generator of $\wedge^r \cs$ 
    for which the lattice $\cl$ is unimodular. 
Let $z = (z_1,\ldots,z_r)$ be any basis of $\ccs$. 
For each singular flag $\gamma :
H_{1} \supset 
H_{1} \cap H_{2} \supset \cdots \supset
H_{1} \cap \cdots \cap H_{r}$ 
with defining equations 
    $H_k = \{m \in \cs_\CC : e_k(m) + \frac{s_k}{2\pi i} = 0\}$, 
    let $J_\gamma = \partial e/\partial z$ be the Jacobian of 
    the map $(e_1,\ldots,e_r) \colon \cs \to \RR^r$ 
    with respect to $z$. 
Let $B \subset \GL_r(\CC)$ 
    denote the subgroup of upper-triangular matrices. 
Let $\CC^{\Sigma(1)}_+ \subset \CC^{\Sigma(1)}$ 
    denote the half-space with positive real parts. 

\begin{definition}\label{defn:ResidualFunction} 
Let $s \in (\CC^{\Sigma(1)}_+)^{\Gamma_w}$. 
Let $\omega(m) = \widehat{H_{\infty}}(\chi_m,-s)\,dm$, 
    a meromorphic $r$-form on $\cs$. 
For any singular flag $\gamma$ whose Jacobian 
    $J_\gamma$ is in $B^T B$ 
    we set 
\begin{equation} 
    R_{\gamma}(s) = 
    (2\pi i)^{r}
    \,
    \itres_z[\omega,\gamma].
\end{equation} 
The condition that 
    $J_\gamma \in B^T B$ 
    ensures that $R_\gamma$ is 
    a rational function of 
    $s \in (\CC^{\Sigma(1)})^{\Gamma_w}$ 
    which is independent of $z$ 
    (cf.~\cite[Prop.~4]{residue}). 
Note that whether or not 
    $J_\gamma$ is in $B^T B$ 
    does not depend on $s$ nor 
    the particular choice of defining equations for $\gamma$. 
\end{definition} 


\begin{remark} 
It is important to allow for 
    $s \in (\CC^{\Sigma(1)})^{\Gamma_w}$ 
    and not just $s \in (\CC^{\Sigma(1)})^{\Gamma}$ 
    even though in the end we are really only concerned with 
    the values of the height zeta function on 
    $(\CC^{\Sigma(1)})^{\Gamma}$. 
In the Poisson formula 
    each residual function is summed over 
    the lattice of local exponents $m_\xi$ of 
    regulator characters $\xi \in R^\vee$ in 
    $\cs_0 \subset (\CC^{\Sigma(1)})^{\Gamma_w}$ 
    which are not in $(\CC^{\Sigma(1)})^{\Gamma}$ 
    if $\xi\neq 1$. 
\end{remark} 

Fix an $s \in (\CC^{\Sigma(1)})^{\Gamma_w}$. 
Let $z$ be the terminal point of 
a singular flag of $\widehat{H_\infty}(\chi_m,-s)\,dm$. 
If there are distinct $r$-subsets $E,E' \subset \Sigma_w(1)$ 
such that 
    $\bigcap_{e \in E} H_e = \{z\} = \bigcap_{e' \in E'} H_{e'}$ 
    then $z$ is \defn{ambiguous}, 
    otherwise $z$ is \defn{simple}. 

\begin{lemma}\label{lemma:ResidualFunctionFormula}
There is a dense open subset 
    $U \subset (\CC^{\Sigma(1)}_+)^{\Gamma_w}$ 
    such that for any fixed $s \in U$, 
    the terminal point of each singular flag 
    of $\widehat{H_\infty}(\chi_m,-s)\,dm$ is simple. 
Thus, for any $s$ in this open set, 
    each singular flag $\gamma$ in $\cs \otimes \CC$ is cut out by 
    a unique ordered $r$-subset $E_\gamma \subset \Sigma_w(1)$. 
Let $z_\gamma$ denote the terminal point of $\gamma$ 
    and set $s_\gamma = -2\pi iz_\gamma$. 
The residual function $R_\gamma$ for such a flag $\gamma$ is given by 
\begin{equation}\label{eqn:ResidualFnSimpleCase} 
    R_\gamma(s)
    =
    \sum_{\substack{\tau \in \Sigma_w(r_\infty)\\E_\gamma \subset \tau(1)}}
    \frac{[\ccl_w:\ZZ\tau(1)]}{[\ccl : \ZZ \pi E_\gamma]}
    \prod_{\substack{e \in \tau(1)\\e \not \in E_\gamma}}
    \frac{1}{(s -s_\gamma)_e}.
\end{equation} 
Now consider an arbitrary parameter $s \in (\CC^{\Sigma(1)}_+)^{\Gamma_w}$ 
    and a singular flag $\gamma$ of $\widehat{H_\infty}(\chi_m,-s)\,dm$. 
Let $F$ be the set of singular flags of $\widehat{H_\infty}(\chi_m,-s')\,dm$ 
    which converge to $\gamma$ when $s' \to s$ through values in $U$. 
Then 
\begin{equation}\label{eqn:ResidualFunctionGeneral} 
    R_\gamma(s)
    =
    \lim_{\substack{s' \to s\\ s' \in U}}
    \sum_{\gamma' \in F}
    R_{\gamma'}(s').
\end{equation} 
In the last expression it suffices to sum over those $\gamma' \in F$ 
    whose defining subset $E_{\gamma'}$ is $\Sigma$-convex. 
\end{lemma}

\begin{proof} 
Each singular hyperplane 
    $H_e = \{e(m) + \frac{s_e}{2\pi i} = 0\}$ 
    is associated to an independent parameter 
    $s_e$ ($e \in \Sigma_w(1)$) 
    as $s$ varies in $(\CC^{\Sigma(1)})^{\Gamma_w}$ 
    so the first assertion is clear. 
Now fix $s \in U$ and let 
    $\gamma$ be a singular flag 
    of $\widehat{H_\infty}(\chi_m,-s)\,dm$. 
Since the pole $z_\gamma$ is simple, 
    each step of the iterated residue 
    (starting with $\widehat{H_\infty}(\chi_m,-s)\,dm$) 
    can be computed by eliminating 
    the unique singular factor 
    and restricting the remaining quantities 
    to the next step of $\gamma$. 
This process results in the first formula. 
The general case follows from continuity of 
    iterated residues in the parameters $s$ 
    (cf.~\cite[\S II.5.4]{zbMATH00107779} 
    and \cite[Prop.~4]{residue}). 
For the last statement, if $E_{\gamma'}$ is not $\Sigma$-convex 
    then $R_{\gamma'} \equiv 0$ 
    since the indexing set for the sum in 
    \eqref{eqn:ResidualFnSimpleCase} is empty. 
\end{proof} 

\begin{lemma}\label{lemma:ResidualPicard} 
Each residual function is invariant 
    under translation by $s \mapsto s+2\pi im$ 
    for any linear form $m \in \cs \otimes \CC$. 
Thus, each residual function factors through the quotient 
$(\CC^{\Sigma(1)})^{\Gamma} \twoheadrightarrow  
    (\CC^{\Sigma(1)})^{\Gamma}/(\cs \otimes \CC )
    = (\Pic Y) \otimes \CC \dashrightarrow \CC$ 
to a rational function on the Picard group. 
\end{lemma} 

\begin{proof} 
We claim $s-s_\gamma$ is invariant 
    under the given translation. 
Let $\pi \colon \ccs_E \to \ccs$ denote 
    the linear projection map 
    given by restriction to $\cs$. 
The pole $z_\gamma$ is the intersection of 
    the hyperplanes $H_e \subset \cs_\CC$ for $e \in E_\gamma$ 
    so it is the unique linear map on $\ccs_\CC$ 
    satisfying 
    $z_\gamma(\pi e) = e(z_\gamma)= -\frac{s_e}{2\pi i}$ 
    for each $e \in E_\gamma$. 
From this one easily finds that, 
    if $z_\gamma'$ denotes the terminal point of 
    the translation of $\gamma$, then 
    $z_\gamma' = z_\gamma - m$ 
    and $s_\gamma' = s_\gamma + 2\pi im$. 
This proves invariance of $R_\gamma$. 

Now pass to $\Gamma$-invariants of the first row 
    from the diagram in Proposition~\ref{prop:supportFunctions} 
    to obtain the long exact sequence 
    $0 \longrightarrow \cl
    \longrightarrow \Div_T Y 
    \longrightarrow \Pic Y 
    \longrightarrow H^1(\Gamma,\cl_E) 
    \longrightarrow \cdots$. 
Since $H^1(\Gamma,\cl_E)$ is finite 
    the complexified Picard group is isomorphic to 
    $(\CC^{\Sigma(1)})^{\Gamma}/(\cs \otimes \CC)$. 
\end{proof} 

\subsubsection{Nonvanishing of residual forms} 

We next consider a class of $\PP_n$-like toric varieties 
    whose Galois action satisfies a convexity condition. 

\begin{definition}
    The Galois action on the fan $\Sigma$ is \defn{convex} 
    if the $\Gamma$-orbit $\Gamma e$ is $\Sigma$-convex 
    for every ray generator $e \in \Sigma_w(1)$. 
Consider a $\Sigma$-convex $r$-subset $E \subset \Sigma_w(1)$ 
    whose restriction to $\cs$ is linearly independent. 
If the Galois action on $\Sigma$ is convex and 
    the set $\pi E=\{\pi e : e \in E\}$ is $\Sigma$-convex 
    for every such $E$ then we say 
    the Galois action on $\Sigma$ is \defn{strongly convex}. 
\end{definition}

\begin{remark}
Twists of projective space, 
and more generally their quotients 
    by finite subgroups of the torus, 
    have fans with strongly convex Galois action. 
\end{remark}

\begin{example}\label{example:antenna} 
The depicted fan with Galois action 
    determines a complete toric surface over $\QQ$ 
    with cyclic quotient singularities 
    whose open torus splits over a quadratic extension. 
The Galois action is not convex. 
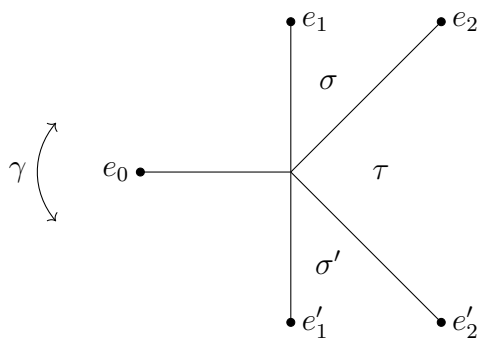
\begin{figure}[h!]
    \begin{tikzpicture}[scale=2] 


\draw (0,0) -- (0,1);
\draw (0,0) -- (-1,0);
\draw (0,0) -- (0,-1);
\draw (0,0) -- (1,1);
\draw (0,0) -- (1,-1);

\filldraw[black] (0,1) circle (.75pt) 
        node[anchor=west]{$e_1$};
\filldraw[black] (0,-1) circle (.75pt) 
        node[anchor=west]{$e_1'$};
\filldraw[black] (1,1) circle (.75pt) 
        node[anchor=west]{$e_2$};
\filldraw[black] (1,-1) circle (.75pt) 
        node[anchor=west]{$e_2'$};
\filldraw[black] (-1,0) circle (.75pt) 
        node[anchor=east]{$e_{0}$};

\node (a) at (-1.5,.4) {};
\node (b) at (-1.5,-.4) {};

\node at (.6,0) {$\tau$};
\node at (.25,.6) {$\sigma$};
\node at (.25,-.6) {$\sigma'$};

\draw [<->] (b) to [bend left=40] node [left] {$\gamma$} (a) ;




\end{tikzpicture} 
\end{figure}
\end{example} 

\begin{proposition}\label{prop:nonvanishingR}
    Assume the Galois action is strongly convex. 
    Let $\sigma \in \Sigma(r)$ and let $\gamma \in Z_\sigma$.  
Then $$R_\gamma(s)>0$$ for any $s \in \Pic Y \otimes \RR$ 
    in the interior of the cone of effective divisors. 
In particular, if $L$ is a big line bundle on $Y$, then 
    $R_\gamma(t[L]) >0$ for any $t >0$. 
\end{proposition}

\begin{proof} 
The rational function $R_\gamma(s)$ is regular 
    for any $s \in (\CC^{\Sigma(1)}_+)^{\Gamma_w}$ 
    since the iterated residue is always well-defined. 
The iterated residue does not vanish identically 
    since $\gamma$ may be defined 
    by a $\Pi_\sigma$-stable Jacobian. 
In view of \eqref{eqn:ResidualFunctionGeneral} 
    it suffices to consider the case 
    that the terminal point of $\gamma$ is simple. 
It is clear from \eqref{eqn:ResidualFnSimpleCase} 
    that the sign of the real number $R_\gamma(s)$ can only 
    change at a singularity of $R_\gamma$, 
    so it suffices to show that $R_\gamma(s)>0$ 
    for a single $s$ in the interior of the effective cone. 
Since $Y$ is projective, it admits an ample line bundle $L$. 
Replacing $L$ with a positive multiple of itself, 
    we may assume it admits a $T$-linearization 
    since $H^1(\Gamma,\cl_E)$ is finite. 
Let $s$ be the element of 
    $(\ZZ^{\Sigma(1)})^\Gamma \subset (\CC^{\Sigma(1)})^\Gamma$ 
    corresponding to $L$ with its $T$-linearization 
    (Proposition~\ref{prop:supportFunctions}). 
Let $\varphi$ be the corresponding $\Sigma$-linear function. 
Since $L$ is ample it is globally generated 
    (not true in general for proper varieties 
    but true for proper toric varieties \cite[p.~70]{MR1234037}). 
It is well-known that $L$ is globally generated if and only if 
    $\varphi$ is upper convex \cite[p.~68]{MR1234037}, so 
    \begin{equation}\label{eqn:UpperConvex} 
    (\varphi-\varphi|_\tau)(v) \leq 0
    \end{equation} 
    for any cone $\tau \in \Sigma$ and $v \not \in \tau$ 
    where $\varphi|_\tau$ is a linear map 
    on $\ccs_E$ agreeing with $\varphi$. 
The terminal point of $\gamma=\gamma(s)$ is 
    $z_{\gamma}=(2\pi i)^{-1}\mu_{E_{\gamma}}(\varphi|_\tau)$ 
    (Lemma~\ref{lemma:VariationOfPoles}). 
Since the Galois action is convex we have 
    $\mu_{E_{\gamma}}(\varphi|_\tau) = \varphi|_{\pi E_{\gamma}}$ 
    where $\varphi|_{\pi E_{\gamma}}\in \cs\otimes \CC$ 
    is the unique linear map agreeing with $\varphi$ 
    on the basis $\pi E_{\gamma}$. 
Since the Galois action is strongly convex, 
    $\varphi|_{\pi E_{\gamma}}=\varphi|_{\sigma}$ 
    where $\sigma=\langle \pi E_\gamma \rangle\in\Sigma_\Gamma(r)$. 
In particular, $s_\sigma=s_\gamma$ only depends on $\sigma$. 
Now let $\tau \in \Sigma_w(r_w)$ be a cone containing $E_{\gamma}$ 
and let $e \in \tau(1) - E_{\gamma}$. 
Since the Galois action is convex, 
    the ray $\tau_e=\RR_{\geq 0}\pi e$ 
    is a cone in $\Sigma_\Gamma$ 
    by Lemma~\ref{lemma:subfan}. 
Suppose the ray $\RR_{\geq 0}\pi e$ were contained in $\sigma$. 
Then it would be an edge of $\sigma$, 
    and so $\pi e = \lambda \pi e'$ for some $e' \in E_\gamma$ 
    and $\lambda>0$. 
This would imply $e$ and $e'$ are in the same Galois orbit. 
Since $s$ is $\Gamma$-invariant, this would mean that 
    $e+\frac{s_e}{2\pi i}$ and $e'+\frac{s_{e'}}{2\pi i}$ 
    define the same hyperplane in $\cs \otimes \CC$, 
    which contradicts simplicity of $\gamma$. 
We conclude that $\pi e$ is 
    not contained in $\sigma$ by simplicity of $\gamma$. 
Now the desired positivity 
\begin{equation} 
    (s-s_\gamma)_e = 
    -(\varphi-\varphi|_\sigma)(e) =
    -(\varphi-\varphi|_\sigma)(\pi e) >0
\end{equation} 
for $e$ follows from \eqref{eqn:UpperConvex} 
and regularity of $R_\gamma$. 
\end{proof} 


\subsubsection{Estimates on residual forms} 

In this section we show that the form 
    $\omega = \widehat{H_\infty}(\chi_m,-s)\,dm$ 
    admits an iterated residue expansion 
    for any fixed $s \in (\CC^{\Sigma(1)})^{\Gamma_w}$ 
    with positive real parts. 


\begin{lemma}\label{lemma:residualFormUpperBound} 
Let 
    $\gamma : \cs_{\CC} \supset H_1 \supset \cdots 
    \supset H_1 \cap \cdots \cap H_k$ 
    be a $k$-step singular flag in $\cs_{\CC}$. 
Then 
\begin{equation} 
    |\!\itres_z[\omega,\gamma]| \ll 
    \frac{1}{\norm{m}}
    \sum_{\tau\in\Sigma_\infty(r_\infty)}
    \prod_{\substack{e \in \tau(1)\\H_e \not \in \{H_1,\ldots,H_k\}}}
        \left|m(e) + \frac{s_e}{2\pi i}\right|^{-1}
        \qquad
    (m \in \cs_{\CC}). 
\end{equation} 
\end{lemma} 

\begin{proof} 
We use part of the proof of \cite[Prop.~2.3.2]{IMRN}. 
Recall the earlier formula 
    \eqref{eqn:intermediateFormulaH} 
    for $\widehat{H_\infty}$: 
\begin{equation} 
    \widehat{H_\infty}(\chi_m,-s) = 
\int_{\ccs_w}
    e^{\varphi(y)} e(-\langle y,m\rangle)
    \,d\mu_\infty(y).
\end{equation} 
Let $(v_1,\ldots,v_{r_w})$ be any basis for $\ccs_w$ 
and write a general element of $\ccs_w$ as 
    $y = y_1v_1+\cdots +y_{r_w}v_{r_w}$. 
Integrating by parts shows that for any $j \in \{1,\ldots,r_w\}$ 
and $m \not \in \ker v_j$, 
\begin{equation} 
\widehat{H_\infty}(\chi_m,-s) =
\int_{\ccs_w}
    e^{\varphi(y)} e(-\langle y,m\rangle)
    \,d\mu_\infty(y)
    =
    \frac{1}{-2\pi i v_j(m)}
\int_{\ccs_w}
    \frac{\partial}{\partial y_j}
    (e^{\varphi(y)})
    e(-\langle y,m\rangle)
    \,d\mu_\infty(y).
\end{equation} 
The function $\varphi(y)$ is linear on each cone so the formula 
    $\frac{\partial}{\partial y_j}(e^{\varphi(y)})
    =\frac{\partial \varphi}{\partial y_j}e^{\varphi(y)}$ 
    is valid and well-defined almost everywhere. 
Splitting up the integral over cones obtains 
\begin{align} 
    \widehat{H_\infty}(\chi_m,-s,D)
    &=
\frac{1}{-2\pi i v_j(m)}
\sum_{\tau \in \Sigma_\infty(r_\infty)}
    \frac{\partial \varphi|_\tau}{\partial y_j}
    \int_{\tau} 
    e^{\varphi(y)} e(-\langle y,m\rangle)
    \,d\mu_\infty(y)\\
    &=
\frac{1}{-2\pi i v_j(m)}
        \left(\frac{1}{2\pi i}\right)^{r_w}
        \sum_{\tau \in \Sigma_\infty(r_\infty)}
        [\ccl_w:\ZZ \tau(1)] 
    \frac{\partial \varphi|_\tau}{\partial y_j}
    \prod_{e \in \tau(1)}
        \left(m(e) + \frac{s_e}{2\pi i}\right)^{-1}. 
\end{align} 
The nonzero form $v_j$ cannot vanish identically on 
    any of the affine hyperplanes $H_1,\ldots,H_k$ 
    so we may evaluate the residual form 
    with this expression. 
Thus $\itres_z[\omega,\gamma]$ 
is a linear combination of forms 
(with coefficients depending only on $\Sigma$) 
\begin{equation} 
\frac{1}{v_j(m)}
\prod_{\substack{e \in \tau(1)\\H_e \not \in \{H_1,\ldots,H_k\}}}
        \left(m(e) + \frac{s_e}{2\pi i}\right)^{-1}\, dm, 
\end{equation} 
and 
\begin{equation} 
    |\!\itres_z[\omega,\gamma]| \ll 
    \frac{1}{|v_j(m)|}
    \sum_{\tau\in\Sigma_\infty(r_\infty)}
    \prod_{\substack{e \in \tau(1)\\H_e \not \in \{H_1,\ldots,H_k\}}}
        \left|m(e) + \frac{s_e}{2\pi i}\right|^{-1}.
\end{equation} 
This holds for any $j \in \{1,\ldots,r_w\}$ which proves 
    the claimed formula. 
\end{proof} 

\begin{proposition}\label{prop:upperbound3} 
    The $r$-form $\omega = \widehat{H_\infty}(\chi_m,-s)\,dm$ 
    admits an iterated residue expansion 
    with respect to any polyhedron 
    $\Pi$ with boundary $\cs$. 
    In particular, $\chi_m(n)^{-1}\omega$ 
    admits an iterated residue expansion 
    with respect to $\Pi$ if $\chi_m(n)^{-1}$ 
    is bounded on $\Pi$. 
\end{proposition} 

Note that $\chi_m(n)^{-1}$ 
    is bounded on $\Pi$
    if and only if $-\Lt(n) \in \Theta^\vee$ 
    where $\Theta = \mathrm{Im}(\Pi)$. 

\begin{proof} 
Let 
    $\gamma : \cs_\CC \supset H_1 \supset \cdots 
    \supset H_1 \cap \cdots \cap H_k$ $(1 \leq k < r)$ 
    be a $k$-step singular flag. 
Let $S$ denote the subset of generators $e \in \Sigma_w(1)$ 
    whose corresponding hyperplane $H_e$ is not in 
    $\{H_1,\ldots,H_k\}$. 
For any $\delta > 0$ let 
    $T_\delta \subset \gamma(k)$ 
    denote the open subset of elements $m\in\gamma(k)$ satisfying 
\begin{equation}\label{eqn:TDeltaDefn} 
    \left|e(m) + \frac{s_{e}}{2\pi i}\right| > \delta
\end{equation} 
    for all $e \in S$. 

We claim that 
    $$\itres_z[\omega,\gamma]=O_{\delta,s}(\norm{m}^{-2})$$ 
    as $m \in T_\delta \subset \gamma(k)$ tends to infinity. 
The convergence of the iterated residue expansion 
    follows immediately from this estimate. 
Let $\pi_k \colon \ccs_\CC \twoheadrightarrow 
    \ccs_\CC/\langle e_1,\ldots,e_k\rangle$ 
    be the natural quotient map. 
Suppose a cone $\tau$ in $\Sigma_\infty(r_\infty)$ 
    has a summand in $\widehat{H_\infty}$ which 
    contributes to $\itres_z[\omega,\gamma]$. 
The generators in $\tau(1)$ which are not in $S$ 
    are in the subspace $\langle e_1,\ldots,e_k\rangle$. 
Since $\tau$ is simplicial, the image of 
    $\tau(1)\cap S$ 
    under $\pi_k$ will span the quotient space 
    $\ccs_\CC/\langle e_1,\ldots,e_k\rangle$. 
Thus 
\begin{equation}\label{eqn:emsLowerBound} 
\max_{e \in \tau(1)\cap S}
\left|e(m) + \frac{s_e}{2\pi i}\right| 
\gg_s
\norm{m}
\end{equation} 
    for $m \in \gamma(k)$. 
With the help of 
    \eqref{eqn:TDeltaDefn} and \eqref{eqn:emsLowerBound} 
    we have 
\begin{equation}\label{eqn:boundOnTDelta} 
    \prod_{e \in \tau(1) \cap S}
    \left|e(m) + \frac{s_e}{2\pi i}\right|^{-1} 
    \ll_{\delta,s}
    \norm{m}^{-1} \qquad
    (m \in T_\delta \subset \gamma(k)). 
\end{equation} 
Now use Lemma~\ref{lemma:residualFormUpperBound}. 
It is easy to verify 
    $\chi_m(n)^{-1}$ is bounded on $\Pi$ 
    if and only if $-\Lt(n) \in \Theta^\vee$. 
\end{proof} 

\begin{remark} 
The following estimate 
    is proven in \cite[Prop.~2.3.2]{IMRN} 
    for fixed $s \in (\CC^{\Sigma(1)})^{\Gamma_w}$ 
    and $m \in \cs_w$ tending to infinity: 
\begin{equation}\label{eqn:CorrectEstimate} 
    \widehat{H_\infty}(\chi_m,-s)\ll_s
    \max_{\tau \in \Sigma_w(r_w)}
    \prod_{e \in \tau(1)}
    \left(\frac{1}{1 + |m(e)|}\right)^{1+1/r_w}.
\end{equation} 
In a later paper 
    \cite[Lemma~4.5]{batyrev1995maninsconjecturetoricvarieties} 
    this estimate is cited to justify the weaker estimate  
    $$
    \widehat{H_\infty}(\chi_m,-s)\ll_s
    \left(\frac{1}{1 + \norm{m}}\right)^{1+1/r_w}. 
    $$ 
Unlike \eqref{eqn:CorrectEstimate} 
    this does not suffice to prove that the integral 
    $\int_{\cs}\widehat{H_\infty}(\chi_m,-s)\,dm$ 
    converges, which is what is claimed in 
    \cite[Theorem~4.3]{batyrev1995maninsconjecturetoricvarieties} 
    and needed for the Poisson formula. 
In the subsequent paper 
    \cite[Prop.~4.12]{HZF}
    this is apparently corrected and \eqref{eqn:CorrectEstimate} 
    is now cited to justify the stronger estimate 
\begin{equation}\label{eqn:incorrectUpperBound} 
    \widehat{H_\infty}(\chi_m,-s)\ll_s
    \left(\frac{1}{1 + \norm{m}}\right)^{r_w+1} 
\end{equation} 
however this is false in general. 
Already for the projective plane over $\QQ$ 
    it is possible for $\widehat{H_\infty}(\chi_m,-s)$ 
    to be as large as $\gg_s |m|^{-2}$ 
    along certain lines in $\cs_w$: 
    the Fourier transform $\widehat{H_\infty}$ is given by 
\begin{equation}\label{eqn:HInfProjPlane} 
\widehat{H_\infty}(\chi_m,-s)
=
    \left(\frac{1}{2\pi i}\right)^2
    \frac{(s_0+s_1+s_2)/2\pi i}{(x+\frac{s_1}{2\pi i})(y+\frac{s_2}{2\pi i})(-x-y+\frac{s_0}{2\pi i})}
\end{equation} 
    which is $\gg \norm{m}^{-2}$ when $m$ tends to infinity 
    along $x+y=0$. 
\end{remark} 



\subsubsection{Periods} 

Let $\gamma$ be a singular flag. 
We will need to evaluate the following constants: 
\begin{equation} 
    \rho_\gamma =
    \lim_{\substack{s=(t,\ldots,t)\\ t\to \infty}}
    \sum_{\xi\in R^\vee}
    R_\gamma(s+2\pi im_\xi).
\end{equation} 
We will see that this equals 
    the product of the level $K$ regulator of $T$ 
    with a period obtained by integrating $R_\gamma$. 
Recall that $R_\gamma$ factors through the quotient 
$(\CC^{\Sigma(1)})^{\Gamma} \twoheadrightarrow  
    (\CC^{\Sigma(1)})^{\Gamma}/(\cs \otimes \CC )
    = (\Pic Y) \otimes \CC \dashrightarrow \CC$ 
to a rational function on the Picard group. 
In particular, it defines a rational function on 
the subspace $\cs_w/\cs \subset (\Pic Y) \otimes \CC$; 
this subspace is isomorphic to 
the dual $\cs_0$ of the trace-zero subspace. 

\begin{lemma}\label{lemma:rhoPeriod}
    The constant $\rho_\gamma$ is equal to 
\begin{equation} 
    |R_K|
    \int_{\ccs_0^\vee}
    R_\gamma(s_{AC}+2\pi im)\,dm
\end{equation} 
where $|R_K|$ is 
    the volume of the level $K$ regulator group of $T$. 
\end{lemma}

\begin{proof} 
Let $\ToricUnits$ denote 
    the free part of the finitely generated abelian group 
\begin{equation} 
    T(\QQ) \cap 
    \left(T(\RR) \times 
    \prod_{v\text{ finite}} K_v\right) .
\end{equation} 
The group $\ToricUnits$ has rank $r_w-r$ 
    \cite[Thm.~4, p.~285]{MR114817}. 
It is well-known that the natural map 
    $$\ToricUnits \to T(\RR) \xrightarrow{\lt{w}} \ccs_w$$ 
    embeds $\ToricUnits$ as a full-rank lattice of 
    the trace-zero subspace $\ccs_0$,  
    and the dual group $R^\vee$ 
    is canonically isomorphic to 
    the dual lattice to $\lt_w \ToricUnits$ in $\ccs_0^\vee$. 
By the definition of the Riemann integral, 
\begin{equation} 
    \lim_{\substack{t\to \infty}}
    \sum_{\xi\in R^\vee}
    R_\gamma(ts_{AC}+2\pi im_\xi)
    =
    \lim_{\substack{t\to \infty}}
    t^{-(r_w-r)}\sum_{\xi\in R^\vee}
    R_\gamma(s_{AC}+2\pi im_\xi/t)
    =
    V^{-1}
    \int_{\ccs_0^\vee}
    R_\gamma(s_{AC}+2\pi im)\,dm
\end{equation} 
where $V$ is the volume of a fundamental domain 
    for the dual lattice to $\lt_w \ToricUnits$ in $\ccs_0^\vee$. 
Since this is isomorphic to $R^\vee$ we have $V^{-1} = |R_K|$. 
\end{proof} 

\begin{example}[real quadratic Severi--Brauer surfaces]
The depicted fan 
    determines a complete smooth toric surface $Y$ over $\QQ$ 
    whose open torus $T=R^E_\QQ \GG_m$ 
    splits over a real quadratic extension $E$ 
    with Galois group $\Gamma = \langle \gamma \rangle$. 
    (This surface is isomorphic to $\PP_2$ over $\QQ$ 
    as a variety but not as a toric variety.) 
\begin{figure}[h!]\label{fig:quadraticSB}
    \begin{tikzpicture}[scale=2] 


\draw (0,0) -- (0,1);
\draw (0,0) -- (1,0);
\draw (0,0) -- (-1,-1);

\filldraw[black] (0,1) circle (.75pt) 
        node[anchor=west]{$e_2$};
\filldraw[black] (1,0) circle (.75pt) 
        node[anchor=west]{$e_1$};
\filldraw[black] (-1,-1) circle (.75pt) 
        node[anchor=east]{$e_{0}$};

\node (a) at (-1.5,-1) {};
\node (b) at (-1,-1.5) {};


\draw [<->] (b) to [bend left=40] 
    node [left, xshift=-.3em, yshift=-.3em] {$\gamma$} (a) ;



\fill[pattern={mydots}]
    (0,0) rectangle +(1.01,1.01);
\end{tikzpicture} 
    \caption{The fan for a quadratic Severi--Brauer surface.}
\end{figure}
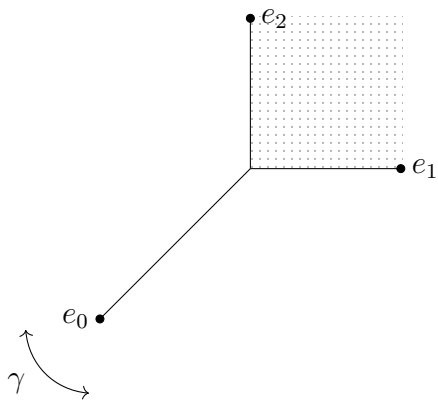
The boundary divisor $D$ is Cartier 
    and we use the canonical integral model for $Y-D$ 
    (Remark~\ref{rmk:CartierIntegrality}). 
The $D$-integral region for the divisor $D$ corresponding 
    to the ray $e_0$ is shaded in the figure. 
The $D$-integral rational points of $T$ are 
    $[1:a:\gamma a]$ for $a \in O_E \backslash \{0\}$. 
With the help of \eqref{eqn:HInfProjPlane}, 
    one finds the residual function for the flag $\gamma$ 
    determined by the hyperplane $H_0$ is 
    (with respect to the coordinate $z=-(x+y)$ on $\ccs$) 
\begin{align} 
R_\gamma(s)=
    (2\pi i)\itres_z[\widehat{H_\infty}(\chi_m,-s)\,dm,\gamma]
&=\frac{s_0+s_1+s_2}{(\frac{s_0}{2}+s_1)(\frac{s_0}{2}+s_2)}.
\end{align} 
The regulator group 
    $R\subset E^\times \backslash 
    \AA_E^{\times,1}/K$ 
    is isomorphic to 
    $\RR/(\ZZ \log u)$ where $u$ is 
    a positive generator of $O_E^\times$. 
(Indeed, if $E^\times x K$ 
is in $R$ then 
its image in the class group of $E$ is trivial, which means 
$x_{\mathrm{fin}}K_{\mathrm{fin}} 
\in E^\times K_{\mathrm{fin}}$. 
Thus we have an isomorphism 
    $R \to \RR/(\ZZ \log u):E^\times x K \mapsto \log y$ 
where we have translated by $E^\times K$ to ensure that 
$x = x_{\mathrm{fin}}x_\infty = (1,\ldots,1;y,1/y)$ 
and $y >0$.) 
A generating Hecke character is given by 
    $\xi(E^\times xK) = 
    |y|^{\frac{2\pi i}{\log u}}$. 
This character has local exponent 
$m_\xi = \frac{-1}{2\log u}(x-y)$. 
Thus for any $\Gamma$-invariant $s$ we have 
$s+2\pi im_\xi = \left(s_0,s_1-\frac{\pi i}{\log u},
s_1+\frac{\pi i}{\log u}\right)$. 
The constant $\rho_\gamma$ is 
\begin{align} 
    \rho_\gamma = 
    \lim_{s=(t,t,t)\to \infty}
      R_\gamma(s+2\pi im_\xi)
    = 
    \lim_{t\to \infty}
    \sum_{k \in \ZZ}
    \frac{4t}
    {t^2+\left(\frac{2\pi k}{\log u}\right)^2}
    =
    \lim_{t\to \infty}
    2 \log u  \coth \left(\tfrac 12t \log u 
   \right)=2\log u.
\end{align} 
Or using the lemma, we check that 
\begin{equation} 
    \rho_\gamma = 
    |R_K|
    \int_{-\infty}^\infty
    \frac{4}
    {1+\left({2\pi m}\right)^2}
    \,dm
    =2\log u.
\end{equation} 
\end{example} 


\subsubsection{Conical multiple Dirichlet series}\label{sec:MultipleDirSeries} 

Let $K_{\mathrm{fin}}$ be an arbitrary open compact subgroup  
    of $T(\AA_{\mathrm{fin}})$ and set 
    $K = K_{\mathrm{fin}} \times \MaximalCompact{\infty}$. 
Assume $K$ contains $\IntegralCompact{}$. 
Let $\Phi=\prod_{v \in \Places{\QQ}}\Phi_v \colon T(\AA) \to \CC$ be 
    a locally constant factorizable $K$-invariant function 
(e.g.~$\Phi=\xi \in (T(\QQ)\backslash T(\AA)/\IntegralCompact{})^\vee$ 
    may be an automorphic character of level $\IntegralCompact{}$). 
Let $\sigma\subset \ccs$ be a subset 
(e.g.~an open or closed cone). 
The \defn{multiple Dirichlet series of $\Phi$ 
over $D$-integral ideals in $\sigma$} is 
\begin{equation} 
    L(s,\Phi,\sigma,D) =
    \sum_{\substack{n \in \ToricIdeals\\\Lt(n) \in \sigma}}
    \Phi(n)
    n^{-s}. 
\end{equation} 
The \defn{multiple Dirichlet series of $\Phi$ over principal $D$-integral ideals in $\sigma$} is 
\begin{equation} 
    L_1(s,\Phi,\sigma,D) =
    \sum_{\substack{n \in \ToricPrinIdeals\\\Lt(n) \in \sigma}}
    \Phi(n)
    n^{-s}
    =
    \frac{1}{|C|}
    \sum_{\psi \in C^\vee}
    L(s,\Phi\psi,\sigma,D).
\end{equation} 

The difference between 
    these conical series and the multiple Dirichlet series 
    $\widehat{H}_{\mathrm{fin}}(\chi,-s,D)$ 
    from Corollary~\ref{cor:FiniteHeights} 
    is that these series impose \emph{global conditions}. 
Accordingly, $Z(s,\xi,\sigma,D)$ and $Z_1(s,\xi,\sigma,D)$ 
    do not admit Euler products in general 
    even though $\widehat{H}_{\mathrm{fin}}(\chi,-s,D)$ always does, 
    and this is where the failure of local-to-global properties 
    begins for the associated integral Diophantine problem. 
More precisely, the condition $n \in \ToricPrinIdeals$ 
    is a global condition (at least when $C \neq 1$), 
    and the condition $\Lt(n) \in \sigma$ 
    is a global condition 
    (when $\sigma$ does not contain $\Lt(\ToricIdeals)$). 

\begin{example}[Comparison with Hecke $L$-series for restriction tori]
Continue the context of Example~\ref{example:Restriction2} 
    with $T = R^E_\QQ$ and 
    $\IntegralCompact{} = \MaximalCompact{}$. 
The formula $n^{-s} = N(I)^{-s_1}$ shows that 
    if $\sigma$ is 
    the positive ($D_\infty$-integral) cone of $\ccs$, then 
    $$
    L(s_1,\xi) = L(s,\xi,\sigma,D_\infty)$$
    where $L(s_1,\xi)$ is the Hecke $L$-series of $\xi$. 
This conical series does admit an Euler product, 
    and this is because $\sigma$ contains $\Lt(\ToricIdeals)$. 
\end{example}


\subsection{Proof of the main theorem on height zeta functions}\label{sec:theMainTheorem} 

A flag in $Z$ is called \defn{$\Sigma$-convex} 
    if it is cut out by singular hyperplanes $H_1,\ldots,H_r$ 
    whose generators $e_1,\ldots,e_r$ form a $\Sigma$-convex set. 
For each cone $\sigma$ in $\Sigma_\Gamma(r)$ 
    let $$\Pi_{\sigma} = \cs + i\sigma^\vee 
    \subset \cs \otimes \CC,$$ 
    a polyhedron with boundary $\cs$ 
    and defining isomorphism determined by $\sigma(1)$, 
and let $Z_{\sigma}$ denote the subset of $\Sigma$-convex 
    singular flags in $Z$ 
    cut out by a $\Pi_\sigma$-stable 
    collection of singular hyperplanes. 
Let $\Sigma_D$ be the smallest subfan of $\Sigma_\Gamma$ 
    whose support $|\Sigma_D|$ contains 
    $-\Lt(\ToricIdeals)$. 
For each $\tau \in \Sigma_D$ choose 
    any maximal dimensional cone $\sigma_\tau\in\Sigma_D(r)$ 
    containing $\tau$. 
Let $\tau^\circ=
    \tau - \cup_{\tau' \subsetneq \tau} \tau'$ 
    denote the interior of a cone $\tau$. 

\begin{theorem}\label{thm:hzfformula} 
Assume that each $\Sigma$-convex singular flag is compatible 
    with the polyhedron $\Pi_\sigma$ 
    for every $\sigma \in \Sigma_D(r)$. 
Then the height zeta function 
    for $D$-integral rational points 
    on the torus 
\begin{equation} 
    Z(s)=\sum_{\substack{t \in T(\QQ)\\\textrm{\emph{$D$-integral}}}}
    \frac{1}{H(t,s)} 
\end{equation} 
is equal to 
\begin{align}
    Z(s) =
    \frac{|\ToricROU|}{|C||R|P}
    \sum_{\xi\in R^\vee}
    \sum_{\psi\in C^\vee}
    \sum_{\tau \in \Sigma_D}
    \sum_{\gamma \in Z_{\sigma_\tau}}
    L(s+2\pi im_{\gamma \xi},(\xi\psi)^{-1},-\tau^\circ,D)
    R_\gamma(s+2\pi im_\xi)
\end{align} 
where 
\begin{equation} 
    P=\sum_{\gamma \in Z_{\sigma_1}}
    \int_{\ccs_0^\vee}
    R_\gamma(s_{AC}+2\pi im)\,dm.  
\end{equation} 
in the region of $s\in(\CC^{\Sigma(1)})^\Gamma$ where 
    both sides converge. 
\end{theorem} 

\begin{remark}
As is well-known, the left-hand side converges absolutely in 
    a half-space of $(\CC^{\Sigma(1)})^\Gamma$ where 
    $\mathrm{Re}(s_e) \gg_{\Sigma,T} 0$ 
    for every $e \in \Sigma(1)$. 
Meanwhile the right-hand side converges absolutely 
    for any $s\in (\CC^{\Sigma(1)})^\Gamma$ for which 
\begin{equation}\label{eqn:RSumConvergence} 
    \sum_{\xi\in R^\vee}
        R_{\gamma}(s+2\pi im_\xi)
\end{equation} 
    is absolutely convergent 
    for every cone $\sigma \in \Sigma_D(r)$ 
    and singular flag $\gamma \in Z_\sigma$. 
\end{remark}

\begin{proof} 
The formula for $Z$ will be obtained by 
    applying the Poisson formula 
    to the subgroup $T(\QQ) \subset T(\AA)$ 
    and the function $x \mapsto H(x,-s,D)$, 
    then evaluating the integral 
    using the residue formula \cite[Theorem~1]{residue}. 
We must check that $H(x,-s,D)$ is in $L^1(T(\AA))$ 
    and the restriction of $\widehat{H}(\chi,-s,D)$ 
    is in $L^1(T(\QQ)^\perp)$. 
    The integral to be computed is 
\begin{equation}\label{eqn:hzf} 
    Z(s) = 
    \int_{(T(\QQ)\backslash T(\AA))^\vee} 
    \widehat{H}(\chi,-s,D)
    \,d\chi.
\end{equation} 

First we verify that 
    $x \mapsto f(x)=H(x,-s,D)$ is in $L^1(T(\AA))$. 
The chain of inequalities 
    $$\int_{T(\AA)} f(x) \,d\mu(x)
    =
    \lim_{C \text{ compact}}
    \int_{C} f(x) \,d\mu(x) 
    \leq 
    \lim_{S \text{ finite}}
    \int_{T(\AA_{S})} f(x) \,d\mu(x) 
    \leq 
    \int_{T(\AA)} f(x) \,d\mu(x)$$ 
    in the limits of larger $C$ and $S$ 
    shows that 
$$\int_{T(\AA)} f(x) \,d\mu(x)
=
    \lim_{S}
\int_{T(\AA_{S})} f(x) \,d\mu(x) 
    \leq
    \lim_{S}
    \prod_{\substack{v \in S}}
    \int_{T(\QQ_v)} f_v(x)\,d\mu_v(x).
$$ 
The convergence of the right-hand side 
follows from Proposition~\ref{prop:infiniteheightfouriertrfm} 
and Corollary~\ref{cor:FiniteHeights}. 

Next we check that 
    the restriction of $\widehat{H}(\chi,-s,D)$ to 
    $T(\QQ)^\perp\cong (T(\AA)/T(\QQ))^\vee$ is integrable. 
Let $K = \IntegralCompact{}$ denote the subgroup of 
    the maximal compact subgroup fixing $\charfn^D$. 
We have the exact sequence 
\begin{equation} 
    1 \longrightarrow C^\vee
    \longrightarrow (T(\mathbb Q)\backslash T(\AA)^1/K)^\vee
    \longrightarrow
    R^\vee
    \longrightarrow 1.
\end{equation} 
Since $R^\vee$ is a free abelian group 
    this sequence splits, 
    so any character of $R$ lifts 
    to a character of $T(\QQ)\backslash T(\AA)^1/K$. 
Furthermore, any character of 
$T(\QQ)\backslash T(\AA)^1/K$ 
lifts to a character of $T(\QQ)\backslash T(\AA)/K$ 
(see \eqref{eqn:trivialsplitting}). 
Thus, any character of $T(\QQ)\backslash T(\AA)^1/K$ 
may be factored uniquely as a product of three characters 
\begin{equation} 
    \xi \psi \chi
\end{equation} 
where (identifying characters with their lifts 
to $T(\QQ)\backslash T(\AA)/K$) 
$\xi \in R^\vee$, $\psi \in C^\vee$ and $\chi = \chi_m$ 
is the character given by 
$\chi_m(t) = e(\langle \Lt(t),m\rangle)$ 
for $m \in \cs$. 
Recall that $T(\QQ)\backslash T(\AA)^1$ is compact 
    \cite[Theorem 3.1.1]{ono} 
    so $(T(\QQ)\backslash T(\AA)^1)^\vee$ 
    is discrete and equipped with the counting measure, 
    so the subset $(T(\QQ)\backslash T(\AA)^1/K)^\vee$ is also 
    equipped with the counting measure. 
Since $\widehat{H}$ is supported on 
$(T(\QQ)\backslash T(\AA)/K)^\vee$, 
the height zeta function is proportional to 
\begin{equation}
    \sum_{\xi\in R^\vee}
    \sum_{\psi\in C^\vee}
    \int_{m\in \ccs^\vee} 
    \widehat{H}( \xi \psi \chi_m ,-s,D)
    \,dm
\end{equation} 
The infinite component 
    $\widehat{H_\infty}
    (\xi \psi \chi_m  ,-s,D)$ 
    is independent of $\psi$ by the last assertion of 
    Lemma~\ref{lemma:normonesubgroup}. 
Recall the notation 
$\ToricFracIdeals=
T(\AA_{\mathrm{fin}})/\IntegralCompact{\mathrm{fin}}$ 
of \S\ref{sec:ToricIdeals}, 
where for each $n \in \ToricFracIdeals$ we have set 
\begin{equation} 
    n^{-s}\coloneqq 
    H(n,-s)=
    \prod_{p \in \Places{\QQ}^{\mathrm{fin}}}
        p^{\frac{1}{e_p}\varphi(\lt_w(n_p))} 
        \in \CC^\times.
\end{equation} 
Expand $\widehat{H}_{\mathrm{fin}}$ of $\widehat{H}$ 
    as an absolutely convergent multiple Dirichlet series 
    (Corollary~\ref{cor:FiniteHeights}) 
    and then substitute it into $Z(s)$ to obtain 
\begin{align}
    Z(s)&=
    \kappa
    \sum_{\xi\in R^\vee}
    \sum_{\psi\in C^\vee}
    \int_{m \in \cs} 
    \sum_{n \in \ToricFracIdeals}
    (\xi\psi\chi_m)^{-1}(n)
    \charfn^D(n)
    n^{-s}
    \widehat{H_\infty}(\xi \chi_m,-s)
    \,dm\\
    &=
    \kappa
    \sum_{\xi\in R^\vee}
    \sum_{\psi\in C^\vee}
    \sum_{n \in \ToricIdeals}
    (\xi\psi)^{-1}(n)
    n^{-s}
    \int_{m \in \cs} 
    \chi_m(n)^{-1}
    \widehat{H_\infty}(\xi \chi_m,-s)
    \,dm
\end{align} 
for some yet to be determined positive constant $\kappa$. 

Now we evaluate the inner integral using 
    the residue formula from \cite{residue}. 
For each $n \in \ToricIdeals$ let $\sigma_n$ denote 
    any cone in $\Sigma_D(r)$ containing $-\Lt(n)$. 
By Proposition~\ref{prop:upperbound3} 
    the integral $\int_{m \in \cs} 
    \chi_m(n)^{-1}
    \widehat{H_\infty}(\xi \chi_m,-s)
    \,dm$ 
    admits an iterated residue expansion with respect to 
    the polyhedron $\Pi_{\sigma_n} =\cs + i\sigma_n^\vee$. 
Let $z=\sigma(1)$ be the dual basis to $\sigma(1)^\vee$. 
The residue formula is 
\begin{equation} 
    \int_{m \in \cs} 
    \chi_m(n)^{-1}
    \widehat{H_\infty}(\xi \chi_m,-s)
    \,dm
    =
    {(2\pi i)^r}
    \sum_{\gamma \in Z_{\sigma_n}}
    \itres_z\left[\chi_m(n)^{-1}
    \widehat{H_\infty}(\xi \chi_m,-s)
    \,dm,\gamma\right].
\end{equation} 
By definition of $R_\gamma$ we have that 
\begin{equation} 
    {(2\pi i)^r}
    \itres_z\left[\chi_m(n)^{-1}
    \widehat{H_\infty}(\chi_m,-s)
    \,dm,\gamma\right]
    =
    \chi_{m_\gamma}(n)^{-1} 
    R_\gamma(s)
\end{equation} 
for any $s \in (\CC^{\Sigma(1)})^{\Gamma_w}$ 
where $m_\gamma=m_\gamma(s)$ is the terminal point of $\gamma$. 
With our convention that $m$ is associated 
    to the element $(m(e))_e$ in $\CC^{\Sigma(1)}$ 
    one easily checks that 
    $\chi_{m}(n) = n^{2\pi im}$ 
    (Remark~\ref{rmk:temptingIdentification}) 
and 
$\widehat{H_\infty}(\xi \chi_m,-s)
=
    \widehat{H_\infty}(\chi_m,-(s+2\pi im_\xi))$ 
for any $\xi \in R^\vee$ 
where $m_\xi \in \cs_0$ is determined by the identity 
    $\xi(t_\infty) = e(\langle m_\xi, \lt_w(t_\infty) \rangle)$ 
    for $t_\infty \in T(\RR) \subset T(\AA)$. 
Write 
    $m_{\gamma\xi} \coloneqq m_{\gamma}(s+2\pi im_\xi) 
    \in \cs_\CC$. 
Then 
\begin{equation} 
    \int_{m \in \cs} 
    \chi_m(n)^{-1}
    \widehat{H_\infty}(\xi \chi_m,-s)
    \,dm
    =
    \sum_{\gamma \in Z_{\sigma_n}}
    \chi_{m_{\gamma\xi}}(n)^{-1} 
    R_\gamma(s+2\pi im_\xi)
    =
    \sum_{\gamma \in Z_{\sigma_n}}
    n^{-2\pi im_{\gamma\xi}} 
    R_\gamma(s+2\pi im_\xi).
\end{equation} 
Putting this back into the earlier expression for $Z$ obtains 
\begin{align}\label{eqn:hzf2} 
    Z(s)=
    \kappa
    \sum_{\xi\in R^\vee}
    \sum_{\psi\in C^\vee}
    \sum_{n \in \ToricIdeals}
    (\xi\psi)^{-1}(n)
    n^{-s}
    \sum_{\gamma \in Z_{\sigma_n}}
    n^{-2\pi im_{\gamma\xi}} 
    R_\gamma(s+2\pi im_\xi).
\end{align} 
Collecting $n \in \ToricIdeals$ 
    into summands according to whichever cone 
    $\tau \in \Sigma_D$ 
    contains $-\Lt(n)$ rearranges to the claimed formula 
    up to the constant $\kappa$. 

To determine $\kappa$, specialize $s = (t,\ldots,t)=ts_{AC}$ 
    and let $t \to +\infty$. 
By the defining series for $Z(s)$, we see it approaches 
    the number of $D$-integral rational points of $T$ 
    with height one. 
This is also the order of 
    the torsion subgroup $\mu$ of 
\begin{equation} 
    T(\QQ) \cap 
    \left(T(\RR) \times 
    \prod_{v\text{ finite}} \IntegralCompact{v}\right). 
\end{equation} 
On the other hand 
$Z(s)$ approaches 
\begin{align}
    \kappa
    \sum_{\xi\in R^\vee}
    \sum_{\psi\in C^\vee}
    \sum_{\substack{n \in \ToricIdeals\\ n^{-s} = 1}}
    (\xi\psi)^{-1}(n)
    \sum_{\gamma \in Z_{\sigma_1}}
    R_\gamma(ts_{AC}+2\pi im_\xi)
\end{align} 
The subset of $n \in \ToricIdeals$ satisfying $n^{-s} = 1$ 
    for all $s$ is 
    $\MaximalCompact{\mathrm{fin}}/\IntegralCompact{\mathrm{fin}}$. 
Recall that $(\xi\psi)^{-1}(n)$ 
    is evaluated using the maps 
\begin{equation} 
\MaximalCompact{\mathrm{fin}}/\IntegralCompact{\mathrm{fin}}
\hookrightarrow T(\QQ)\backslash T(\AA)/\IntegralCompact{} 
\xrightarrow{\sim} T(\QQ)\backslash T(\AA)^1/\IntegralCompact{} \times \cs 
\xrightarrow{\sim} R \times C \times \cs.
\end{equation} 
The identity component of 
$T(\QQ)\backslash T(\AA)/\IntegralCompact{}$ 
is 
$T(\QQ)\backslash (T(\QQ)T(\RR)\IntegralCompact{})/\IntegralCompact{}$ 
    (see the proof of Lemma~\ref{lemma:normonesubgroup}) 
    which shows that 
$$
(\MaximalCompact{\mathrm{fin}}/\IntegralCompact{\mathrm{fin}})
\cap (R \times 1 \times \cs) = 1.
$$ 
Thus 
$(\MaximalCompact{\mathrm{fin}}/\IntegralCompact{\mathrm{fin}}) \to C$ 
is injective, 
so by orthogonality of characters $Z(s)$ approaches 
\begin{align}
    \kappa
    |C|
    \sum_{\xi\in R^\vee}
    \sum_{\gamma \in Z_{\sigma_1}}
    R_\gamma(ts_{AC}+2\pi im_\xi).
\end{align} 
With the help of Lemma~\ref{lemma:rhoPeriod} this approaches 
\begin{equation} 
    \kappa|C|
    \sum_{\gamma \in Z_{\sigma_1}}
    \rho_\gamma
    =
    \kappa|C|
    |R_{\IntegralCompact{}}|
    \sum_{\gamma \in Z_{\sigma_1}}
    \int_{\ccs_0^\vee}
    R_\gamma(s_{AC}+2\pi im)\,dm = |\ToricROU|.
\end{equation} 


Now assume \eqref{eqn:RSumConvergence} 
    is absolutely convergent 
    for every cone $\sigma \in \Sigma_D(r)$ 
    and singular flag $\gamma \in Z_\sigma$. 
Then \eqref{eqn:hzf2} is absolutely convergent, 
    i.e.~$\widehat{H}(\chi,-s,D) 
    \in L^1((T(\QQ)\backslash T(\AA))^\vee)$. 
It follows from the Poisson formula (cf.~\S\ref{sec:Poisson}) that 
\begin{equation}\label{eqn:poissonAE} 
    \sum_{x\in T(\QQ)} H(xy,-s,D)
    =
    \int_{(T(\QQ)\backslash T(\AA))^\vee} 
    \widehat{H}(\chi,-s,D)
    \chi(y)
    \,d\chi
\end{equation} 
holds for almost every $y \in T(\AA)$. 
Since 
    $\widehat{H}(\chi,-s,D) 
    \in L^1((T(\QQ)\backslash T(\AA))^\vee)$ 
    the right-hand side is continuous in $y$. 
The equality thus holds for every $y$ 
    if $x \mapsto H(xy,-s,D)$ is in $L^1(T(\QQ))$ 
    for any $y \in T(\AA)$, 
    as this implies continuity of the left-hand side. 
Since $y$ only affects the height at finitely many places 
    in a manner independent of $x$, 
    it suffices to consider $y=1$. 
The left-hand side can be bounded from above 
    by the same sum where $s$ is replaced by 
    a $\ZZ$-linear combination of toric divisors 
    representing a $T$-Cartier divisor 
    and satisfying $\mathrm{Re}(s_e) > 0$ for all $e \in \Sigma(1)$. 
Such an $s$ corresponds to 
    a big $T$-linearized line bundle $L$. 
Since $L$ is big, the number of rational points 
    of $L$-height $\leq H$ not in 
    the base locus of $L$ 
    is bounded by a polynomial in $H$. 
The base locus of $L$ is disjoint from $T$, 
    and this proves the left-hand side is absolutely convergent in 
    a half-space of $(\CC^{\Sigma(1)})^\Gamma$ where 
    $\mathrm{Re}(s_e) \gg_{\Sigma,T} 0$ 
    for every $e \in \Sigma(1)$. 
In particular, the formula holds when $y = 1$. 
\end{proof} 

\begin{remark} 
The proof shows 
    the {twisted} Poisson formula 
    \eqref{eqn:poissonAE} is valid 
    for any parameter $y \in T(\AA)$ 
    whenever both sides converge. 
This can be useful for applications \cite{CF}. 
\end{remark} 

%




\section{An orbit parametrization for polynomials with given Galois group}\label{sec:orbitparam} 

Let $G$ be a finite group. 
Consider the projective space of lines 
\begin{equation} 
\PP(\AA^1 \oplus \AA^n) 
    = \mathrm{Proj}\, \ZZ[Y,X_1,\ldots,X_n]
\end{equation} 
in the $G$-representation formed 
    from the trivial representation 
    $\AA^1 = \Spec \ZZ[Y]$ 
    and the regular representation 
    $\AA^n =\Spec \ZZ[X_1,\ldots,X_n]$ 
    where $X_1,\ldots,X_n$ is the dual basis to 
    $G = \{g_1=e,\ldots,g_n\}$. 
In this section we show that the quotient $\PP(\AA^1 \oplus \AA^n)/G$ 
    can be used to parametrize polynomials with Galois group $G$. 
This is a modification of the construction from \cite{moge} 
    where a trace condition was imposed. 
We will later see that $\PP(\AA^1 \oplus \AA^n)/G$ 
    is a toric variety when $G$ is abelian. 


\subsection{Normal bases for torsors of finite groups} 

Let $G$ act by $S$-automorphisms on a finite $S$-scheme 
$f\colon T \to S$. 
If $(f_\ast \mcO_T)^G = \mcO_S$ and the induced action of $G$ 
on the set $T(k)$ is free for any $S$-field $k$, 
then we say \defn{$T/S$ is Galois with Galois group $G$}, 
or that \defn{$T/S$ is a $G$-torsor}. 
A $G$-stable ordered list $(x_1,\ldots,x_n)$ of regular functions 
on a $G$-torsor $T/S$ is called a \defn{normal basis} 
if it forms a basis of $\mcO(T)$ as an $\mcO(S)$-module 
and is isomorphic as a $G$-set to $G$ itself with its left regular action. 
Let $\lb$ be a $G$-line bundle over $T$. 
Let $x \mapsto {^g}x$ 
denote the action of $G$ on sections given by 
$^g x(t) = g(x(g^{-1}t))$ for $t \in T$, 
where the outer $g$ on the right-hand side 
denotes the action of $g$ on the fibers 
$({g^{-1}t})^\ast \lb \to t^\ast\lb$ 
coming from the structure of $\lb$ as a $G$-line bundle. 

We will equip $G$-torsors $T/S$ with extra data: 
\begin{enumerate} 
    \item a $G$-line bundle $\lb$ over $T$, and 
    \item global sections $y$ and $x$ of $\lb$ 
\end{enumerate} 
subject to three conditions: 
\begin{enumerate} 
    \item $Gx \cup \{y\}$ generates $\lb$, 
    \item ${^gx}(t) \neq {^hx}(t)$ if $g \neq h$ and $t \in T$, and 
    \item $y$ is $G$-invariant. 
\end{enumerate} 
We regard data $(T/S,\lb,x,y)$ and $(T'/S,\lb',x',y')$ 
    as equivalent 
    if there is a $G$-equivariant isomorphism 
    $\phi \colon T \to T'$ over $S$ 
    and an isomorphism 
    $\psi \colon \phi^\ast \lb' \to \lb$ 
    of $G$-line bundles 
    such that 
    $(\psi\phi^\ast x',\psi\phi^\ast y')=(x,y)$. 
We call $(T/S,\lb,x,y)$ \defn{monic} 
    if $y$ is nowhere vanishing. 
If $(T/S,\lb,x,y)$ is monic then 
by rescaling $x$ and $y$ we may take $\lb = \mcO(T)$, 
and by rescaling $x$ by $\mcO(T)^\times$ 
we may take $y = 1$; 
thus monic points may be written more compactly as $(T/S,x)$. 
Let $A \subset \AA^n$ denote the complement of 
the union of the hyperplane divisors $\{X_i/Y = X_j/Y\}$ 
for all $1 \leq i < j \leq n$. 
Let $\G \subset A$ denote the open subscheme 
where the group determinant 
$\Delta_G = \det(X_{gh}/Y)_{g,h\in G}$ 
is invertible. 
The group algebra structure on $\AA^n$ 
gives $\G$ the structure of a group scheme. 

\begin{proposition}\label{prop:orbitParam} 
The $S$-points of the affine variety $A/G$ 
    are in bijection with 
    equivalence classes of $(T/S,\lb,x,y)$, 
    functorially in $S$. 
Any $S$-point of $\G/G$ is monic. 
Let $(T/S,x)$ be a monic $S$-point of $A/G$ 
    and write $(x_1,\ldots,x_n)$ 
    for the regular functions 
    $({^{g_1}x},\ldots,{^{g_n}x})$ on $T$. 
Then  
    $(T/S,x)$ is contained in $\G/G$ 
if and only if 
    $(x_1(t),\ldots,x_n(t))$ is a normal basis 
    for $\mcO(T) \otimes k(f(t))$ over $k(f(t))$ 
    for every $t \in T$. 
\end{proposition} 

\begin{proof} 
First observe that $A$ is contained 
    in the open locus of $\PP$ where $G$ has no isotropy, 
    and therefore $G$ acts freely on $A(k)$ 
    for any $S$-field $k$. 
It follows that the scheme $A/G$ represents the stack theoretic 
    quotient $[A/G]$. 
This means that the $S$-points of $A/G$ are in bijection with 
    equivalence classes of pairs $(T/S,\phi)$ 
    where $T/S$ is a $G$-torsor 
    and $\phi$ is a $G$-equivariant morphism $\phi \colon T \to A$. 
Thus $\phi$ is determined by 
    the line bundle $\lb = \phi^\ast \mathcal O(1)|_A$ 
    and the global sections 
    $y = \phi^\ast Y, 
    x=x_1=\phi^\ast X_1,\ldots,x_n=\phi^\ast X_n$. 
Since $\phi$ is $G$-equivariant, these global sections 
    satisfy the required conditions: 
    they clearly generate $\lb$, and 
    if ${^gx}(t) = {^hx}(t)$ 
    then $\phi(t) \in \{X_g/Y = X_h/Y\}$ which implies $g = h$. 
Meanwhile $\lb$ has a canonical $G$-line bundle structure 
    inherited from $\mathcal O(1)$ 
    for which $y$ is invariant since 
    $gy = g\phi^\ast Y = \phi^\ast gY = \phi^\ast Y = y$. 
In the other direction, from a datum $(T/S,\lb,x,y)$, 
we get a morphism $\phi \colon T \to A$ from $y$ 
    and the sections $x_g \coloneqq {^gx}$ 
    since $Gx \cup \{y\}$ is generating. 
This morphism will be $G$-equivariant since 
\begin{align} 
    \phi(gt) 
    = [y(gt):(x_h(gt))_h] 
    &= [g({^{g^{-1}}y})(t):(g({^{g^{-1}}x_h})(t))_h] \\
    &= [gy(t):(gx_{g^{-1}h}(t))_h] \\
    &= g\phi(t) . 
\end{align} 

An $S$-point of $\G/G$ is automatically monic since 
    $\G$ is contained in the open locus of $\PP$ 
    where $Y$ is invertible. 
A monic $S$-point $(T/S,x)$ will be contained in $\G/G$ 
    if and only if $\phi$ factors through $\G$, 
    which occurs if and only if $\phi^\ast \Delta_G$ 
    is invertible on $T$. 
    If $t \in T$, then 
    $\phi^\ast \Delta_G(t) = \det(x_{gh}(t))_{g,h}$ 
    is invertible if and only if $(x_1(t),\ldots,x_n(t))$ 
    forms a basis for $\mcO(T) \otimes k(f(t))$ over $k(f(t))$. 
\end{proof} 

%



\subsection{Characteristic maps for torsors of finite groups}\label{sec:Characteristic} 

    The \emph{root height} of a polynomial 
$f=t^n+a_1t^{n-1}+\cdots+a_n \in \ZZ[t]$ is 
$$H(f) = \max(|a_1|,|a_2|^{1/2},\ldots,|a_n|^{1/n}).$$ 
We now explain how to express the root height 
    in terms of a Weil height on $\PP/G$ in the usual sense, 
    i.e.~the height function obtained 
    by pulling back the standard height on projective space 
    via some morphism from $\PP/G$ to projective space. 
To do this, we consider some rational maps from $\PP/G$ 
    to projective space 
    built from the characteristic polynomial. 
The coefficients of the characteristic polynomial 
    will be valued in a certain line bundle on $\PP/G$. 
Write $\PP = \PP(\AA^1 \oplus \AA^n)$ 
and let $\pi \colon \PP \to \PP/G$ denote the quotient morphism. 
Let $l = \lcm(1,\ldots,n)$. 
Let $e_k(X) \in \mcO(k)$ denote 
    the degree $k$ elementary symmetric polynomial. 

\begin{lemma}
The coherent sheaf $\dlb \coloneqq (\pi_\ast \mcO(l))^G$ is an ample line bundle on $\PP/G$. 
The $l$th power of the root height 
    is a Weil height on $\PP/G$ for 
    the linear system of $\dlb$ 
    spanned by the generating sections 
    $Y^l$ and $e_k^{l/k}\in H^0(\PP/G,\dlb)$ for $1 \leq k \leq n$. 
\end{lemma}

\begin{proof} 
By \cite[Prop.~3.2]{moge} 
    the $G$-line bundle $\mcO(l)$ on $\PP$ 
    descends to a line bundle on $\PP/G$ 
    (passing to the $l$th power of $\mcO(1)$ 
    trivializes any action of isotropy subgroups 
    on fibers over fixed points). 
Sections of this descended line bundle 
    on an open subset $U \subset \PP/G$ 
    are canonically identified 
    via pullback along $\pi$ 
    with $G$-invariant sections of $\mcO(l)$ 
    on $\pi^{-1}(U)$. 
Note $\dlb$ is ample 
    since $\pi$ is finite and 
    its pullback $\pi^\ast\dlb = \mcO(l)$ is ample. 

Consider the projective morphism 
    $\chp \colon \PP/G \to \PP_n$ 
    defined by the generating sections 
    $Y^l$ and $e_k^{l/k}\in H^0(\PP/G,\dlb)$ for $1 \leq k \leq n$. 
On monic $S$-points of $A/G$ it is given by 
\begin{equation}
    \chp(T/S,\mcO_T,x,y)=
    [y^l:e_1(Gx)^l:e_2(Gx)^{l/2}:\cdots:e_n(Gx)^{l/n}].
\end{equation}
These sections are globally generating for $\dlb$ since 
    $e_1(X) = e_2(X) = \cdots =e_n(X) = 0$ if and only if 
    $X = 0$, so the base locus of this linear system is empty. 
\end{proof} 



%

\begin{remark}[coefficient height] 
    By contrast, the coefficient height 
    $\max(|a_1|,|a_2|,\ldots,|a_n|)$ 
    is not equivalent to a power of a Weil height on $\PP/G$. 
The proof of the lemma 
    also shows that $(\pi_\ast \mcO(n))^G$ is a line bundle. 
There is a closely related map with coordinates 
    $Y^{n-k}e_k \in H^0(\PP/G,(\pi_\ast \mcO(n))^G)$ 
    for $0 \leq k \leq n$. 
On monic $S$-points of $A/G$ this map is given by 
\begin{equation}
    (T/S,\mcO_T,x,y) \dashrightarrow
    [y^n:y^{n-1}e_1(Gx):y^{n-2}e_2(Gx):\cdots:e_n(Gx)].
\end{equation}
The ``problem'' with this map 
    is that it has a base locus $B=\{Y = e_n(X) = 0\}$ 
    which is not $\G$-stable. 
The line bundle associated to this height 
    is not on $\PP/G$ 
    but on the blowup of $\PP/G$ along $B$, 
and when $G$ is abelian the blowup 
    no longer has a canonical toric variety structure 
    and the methods of this paper are not applicable. 
The ``correct'' compactification of $A/G$ 
    to consider the coefficient height is 
    $(\PP_1 \times \cdots \times \PP_1)/ G$ 
    where $G$ permutes the $G$ copies of the projective line, 
    however this alternative quotient 
    does not have a natural toric structure. 
\end{remark} 

\begin{remark}[anticanonical height] 
Let $D = \{Y=0\}$. 
Then $\dlb = \mathcal O(lD)$. 
The divisor $D$ is toric ($T$-stable)  
    and corresponds to the ray $e_0$ in the fan of $Y$. 
Let $s_D\in (\ZZ^{\Sigma(1)})^\Gamma$ be 
    the corresponding element 
    ($1$ in the $e_0$ coordinate and $0$ elsewhere). 
One computes that $s_{AC} \equiv (n+1) s_D \pmod {\cl}$, 
    so the root-height is equivalent to 
    $H_{AC}^{\frac{1}{n+1}}$. 
\end{remark}

\subsubsection{Fibers of the characteristic map} 

Any $G$-torsor over $\QQ$ has the form 
$T=\Spec K \sqcup \cdots \sqcup \Spec K$ 
for some field $K$ with Galois group $D \leq G$. 
We call $D$ the monodromy group of $T$. 

\begin{lemma}
Let $D \leq G$ and let $(K/\QQ,x)$ 
    be a rational point of $\G/G$ with monodromy group $D$. 
The number of elements of $(\G/G)(\QQ)$ 
    which have the same characteristic polynomial 
    as $(K/\QQ,x)$ is equal to $[N_{S_n}(D):N_{G}(D)]$. 
\end{lemma}

\begin{proof} 
Suppose two rational points $(K/\QQ,x)$, $(K'/\QQ,x')$ 
    have the same characteristic polynomial. 
Since $G$ acts transitively on the roots, 
    we may assume that $x = x'$. 
Now we are in the situation that $K$ and $K'$ 
    are isomorphic as $\QQ$-algebras  
    but not necessarily as $G$-algebras, 
    since an isomorphism of $G$-algebras 
    must be $G$-equivariant. 
A $G$-equivariant isomorphism 
    $K\to K'$ that fixes $x$ 
    must be the identity, 
    so we are computing the number of non-isomorphic $G$-algebra 
    structures on a given $G$-algebra with monodromy group $D$. 
A $G$-algebra (ring of functions on a $G$-torsor) 
    with monodromy group $D$ 
    is the same thing as a cocycle of $H^1(\QQ,G)$ 
    whose image is $G$-conjugate to $D \leq G$, 
    while 
    an \'etale $\QQ$-algebra of degree $n$ 
    with monodromy group $D$ 
    is the same thing as a cocycle of $H^1(\QQ,S_n)$ 
    whose image is $S_n$-conjugate to $D$. 
We must compute the number of cocycles in $H^1(\QQ,G)$ 
    lying over a given cocycle $[K/\QQ] \in H^1(\QQ,S_n)$ 
    whose image is $S_n$-conjugate to $D$ 
    under the natural homomorphism 
    $H^1(\QQ,G) \to H^1(\QQ,S_n)$. 
This is equal to the given quantity. 
\end{proof} 

One can check that $N_{S_n}(C_n) = n \varphi(n)$ 
which proves the following corollary. 

\begin{corollary}\label{cor:cyclicTwists}
Suppose $G = C_n$. 
The number of rational points of $T = \G/C_n$ 
    with the same irreducible characteristic polynomial 
    is $\varphi(n)=|\mathrm{Out}(C_n)|$ 
    where $\varphi$ is Euler's totient function. 
\end{corollary}










\section{The toric structure on \texorpdfstring{$\PP/G$}{P/G}}\label{sec:DiophProblem} 

Assume $G$ is abelian. 
Let $\G$ denote the group of units in 
the group algebra of $G$, 
an algebraic torus over $\QQ$ of rank $n$. 
The group $G$ is a constant subgroup of $\G$ 
and $T=\G/G$ is again an algebraic torus of rank $n$. 
The quotient variety $\PP/G$ inherits a $T$-action 
from the $\G$-action on $\QQ G$ 
making it into a (normal) toric variety. 
We now describe its fan. 

Let $G^\vee = \Hom(G,\CC^\times)$ denote 
    the set of characters of $G$. 
Let $E$ be the field over $\QQ$ 
    generated by the values of every character $\chi \in G^\vee$. 
The elementary idempotents $v_\chi$ of $\CC G$ 
    are in bijection with 
    characters $\chi\in G^\vee$ where 
    $v_\chi = \frac{1}{n}\sum_{g \in G}\chi(g^{-1}) g$. 
We define 
\begin{equation} 
\ccl_E' =  \ZZ \langle v_\chi : \chi \in G^\vee \rangle \subset \CC G,
\qquad
\ccs_E = \ccl_E' \otimes \RR.
\end{equation} 
Let $v_0 = -\sum_\chi v_\chi$. 
Let $\Sigma$ be the set of cones $\sigma_I$ 
where $I \subset \{v_0\} \cup \{v_\chi :\chi \in G^\vee\}$ is any proper subset 
and $\sigma_I \subset \ccs_E$ is the convex span of $I$. 
For a root of unity $\zeta$, let 
    $\ell(\zeta) \in [0,1) \cap \QQ$ 
    be the unique rational number satisfying 
    $e(\ell(\zeta)) = \zeta$. 

\begin{proposition}\label{prop:quotienttoruscocharacterlattice}
Let $G = \langle c_1 \rangle \times \cdots \times \langle c_K\rangle$ 
    be a decomposition of $G$ into cyclic groups. 
Define the elements 
\begin{equation}
    \omega_{j} = 
    \sum_{\chi \in \langle c_j \rangle^\vee}
    \ell(\chi(c^{-1}_j))v_\chi 
    \in \ccl_E' \otimes \QQ
\end{equation}
and the lattice 
    $$\ccl_E = \ccl_E' + \sum_{j=1}^K\ZZ \omega_{j}.$$ 
The toric variety over $\QQ$ associated to 
    the fan $\Sigma$ and 
    the lattice $\ccl_E$ 
    with its canonical Galois action by 
    $\Gamma = \Gal(E/\QQ)$ 
    is the quotient variety $\PP/G$, 
    and its open torus is $T$. 
\end{proposition}

\begin{proof} 
Observe that the group $G$ acts by left translation on 
    $\ccl'_E\otimes \CC$, 
and the dual action on $\cl'_E\otimes\CC$ 
is given by 
\begin{equation}
g\cdot v_\chi^\vee = \chi(g^{-1}) v_\chi^\vee
\end{equation}
where $\{v_\chi^\vee\}$ is the dual basis 
    to $\{v_\chi\}$. 
This induces a $G$-action on 
$\CC[\cl_E'] = 
    \CC[\{x_\chi^{\pm 1}\}]$ 
    where $x_\chi = \chi^{v_\chi^\vee}$. 

Now observe that the quotient group 
    $\ccl_E/\ccl_E'$ acts on $\CC[\cl_E']$ by 
\begin{equation}
v \cdot \chi^{u} = 
    e({\langle u,v\rangle})
\chi^u
\end{equation}
where 
$\langle\cdot, \cdot \rangle \colon 
    \cl_E'/(\ccl_E^\vee) \times \ccl_E/\ccl_E' \to \QQ/\ZZ$ 
is the canonical perfect pairing. 
We have the short exact sequence 
\begin{equation}
0 \longrightarrow \ccl_E' \longrightarrow \ccl_E 
    \longrightarrow G \longrightarrow 0
\end{equation}
where the quotient map takes 
    $\omega_{C_i}$ to the fixed generator of $C_i$. 
    For any character $\chi$, 
\begin{equation}
\omega_C \cdot x_\chi = 
e(\langle v_\chi^\vee,\omega_C\rangle)x_\chi =
    e(\ell(\chi(c^{-1})))x_\chi = 
    \chi(c^{-1})x_\chi = 
    c \cdot x_\chi.
\end{equation}
Thus the isomorphism $\ccl_E /\ccl_E'= G$ is compatible 
    with the action on $\G_\CC = \Spec \CC [\cl_E']$. 
As a $G$-representation, 
$\CC[\cl_E']$ is a direct sum over the weight vectors $\chi^u$ 
for all $u \in \cl_E'$ 
    which implies that 
    the subspace $\CC[\cl_E']^G$ of $G$-invariants 
    is a direct sum over the weight vectors $\chi^u$ 
    with trivial $G$-action. 
In view of the isomorphism $\ccl_E /\ccl_E'= G$, 
    we see $\chi^u$ is $G$-invariant if and only if 
    $u \in \ccl_E^\vee$. 
Thus $\CC[\cl_E']^G = \CC[\ccl_E^\vee]$, 
    and more generally 
    $$\CC[\sigma^\vee \cap \cl_E']^G 
    = \CC[\sigma^\vee \cap \ccl_E^\vee]$$ 
    for any cone $\sigma \in \Sigma$. 
The toric variety defined by $\Sigma$ and $\ccl_E'$ 
    is $\PP$, which is covered by 
    the set of affine toric varieties 
    $U_\sigma=\Spec \CC[\sigma^\vee \cap \cl_E']$ 
    for $\sigma\in\Sigma$. 
Each $U_\sigma$ is $\G$-stable hence also $G$-stable, 
so 
\begin{equation}
    (\PP/G)_\CC
    =
\bigcup_{\sigma \in \Sigma}
U_\sigma/G
=
\bigcup_{\sigma \in \Sigma}
    \Spec \CC[\sigma^\vee \cap \ccl_E^\vee].
\end{equation}
\end{proof} 

%





\section{Monic cyclic polynomials of bounded height} 

In this section we prove Theorem~\ref{thm:mainthm} 
    by applying Theorem~\ref{thm:hzfformula} 
    to an integral Diophantine problem 
    for the toric variety constructed in 
    \S\ref{sec:DiophProblem}. 
To define the integrality condition, 
    we use the divisor $D = \{Y = 0\} \subset \PP/G$. 
This is a $\QQ$-Cartier toric divisor 
    which is not Cartier, 
    however the multiple $lD$ is Cartier 
    where $l = \lcm(1,\ldots,n)$ 
    and corresponds to 
    the line bundle $\dlb = (\pi_\ast\mcO(l))^G$ 
    from \S\ref{sec:Characteristic}. 
In fact $nD$ is already Cartier, 
    but we use $lD$ so that 
    the coefficients $Y^\ell,e_k^{\ell/k} \in H^0(Y,\dlb)$ 
    of the characteristic polynomial 
    are generators of the line bundle. 
The quotient $\PP/G = \mathrm{Proj}\, \ZZ[Y,X_1,\ldots,X_n]^G$ 
    is defined over $\ZZ$, as is 
    the boundary divisor $D = \{Y = 0\}$, 
    so for our integral model we simply take 
    $U$ to be the complement of $D$. 
For this integral model, the indicator function 
    on $D$-integral points in $T(\AA)$ 
    $$\charfn^D = \prod_{v\in \Places{\QQ}} \charfn^{D}_v$$ 
    is invariant under the maximal compact subgroup 
    $\MaximalCompact{}$ of $T(\AA)$. 
The preimage of $U$ in $\PP$ is 
    $\PP-D =\AA^n = \Spec \ZZ[X_1/Y,\ldots,X_n/Y]$. 
The group algebra structure on $\AA^n$ is defined over $\ZZ$, 
    so the action of the unit group $\G$ on $\PP-D$ 
    and therefore the action of the quotient group 
    $T=\G/G$ on $U = \AA^n/G$ 
    are defined over $\ZZ$. 
For finite places $v$ 
    the maximal compact subgroup is 
    $\MaximalCompact{v}= T(\ZZ_v) \subset T(\QQ_v)$, 
    which shows that if 
    $t \in U(\ZZ_v)$ and $k \in \MaximalCompact{v}$ 
    then $tk \in U(\ZZ_v)$. 
Thus $\IntegralCompact{v} = \MaximalCompact{v}$. 

\begin{remark}
For the integral Diophantine problem in \cite{CUBIC}, 
    the indicator function on $D$-integral points 
    was not invariant under the maximal compact 
    subgroup, cf.~\cite[Prop.~4]{CUBIC}. 
This occured because of the trace-one condition 
    that was imposed on polynomials. 
This condition corresponds to splitting off 
    the factor of $\GG_m$ from the unit group 
    $\G = \GG_m \times \G_1$ 
    corresponding to the trivial representation of $C_3$, 
    and using the torus $T_1 = \G_1 / C_3$ 
    instead of $\G/C_3$. 
However the splitting $\G = \GG_m \times \G_1$ 
    is only defined over $\ZZ[\tfrac13]$ 
    (the trivial idempotent is $\tfrac13(1+\sigma+\sigma^2)$), 
    so the action of $T_1$ on the integral model 
    is not defined over $\ZZ$ 
    and the argument given here does not apply. 
\end{remark}

\subsection{Describing the subfan $\Sigma_D$}\label{sec:SubfanDescrip} 

We describe the smallest subfan $\Sigma_D \subset \Sigma$ 
    whose support $|\Sigma_D|$ contains $-\Lt(\ToricIdeals)$. 
The ray generators in $\Sigma_\Gamma$ correspond to 
Galois orbits in 
$\Sigma(1) = \{e_0\} \sqcup \{e_\chi\}_{\chi \in G^\vee}$ 
(Lemma~\ref{lemma:subfan}). 
The Galois action on $\Sigma(1)$ is identified 
with the cyclotomic action on $\mathrm{triv} \sqcup G^\vee$ 
and the $\widehat{\ZZ}^\times$-orbits of $G^\vee$ 
are naturally in bijection with the cyclic subgroups of $G^\vee$. 
The cones in $\Sigma_{\Gamma}$ containing $e_0$ 
    are in bijection with 
    proper subsets of the set $G^\vee/\widehat{\ZZ}^\times$ 
    of $\widehat{\ZZ}^\times$-orbits on $G^\vee$. 
Let $O_1,\ldots,O_r$ be the Galois orbits in 
    $\{e_\chi\}_{\chi \in G^\vee}$. 
We may assume that $O_1 = \{e_1\}$. 
Set $f_i = \sum_{e \in O_i} e$. 
Then $(f_1,\ldots,f_r)$ is a basis of $\ccs$. 
Let 
\begin{equation}\label{eqn:LambdaHyperplanes} 
    \Lambda_j = 
    -\RR_{\geq 0}\langle e_0,f_i : i \neq j\rangle 
    =\{b_1 f_1 + \cdots + b_r f_r : 
    \text{$b_j \geq 0$ and $b_j \geq b_i$ for all $i$}
    \}
    \subset \ccs. 
\end{equation} 
Any element of $\Lt(\ToricIdeals)$ may be expressed as 
    $b_1f_1+\cdots + b_r f_r$ for $b_i \geq 0$, 
    and it is contained in each region $\Lambda_j$ 
    where $j$ is an index for which $b_j$ is largest. 

\subsection{Compatibility of the polyhedra}\label{sec:VerifyingCompat} 

Here we verify that the residue method is applicable for the fan $\Sigma$ 
    to make use of Theorem~\ref{thm:hzfformula}. 
Let $\Pi$ be a polyhedron 
    and let $H$ be 
    a collection of $r$ affine hyperplanes in $\cs_\CC$ 
    defined by equations $f_1-is_1,\ldots,f_r-is_r$ 
    for real linear maps $f_1,\ldots,f_r \colon \cs \to \RR$ 
    and constants $s_1,\ldots,s_r$ 
    with positive real parts. 
The associated Jacobian is 
\begin{equation} 
    J(H,\Pi)
    =
    \begin{bmatrix}
        f_{1}(v_{1})&f_{1}(v_{2})&\cdots&f_1(v_{r})\\
        f_{2}(v_{1})&f_{2}(v_{2})&\cdots&f_2(v_{r})\\
        \vdots&\vdots&\ddots&\vdots\\
        f_{r}(v_{1})&f_{r}(v_{2})&\cdots&f_{r}(v_{r})\\
    \end{bmatrix}
\end{equation} 
where $(v_1,\ldots,v_r)$ is the defining basis of $\Pi$. 
We verify that every 
    every collection $H_1,\ldots,H_r$ 
    of linearly independent singular hyperplanes 
    of $\omega=\widehat{H_\infty}(\chi,-s)\,dm$ 
    is compatible with the polyhedron $\Pi_\sigma$ 
    for every $\sigma \in \Sigma_D(r)$. 
Such a collection corresponds to 
    a choice $E = \{e_1,\ldots,e_r\}$ 
    of $r$ forms in $\Sigma_w(1)$ 
    whose restrictions $\pi E$ to $\cs$ 
    are linearly independent. 

The maximal cones in $\Sigma_D(r)$ 
    are $\sigma_{(j)} = -\Lambda_j$ for $1 \leq j \leq r$, 
    and the polyhedra to be integrated along are of the form 
    $\Pi_j = \cs + i \sigma_{(j)}^\vee$ 
    for $1 \leq j \leq r$. 
Any proper subset of $\Sigma(1)$ is $\Sigma$-convex, 
    so any such collection $E$ will project 
    to the set $\pi E = \sigma(1)$ 
    of generators for some 
    $\sigma = \sigma_{(j)} \in \Sigma_\Gamma(r)$ 
    (Lemma~\ref{lemma:subfan}). 
One can check that the Jacobian $J(\sigma,\Pi_j)$ 
is compatible for every $\sigma \in \Sigma_\Gamma(r)$ and 
$1 \leq j \leq r$. 
They also satisfy the conjecture of \cite[Remark~3]{residue}. 
The Jacobians $J(\sigma,\Pi_j)$ follow the following pattern, 
    e.g. for $r = 4$: 
\begin{equation} 
J(\sigma_{(1)},\Pi_1)=
\begin{bmatrix}
1 & 0 & 0 & 0 \\
0 & 1 & 0 & 0 \\
0 & 0 & 1 & 0 \\
0 & 0 & 0 & 1
\end{bmatrix}
\qquad
J(\sigma_{(1)},\Pi_2)=
\begin{bmatrix}
1 & 0 & 0 & 0 \\
-1 & -1 & -1 & -1 \\
0 & 0 & 1 & 0 \\
0 & 0 & 0 & 1
\end{bmatrix}
\end{equation} 
\begin{equation} 
J(\sigma_{(1)},\Pi_3)=
\begin{bmatrix}
1 & 0 & 0 & 0 \\
0 & 0 & 1 & 0 \\
-1 & -1 & -1 & -1 \\
0 & 0 & 0 & 1
\end{bmatrix}
\qquad
J(\sigma_{(1)},\Pi_4)=
\begin{bmatrix}
1 & 0 & 0 & 0 \\
0 & 0 & 1 & 0 \\
0 & 0 & 0 & 1 \\
-1 & -1 & -1 & -1
\end{bmatrix}
\end{equation} 
\begin{equation} 
J(\sigma_{(2)},\Pi_1)=
\begin{bmatrix}
1 & 0 & 0 & 0 \\
-1 & -1 & -1 & -1 \\
0 & 0 & 1 & 0 \\
0 & 0 & 0 & 1
\end{bmatrix}
\qquad
J(\sigma_{(2)},\Pi_2)=
\begin{bmatrix}
1 & 0 & 0 & 0 \\
0 & 1 & 0 & 0 \\
0 & 0 & 1 & 0 \\
0 & 0 & 0 & 1
\end{bmatrix}
\end{equation} 
\begin{equation} 
J(\sigma_{(2)},\Pi_3)=
\begin{bmatrix}
1 & 0 & 0 & 0 \\
0 & 1 & 0 & 0 \\
-1 & -1 & -1 & -1 \\
0 & 0 & 0 & 1
\end{bmatrix}
\qquad
J(\sigma_{(2)},\Pi_4)=
\begin{bmatrix}
1 & 0 & 0 & 0 \\
0 & 1 & 0 & 0 \\
0 & 0 & 0 & 1 \\
-1 & -1 & -1 & -1
\end{bmatrix}
\end{equation} 
\begin{equation} 
J(\sigma_{(\infty)},\Pi_1)=
\begin{bmatrix}
-1 & -1 & -1 & -1 \\
0 & 1 & 0 & 0 \\
0 & 0 & 1 & 0 \\
0 & 0 & 0 & 1
\end{bmatrix}
\qquad
J(\sigma_{(\infty)},\Pi_2)=
\begin{bmatrix}
0 & 1 & 0 & 0 \\
-1 & -1 & -1 & -1 \\
0 & 0 & 1 & 0 \\
0 & 0 & 0 & 1
\end{bmatrix}
\end{equation} 
\begin{equation} 
J(\sigma_{(\infty)},\Pi_3)=
\begin{bmatrix}
0 & 1 & 0 & 0 \\
0 & 0 & 1 & 0 \\
-1 & -1 & -1 & -1 \\
0 & 0 & 0 & 1
\end{bmatrix}
\qquad
J(\sigma_{(\infty)},\Pi_4)=
\begin{bmatrix}
0 & 1 & 0 & 0 \\
0 & 0 & 1 & 0 \\
0 & 0 & 0 & 1 \\
-1 & -1 & -1 & -1
\end{bmatrix}.
\end{equation} 


\subsection{Finding the poles} 

Having covered $-\Lt(\ToricIdeals)$ with rational cones in $\Sigma_D$, 
    we still need to compute $n^{-(s-s_\gamma)}$ 
    for each $\gamma \in Z_\ast$ in order 
    to determine the poles of the multiple Dirichlet series 
    appearing in the formula for $Z(s)$. 

\begin{proposition}\label{prop:HelpingForMainThm} 
Let $s \in (\CC^{\Sigma(1)})^\Gamma$ be any element 
    with corresponding $\Sigma$-linear map 
    $\varphi\colon \ccs_E \to \CC$. 
The flags in $Z_{-\Lambda_j}$ have a common terminal point 
    which is the unique linear map $m_{-\Lambda_j} \in \cs_\CC$ 
    agreeing with $(2\pi i)^{-1}\varphi$ on $-\Lambda_j$. 
Each $\Gamma$-invariant $\Sigma$-linear function 
    on $\ccs_E$ is determined by its values on 
    $e_0,f_1,\ldots,f_r$. 
The function $\varphi-2\pi im_{-\Lambda_j}$ 
    (corresponding to $s - s_\gamma$) 
    takes zero on each of these elements  
    except for $f_j$ 
    where it has the value 
    $-(s_0+\cdots+s_n)$. 
\end{proposition} 

\begin{proof} 
Since $\varphi-2\pi im_{-\Lambda_j}$ 
    vanishes identically on $-\Lambda_j$, 
    it follows that $\varphi-2\pi im_{-\Lambda_j}$ 
    is zero on $e_0$ and every $f_k$ with $k \neq j$. 
Meanwhile 
\begin{equation} 
2\pi im_{-\Lambda_j}(f_j) = 
    2\pi im_{-\Lambda_j}\left(-e_0-\sum_{k \neq j} f_k\right) 
    =-\varphi(e_0) - \sum_{\substack{e \not \in O_j\\ e \neq e_0}} \varphi(e)
\end{equation} 
and 
$\varphi(f_j) = 
    \sum_{e \in O_j}
    \varphi(e)$ 
by $\Sigma$-linearity and 
    convexity of the Galois action on $\Sigma$. 
Thus 
\begin{equation} 
(\varphi-2\pi im_{-\Lambda_j})(f_j)=
    \varphi(f_j)-2\pi im_{-\Lambda_i}(f_j)
= -(s_0+\cdots+s_n).
\end{equation} 
\end{proof} 


Recall the formula \eqref{eqn:HFin} for the finite part of 
the Fourier transform of the height function: 
\begin{equation}
    \widehat{H}_{\mathrm{fin}}(\chi,-s,D) 
    =
    \prod_{p \in \Places{\QQ}^{\mathrm{fin}}}
    \sum_{n \in T(\QQ_p)/\IntegralCompact{p}}
    \chi_p(n)^{-1} 
    \charfn^D_v(n)
    p^{\frac{1}{e_p}\varphi(\lt_w(n))} 
    =
    \sum_{n \in \ToricIdeals}
    \chi(n)^{-1}
    n^{-s}.
\end{equation} 
To deduce the location of the right-most pole of $Z(s)$ 
    from Theorem~\ref{thm:hzfformula}, 
    we first need a more precise description of 
    the Euler factors of $\widehat{H}_{\mathrm{fin}}$. 
For regular fans this is accomplished
    by \cite[Prop.~2.2.3]{IMRN}, 
    but for our application to polynomials 
    we must consider simplicial fans. 
When a cone in the fan is not regular, 
    the local Euler factors are modified 
    by an additional factor $Q_\Delta$ 
    which generally contributes a pole 
    to the Dirichlet series. 
However, we will show such poles occur to the left 
    of the right-most pole. 
Let 
    $\Delta = \{\sum_{\chi \in G^\vee} a_\chi e_\chi 
    : 0\leq a_\chi <1 \}$. 
For each Galois orbit $O_j \subset \Sigma(1) - \{e_\infty\}$, 
    let $T_j=R^{E_j}_\QQ \GG_m$ be 
    the torus over $\QQ$ determined by 
    the Galois permutation action on $O_j$. 
There are natural homomorphisms $T_j \to T$ 
    of algebraic tori induced by the Galois equivariant 
    lattice maps $\ZZ O_j \to \ccl_E$. 

\begin{proposition}\label{prop:localEulerFactor} 
Let $p$ be a finite place and 
    let $w$ be a place of $E$ lying over $p$. 
Let $s \in (\CC_+^{\Sigma(1)})^\Gamma$ 
    and let $\chi$ be a character of $T(\QQ_p)$ 
    which is trivial on $\IntegralCompact{p}$. 
Let $O_1^w,\ldots,O_{r_w}^w$ be the $\Gamma_w$-orbits in 
    $\{e_\chi\}_{\chi \in G^\vee}$. 
Then 
\begin{equation} 
\widehat{H_p}(\chi,-s,D)
    =
    Q_p(\chi,s)
    \prod_{j=1}^{r_w}
    \left(1-\chi(O_j^w)^{-1}p^{\frac{1}{e_p}\varphi(O_j^w)}\right)^{-1}
    =
    Q_p(\chi,s)
    \prod_{j=1}^{r}
    L_{p}(\chi_j,s_{O_j})
\end{equation} 
where $\varphi(O_j^w) = \sum_{e \in O_j^w} \varphi(e)$, 
    $\chi(O_j^w) = \prod_{e \in O_j^w} \chi(e)$, 
    $s_{O_j} = -\varphi(e)$ for any $e \in O_j$, 
    $\chi_j$ is the pullback of $\chi$ along 
    $T_j(\QQ_p) \to T(\QQ_p)$, 
    $L_p$ is the Euler factor at $p$ of 
    the Hecke $L$-function of $\chi_j$, 
    and 
\begin{equation} 
    Q_{p}(\chi,s)
    =
    \sum_{n \in \im \lt_w \cap \Delta }
    \chi(n)^{-1}p^{\frac{1}{e_p}\varphi(n)} . 
\end{equation} 
In particular, 
    the multiple Dirichlet series 
    $\widehat{H}_{\mathrm{fin}}(\chi,-s,D)$ 
    for an automorphic character $\chi$ of level 
    $\IntegralCompact{}$ is equal to 
\begin{equation}\label{eqn:HFin2} 
    Q(\chi,s) 
    \prod_{j=1}^r
    L(\chi_j,s_{O_j})
\end{equation} 
where $Q(\chi,s) = \prod_{p \in \Places{\QQ}^{\mathrm{fin}}}
    Q_p(\chi,s)$. 
Set $o_k = |O_k|$. 
The multiple Dirichlet series $Q(\chi,s)$ 
    converges absolutely on the half-space 
    $\{\frac12\mathrm{Re}(o_2s_{O_2}+\cdots+o_rs_{O_r})>1\} 
    \subset (\CC_+^{\Sigma(1)})^\Gamma$, 
and the multiple Dirichlet series 
    $\widehat{H}_{\mathrm{fin}}(\chi,-s,D)$ 
    converges absolutely on the polyhedral region 
\begin{equation} 
\{\tfrac12\mathrm{Re}(o_2s_{O_2}+\cdots+o_rs_{O_r})>1\}
    \cap \{\mathrm{Re}(s_{O_2})>1\} \cap \cdots \cap 
    \{\mathrm{Re}(s_{O_r}) > 1\}.
\end{equation} 
If $\chi = \xi \psi \chi_m$ and $\xi \psi \neq 1$ then 
    the multiple Dirichlet series $Q(\chi,s)$ 
    converges conditionally on the half-space 
    $\{\mathrm{Re}(o_2s_{O_2}+\cdots+o_rs_{O_r})>0\} 
    \subset (\CC_+^{\Sigma(1)})^\Gamma$, 
and the multiple Dirichlet series 
    $\widehat{H}_{\mathrm{fin}}(\chi,-s,D)$ 
    converges conditionally on the polyhedral region 
\begin{equation} 
    \{\mathrm{Re}(s_{O_2})>0\} \cap \cdots \cap 
    \{\mathrm{Re}(s_{O_r}) > 0\}.
\end{equation} 
Finally, there is a constant $c > 1$ such that 
for any $u \in \CC$ with positive real part 
\begin{equation}\label{eqn:Qp} 
    Q_{p}(1,us_{AC}) = 
    \begin{cases}
        1 + \phi(n) p^{-\frac{n-1}{2}u}
        + O(p^{-\frac{c(n-1)}{2}u})
            &\text{if $p$ split in $E/\QQ$,}\\
        1 + O(p^{-\frac{c(n-1)}{2e_p}u})
        & \text{otherwise}\\
    \end{cases}
\end{equation} 
where $\phi$ is the Euler totient function, 
$\omega= \frac 1n(e_2+2e_3+\cdots+(n-1)e_n)$, 
and the implied constant is independent of $p$. 
If $G$ has prime order, then this simplifies to 
\begin{equation} 
    Q_{p}(1,s) = 
    \begin{cases}
        1 + (n-1) p^{-\frac{n-1}{2}u}
            &\text{if $p$ split in $E/\QQ$,}\\
        1& \text{otherwise}.
    \end{cases}
\end{equation} 

\end{proposition} 

\begin{proof} 
    The first formula is obtained by applying 
    a known formula for the exponential generating function 
    of integer points in polyhedra \cite[p.~110]{zbMATH05322358}. 
The series $Q(\chi,s)$ is dominated by 
\begin{equation} 
    \prod_p
    \sum_{n \in \ccl_E \cap \Delta }
    p^{\frac{1}{e_p}\varphi(n)} . 
\end{equation} 
The convergence of this series is determined by 
    those nonzero elements of $\ccl_E \cap \Delta$ 
    for which $\mathrm{Re}(\varphi(n))$ is smallest; 
these minima occur when $n$ is in the Galois orbit of $\omega$. 
We have that 
\begin{align} 
    \pi \omega 
    &= \frac{1}{|\Gamma|}
        \sum_{\chi \in G^\vee} \ell(\chi(c^{-1})) 
        \sum_\gamma
        \gamma e_\chi \\
    &=
    \sum_{j=2}^b \frac{1}{o_j}
        f_j
        \sum_{e_\chi \in O_j} 
        \ell(\chi(c^{-1})) \\
    &=
    \sum_{j=2}^b \frac{1}{o_j}
        f_j
        \frac{o_j}{2} 
        \qquad (n \geq 3)\\
    &= \frac12(f_2+\cdots+f_r).
\end{align} 
Since the Galois action on $\Sigma$ is convex 
    we know that $\varphi(\omega') = \varphi(\pi \omega)$ 
    for any $\omega'$ in the Galois orbit of $\omega$, 
    and thus 
    $\varphi(\omega') = \frac12(\varphi(O_2)+\cdots+
    \varphi(O_r)) = -\frac12(o_2s_{O_2}+\cdots+o_rs_{O_r})$. 
The claimed regions of convergence for $Q$ and 
    $\widehat{H}_{\mathrm{fin}}$ now follow from 
    the well-known regions of convergence for Hecke $L$-functions. 

Now we prove the last claim. 
The formula for $Q_p$ follows from summing over 
    $n\in N_E \cap \Delta$ ordered by 
    $\mathrm{Re}(\varphi(n))$. 
The minimal value of 
    $\frac12\mathrm{Re}(o_2s_{O_2}+\cdots+o_rs_{O_r})$ 
    is $\frac{n-1}{2}\mathrm{Re}(u)$ 
    which is achieved on the Galois orbit of $\omega$ 
    which has size $[E:\QQ]=\phi(n)$. 
The set $\im \lt_w \cap \Delta$ 
    is contained in $\ccl_E \cap \Delta$ 
    which has $[\ccl_E: \ccl_E'] = n$ elements 
    where $\ccl_E' = \ZZ \langle e_1,\ldots,e_n\rangle$. 
If $G$ has prime order then the orbit of $\omega$ 
    under $\Gamma=\mathrm{Gal}(\QQ(\zeta_n)/\QQ)$ 
    has size $n-1$. 
Thus if $p$ is split then $\im \lt_w \cap \Delta$ 
    entirely consists of $0 \in \ccl_E$ 
    and the Galois orbit of $\omega$. 
As $\omega$ is moved by any nontrivial element of $\Gamma$, 
    the set $\im \lt_w \cap \Delta$ is 
    $\{0\}$ if $p$ is not split. 
\end{proof} 



\subsection{Parametrizing toric ideals}\label{sec:ParamToricId} 

We will use the above description of $\widehat{H_p}$ 
    to describe toric ideals in terms of ideals 
    of number fields. 
To ease the analysis, we will restrict attention to 
    those toric ideals which are needed to 
    determine the asymptotic rate of growth 
    of the sum of Dirichlet coefficients 
    of the height zeta function $Z(s)$. 
Let $Sp_E$ denote the set of square-free 
    positive integers which are only divisible 
    by primes that are split in $E/\QQ$. 
By the description of the Euler factor \eqref{eqn:Qp},  
    we see that primes which are not split in $E/\QQ$ 
    may be ignored in the sense that 
\begin{equation} 
    \frac{\prod_{p\in Sp_E}(1+\phi(n) p^{-\frac{n-1}{2}u})}{Q(1,us_{AC})}
\end{equation} 
is regular and nonvanishing at $u=(\frac{n-1}{2})^{-1}$. 
We set 
\begin{equation} 
    Q_1(u) =
    \prod_{p\in Sp_E}(1+\phi(n) p^{-\frac{n-1}{2}u})
    =
    \sum_{M \in Sp_E}
        \phi(n)^{\omega(M)}
        M^{-\frac12(n-1)u}
\end{equation} 
where $\omega(M)$ is the number of distinct prime factors of $M$. 
Let $\ToricIdeals_1 \subset \ToricIdeals$ 
    denote the subset of toric ideals in bijection 
    with $Sp_E \times \mathrm{Ideals}(O_{E_1}) 
    \times \cdots \times \mathrm{Ideals}(O_{E_r})$. 
We will regard elements of $\ToricIdeals_1$ 
    as weighted by the factor $\phi(n)^{\omega(M)}$. 
For any $n \in \ToricIdeals_1$ we have that 
\begin{equation}\label{eqn:nFormula} 
    n^{-s}=
    M^{-\frac12(o_2s_{O_2}+\cdots+o_rs_{O_r})}
    N(I_1)^{-s_{O_1}}
    \cdots
    N(I_r)^{-s_{O_r}}
\end{equation} 
and 
\begin{equation}\label{eqn:HfinToricIdealsExplicit} 
    \sum_{n \in \ToricIdeals_1}
        n^{-s}
    =
    \sum_{M,I_1,\ldots,I_r}
        \phi(n)^{\omega(M)}
    M^{-\frac12(o_2s_{O_2}+\cdots+o_rs_{O_r})}
    N(I_1)^{-s_{O_1}}
    \cdots
    N(I_r)^{-s_{O_r}}
\end{equation} 
where $I_1\times \cdots \times I_r$ 
    is an ideal of $O_{E_1} \times \cdots \times O_{E_r}$. 
(The formula for $n^{-s}$ in terms of norms of ideals 
is explained at the end of Example~\ref{example:Restriction2}.) 


\subsection{Proof of the main theorem on polynomials} 

We begin by restricting $Z(s)$ to the complex line 
    spanned by the anticanonical line bundle 
    $[-K] \in \Pic(Y)$. 
Let $s_{AC} \in (\ZZ^{\Sigma(1)})^\Gamma$ 
    denote the element with all coordinates equal to $1$. 
It is well-known that $s_{AC}$ maps to $[-K]$ 
    under the canonical map 
    $(\ZZ^{\Sigma(1)})^\Gamma \to \Pic(Y)$. 
Let $s = u s_{AC}$ for $u \in \CC$ 
    with sufficiently large real part. 
We write $Z(us_{AC}) = \sum_N a_N N^{-(n+1)u}$ 
    and let $A(X) = \sum_{N < X^{\frac{1}{n+1}}} a_N$. 

By the Poisson formula, 
\begin{equation} 
    Z(us_{AC}) = 
    \frac{|\ToricROU|}{|C||R|P}
    \sum_{\xi \in R^\vee} 
    \sum_{\psi \in C^\vee}
    \int_{\cs}
    \widehat{H}_{\mathrm{fin}}(\xi\psi\chi_m,-s,D)
    \widehat{H}_{\infty}(\xi\chi_m ,-s)\, dm
    =\sum_{\xi,\psi}
    Z_{\xi\psi}(u)
\end{equation} 
where 
\begin{equation} 
    Z_{\xi\psi}(u) = 
    \frac{|\ToricROU|}{|C||R|P}
    \int_{\cs}
    \widehat{H}_{\mathrm{fin}}(\xi\psi\chi_m,-s,D)
    \widehat{H}_{\infty}(\xi\chi_m ,-s)\, dm.
\end{equation} 
We evaluate $Z_{\xi\psi}$ 
    by distributing the integral over $\cs$ 
    into the sum over toric ideals 
    and then applying the residue formula 
    (Theorem~\ref{thm:ResidueFormula}). 
(The hypotheses for the residue formula 
    were verified in \S\ref{sec:VerifyingCompat}.) 
This results in the $\xi \psi$-summand 
    of the formula for $Z(s)$ 
    in Theorem~\ref{thm:hzfformula}, which is  
\begin{equation}\label{eqn:ZxipsiPoisson} 
    Z_{\xi\psi}(u)=
    \frac{|\ToricROU|}{|C||R|P}
    \sum_{\tau \in \Sigma_D}
    \sum_{\gamma \in Z_{\sigma_\tau}}
    L(s+2\pi im_{\gamma \xi},(\xi\psi)^{-1},-\tau^\circ,D)
    R_\gamma(us_{AC}+2\pi im_\xi).
\end{equation} 

To determine the main term of $A(X)$ 
    it suffices to consider maximal cones 
    $\sigma$ in the sum over cones $\tau \in \Sigma_D$. 
Since the difference between a cone $\sigma$ 
    and $\sigma^\circ$ is formed from lower-dimensional cones, 
    to compute the main term we may 
    use either $\sigma$ or $\sigma^\circ$. 
So we turn to the Dirichlet series 
    $L(s+2\pi im_{\gamma \xi},
        (\xi\psi)^{-1},-\sigma,D)$ 
    where $\sigma$ is a maximal cone of $\Sigma_D$. 
The quantity $s+2\pi im_{\gamma \xi}$ may be evaluated 
    with the help of Lemma~\ref{lemma:VariationOfPoles} 
    which shows that 
    $2\pi im_{\gamma\xi}=\mu_{E_{\gamma}}(\varphi|_\tau)
    -2\pi i \mu_{E_\gamma} m_\xi$ 
    where $E_\gamma \subset \Sigma_w(1)$ 
    is any defining set of generators for $\gamma$. 
Since the Galois action is convex we have 
    $\mu_{E_{\gamma}}(\varphi|_\tau) = \varphi|_{\pi E_{\gamma}}$ 
    where $\varphi|_{\pi E_{\gamma}}\in \cs\otimes \CC$ 
    is the unique linear map agreeing with $\varphi$ 
    on the basis $\pi E_{\gamma}$. 
Since the Galois action is strongly convex, 
    $\varphi|_{\pi E_{\gamma}}=\varphi|_{\sigma}$ 
    where $\sigma=\langle\pi E_{\gamma}\rangle$ 
    is a maximal cone of $\Sigma_D$. 
Let $s_\gamma \in (\CC^{\Sigma(1)})^\Gamma$ 
    correspond to $\varphi|_{\sigma}$ 
    so $s+2\pi im_\gamma = s-s_\gamma$. 
We conclude that $s+2\pi im_{\gamma \xi} 
= s-s_\sigma - 2\pi i \mu_{E_\gamma} m_\xi$ 
where $s_\sigma=s_\gamma$ 
only depends on $\sigma$. 
    
The maximal cones of $\Sigma_D$ were described in 
    \S\ref{sec:SubfanDescrip}: 
    they have the form 
    $\sigma_{(j)} = -\Lambda_j$ for $1 \leq j \leq r$ 
    where $\Lambda_j$ consists of 
    those elements $b_1f_1+\cdots+b_rf_r \in \ccs$ 
    with $b_i \geq 0$ and $b_j$ maximal. 
We may deduce the main term up to a constant from 
    the part of 
    $L(s-s_\sigma-2\pi i \mu_{E_\gamma}m_\xi,
        (\xi\psi)^{-1},-\sigma,D)$ 
    summed over the subset 
    $\ToricIdeals_1 \subset \ToricIdeals$. 
Namely, if $c$ is the constant of the main term 
    of the summatory function of 
    the Dirichlet series 
    $L(s-s_\sigma-2\pi i \mu_{E_\gamma}m_\xi,
        (\xi\psi)^{-1},-\sigma,D)$ 
    summed over the subset 
    $\ToricIdeals_1 \subset \ToricIdeals$, 
    then the constant of the main term 
    for the summatory function of 
    $L(s-s_\sigma-2\pi i \mu_{E_\gamma}m_\xi,
        (\xi\psi)^{-1},-\sigma,D)$ 
    without restricting to 
    $\ToricIdeals_1 \subset \ToricIdeals$ 
    may be obtained by multiplying $c$ by 
    the value of $Q(1,us_{AC})/Q_1(u)$ 
    at the abscissa of convergence 
    (cf.~\S\ref{sec:ParamToricId}). 
We saw in \S\ref{sec:ParamToricId} 
    that each toric ideal $n\in \ToricIdeals_1$ 
    corresponds to a tuple $(M,I_1,\ldots,I_r)$ 
    where $M \in Sp_E$ and 
    $I_1\times \cdots \times I_r$ 
    is an ideal of $O_{E_1} \times \cdots \times O_{E_r}$. 
The formulas from \S\ref{sec:ParamToricId} 
    imply that 
\begin{equation}\label{eqn:LnExpression} 
    \Lt(n)=
    f_1 \log N(I_1)
    +\frac{1}{o_2}
    f_2 \left( \log N(I_2) + \frac12 o_2\log M \right)
    +\cdots +\frac{1}{o_r}
    f_r \left(\log N(I_r) + \frac12 o_r\log M \right). 
\end{equation} 

First we consider $\Lambda_1$. 
From \eqref{eqn:LnExpression} we see that 
$\Lt( n) \in \Lambda_1$ if and only if 
\begin{equation} 
    N(I_k) \sqrt{M}^{o_k}\leq N(I_1)^{o_k}
    \quad
    \text{for all $k \neq 1$}. 
\end{equation} 
Similarly $\Lt(n) \in \Lambda_\ell$ for $2 \leq \ell \leq r$ 
if and only if 
\begin{equation} 
    N(I_1)^{o_\ell} \leq N(I_\ell)\sqrt{M}^{o_\ell},\quad
    N(I_k)^{\frac{o_\ell}{o_k}} \leq N(I_\ell)
    \quad
    \text{for all $k \neq 1$}. 
\end{equation} 
We sum over tuples $(M,I_1\ldots,I_r)$ 
    satisfying one of the above systems of inequalities 
    to determine the Dirichlet coefficients. 
First note that 
   $$L(s-s_\sigma-2\pi i \mu_{E_\gamma}m_\xi,
        (\xi\psi)^{-1},-\sigma,D) 
    =L(s-s_\sigma,
        (\xi\psi\chi)^{-1},-\sigma,D)$$ 
        where $\chi(n) = n^{-2\pi i \mu_{E_\gamma}m_\xi}$ 
        is a Hecke character of level $\IntegralCompact{}$. 
For the maximal cone $\sigma=-\Lambda_1$ 
we use \eqref{eqn:nFormula} and 
Proposition~\ref{prop:HelpingForMainThm} 
to see that 
\begin{equation}\label{eqn:nForMaximalCone1} 
    n^{-(s-s_{\sigma})} = 
        N^{-t}, 
    \quad N = N(I_1),
    \quad t=s_0+\cdots+s_n=(n+1)u.
\end{equation} 
From the inequalities defining $\Lambda_1$ we find that 
the Dirichlet coefficient of $N^{-t}$ is 
\begin{equation}\label{eqn:LeadingDirCoeffSummed} 
    \sum_{M=1}^{N^2}
        1_{Sp_E}(M)
        \phi(n)^{\omega(M)}
    \prod_{k=2}^r
    \left(
    \sum_{\substack{I_k\subset O_{E_k}\\1 \leq N(I_k)\leq(N/\sqrt{M})^{o_k}}}
        (\xi\psi\chi)_k(I_k)^{-1}
    \right).
\end{equation} 
Let $\delta_{(\xi\psi)_k} = 1$ if $(\xi\psi)_k = 1$ 
and otherwise $0$. 
We also set $\beta_k\coloneqq 1-\frac{2}{1+o_k}$. 
By known zero-free regions of Hecke $L$-functions this is 
\begin{equation} 
    \sum_{M=1}^{N^2}
        1_{Sp_E}(M)
        \phi(n)^{\omega(M)}
        \prod_{k=2}^r
        (\res[\zeta_{E_k}]\delta_{(\xi\psi)_k}
        (N/\sqrt{M})^{o_k}+O((N/\sqrt{M})^{o_k\beta_k})).
\end{equation} 
This is largest when $(\xi\psi)_k = 1$ for all $k$, 
    which occurs if and only if 
    $\xi\psi \in \ker \gamma^\ast$ 
    where $\gamma$ is 
    the surjective homomorphism $P \to T$ 
    from \S\ref{sec:weak}. 
Since $\ker \gamma^\ast$ is torsion 
    and $R^\vee$ is free we must have 
    $\xi = 1$ if $\xi \psi \in \ker \gamma^\ast$. 
For any such $\xi\psi=\psi\in \ker \gamma^\ast$ 
the Dirichlet coefficient is 
\begin{equation}\label{eqn:LargestDirCoeff} 
    \res[\zeta_{E_2}]\cdots \res[\zeta_{E_r}]
    \sum_{M=1}^{N^2}
        1_{Sp_E}(M)
        \phi(n)^{\omega(M)}
        (N/\sqrt{M})^{n-1}
        ((1+
        O((N/\sqrt{M})^{\sum_k o_k \beta_k - (n-1)})).
\end{equation} 
Observe that 
\begin{equation} 
    \sum_{M=1}^{N^2}
        1_{Sp_E}(M)
        \phi(n)^{\omega(M)}
        (N/\sqrt{M})^{n-1}
        =
    N^{n-1}
    \prod_{p \in Sp_E}
    \left(1+\phi(n)p^{-\frac{n-1}{2}}\right)
\end{equation} 
and 
\begin{multline} 
    \sum_{M=1}^{N^2}
        1_{Sp_E}(M)
        \phi(n)^{\omega(M)}
        (N/\sqrt{M})^{n-1}
        O((N/\sqrt{M})^{\sum_k o_k \beta_k - (n-1)})\\
    =
    N^{\sum_k o_k \beta_k}
    \sum_{M=1}^{N^2}
        1_{Sp_E}(M)
        \phi(n)^{\omega(M)}
        O((1/\sqrt{M})^{\sum_k o_k \beta_k }).
\end{multline} 
We have that 
$\sum_k o_k \beta_k =n-1-2\sum_{k|n,k>1}(1+\phi(k))^{-1}>n-1-2\sum_{k|n,k>1}1/k\geq 2$, 
so this secondary term is bounded by 
    $O_\varepsilon(N^{\sum_k o_k \beta_k+\varepsilon})
    =O_\delta(N^{n-1-\delta})$ for some $\delta>0$. 
In conclusion, the coefficient of $N^{-t}$ 
    in the multiple Dirichlet series 
    $L(s-s_\sigma,(\xi\psi\chi)^{-1},-\sigma,D)$ 
    restricted to $\ToricIdeals_1$ 
    if $\xi=1$ and $\psi \in \ker \gamma^\ast$ is 
\begin{equation} 
    \res[\zeta_{E_2}]\cdots \res[\zeta_{E_r}]
    \prod_{p \in Sp_E}
    \left(1+\phi(n)p^{-\frac{n-1}{2}}\right)
    N^{n-1}(1+O_\delta(N^{-1-\delta})),
\end{equation} 
the coefficient of $N^{-t}$ 
    for $L(s-s_\sigma,(\xi\psi\chi)^{-1},-\sigma,D)$ 
    if $\xi=1$ and $\psi \in \ker \gamma^\ast$ is 
\begin{equation}\label{eqn:Coefficient2} 
    \res[\zeta_{E_2}]\cdots \res[\zeta_{E_r}]
    \left(
    \prod_{p}
    \sum_{n \in \im \lt_w \cap \Delta }
    p^{\frac1{e_p} \varphi_{AC}(n)}
    \right)
    N^{n-1}(1+O_\delta(N^{-1-\delta})),
\end{equation} 
while the coefficient 
    for $L(s-s_\sigma,(\xi\psi\chi)^{-1},-\sigma,D)$ 
    if $\xi\psi \not \in \ker \gamma^\ast$ is 
\begin{equation} 
    \res[\zeta_{E_2}]\cdots \res[\zeta_{E_r}]
    \left(
    \prod_{p}
    \sum_{n \in \im \lt_w \cap \Delta }
    p^{\frac1{e_p} \varphi_{AC}(n)}
    \right)
    N^{n-1}O_\delta(N^{-1-\delta}).
\end{equation} 

The other regions may be handled similarly. 
For $\sigma=-\Lambda_\ell$, $2 \leq \ell \leq r$, we find 
\begin{equation} 
    n^{-(s-s_\sigma)} = N^{-t}, 
    \quad N = N(I_\ell)^{1/o_\ell}\sqrt{M}.
\end{equation} 
Once again the coefficient of $N^{-t}$ is 
    largest if and only if 
    $\xi\psi \in \ker \gamma^\ast$, 
    in which case the coefficient is equal to 
\begin{align} 
    &\sum_{M=1}^{N^2}
        1_{Sp_E}(M)
        \phi(n)^{\omega(M)}
    \left(
    \sum_{\substack{I_1\subset O_{E_1}\\1 \leq N(I_1)\leq N}}
        1
    \right)
    \left(
    \sum_{\substack{I_\ell\subset O_{E_\ell}\\N(I_\ell)= (N/\sqrt{M})^{o_\ell}}}
        1
        \right)
    \prod_{\substack{2 \leq k \leq r\\k \neq \ell}}
    \left(
    \sum_{\substack{I_k\subset O_{E_k}\\1 \leq N(I_k)\leq (N/\sqrt{M})^{o_k}}}
    1
    \right).
\end{align} 
This is bounded by 
\begin{align} 
&\sum_{M=1}^{N^2}
    1_{Sp_E}(M)
    \phi(n)^{\omega(M)}
    N^{1+\varepsilon}
    \prod_{\substack{2 \leq k \leq r\\k \neq \ell}}
    ((N/\sqrt{M})^{o_k}+O((N/\sqrt{M})^{o_k\beta_k}))\\
    =&
\sum_{M=1}^{N^2}
    1_{Sp_E}(M)
    \phi(n)^{\omega(M)}
    N^{1+\varepsilon}
    (N/\sqrt{M})^{n-1-o_\ell}
    (1+O((N/\sqrt{M})^{\sum_{k\neq \ell,1} o_k \beta_k - (n-1-o_\ell)}))\\
    =&
    N^{n-o_\ell+\varepsilon}
\sum_{M=1}^{N^2}
    1_{Sp_E}(M)
    \phi(n)^{\omega(M)}
    (1/\sqrt{M})^{n-1-o_\ell}
    (1+O((N/\sqrt{M})^{\sum_{k\neq \ell,1} o_k \beta_k - (n-1-o_\ell)})).
\end{align} 
We have that 
\begin{align} 
\sum_{M=1}^{N^2}
    1_{Sp_E}(M)
    \phi(n)^{\omega(M)}
    (1/\sqrt{M})^{n-1-o_\ell}
    =O_\varepsilon(N^{2+\varepsilon})
\end{align} 
for any $\varepsilon>0$ and 
\begin{align} 
O(N^{\sum_{k\neq \ell,1} o_k \beta_k - (n-1-o_\ell)})
\sum_{M=1}^{N^2}
    1_{Sp_E}(M)
    \phi(n)^{\omega(M)}
    =O_\delta(N^{2-\delta})
\end{align} 
for some $\delta>0$. 
Thus the coefficient of $N^{-t}$ in the multiple Dirichlet series 
    for $\Lambda_\ell$ with $\ell \neq 1$ is 
    $O_\delta(N^{4-\delta})$ for some $\delta>0$. 

Now we determine the main term of $A(X)$. 
The summands of \eqref{eqn:ZxipsiPoisson} 
    for non-maximal cones are bounded by 
    $O(N^{n-1})$ 
    which is smaller than the secondary terms 
    $O_\delta(N^{n-\delta})$ 
    of the summands for maximal cones. 
When $\xi\psi \in \ker \gamma^\ast$, 
    the estimate \eqref{eqn:Coefficient2} 
    shows the abscissa of absolute convergence 
    of the Dirichlet series 
    $L(s-s_\sigma,(\xi\psi\chi)^{-1},-\sigma,D)$ 
    is at $u = \frac{n}{n+1}$. 
Since $R_\gamma(us_{AC})$ is regular and nonvanishing 
    (Proposition~\ref{prop:nonvanishingR}) 
    at $u = \frac{n}{n+1}$, 
    the leading coefficient of $Z(u)$ in 
        the Laurent expansion at 
        $u = \frac{n}{n+1}$ 
    is also the leading coefficient of 
\begin{equation}\label{eqn:LeadingCoefficient} 
    \frac{|\ToricROU|}{|C||R|P}
    \sum_{\sigma \in \Sigma_D(r)}
    \left(
    \sum_{\gamma \in Z_{\sigma}}
    R_\gamma(\tfrac{n}{n+1}s_{AC})
    \right)
    \sum_{\psi \in \ker \gamma^\ast}
    L(s-s_\sigma,
        \psi^{-1},-\sigma,D)
\end{equation} 
    in the asymptotic expansion 
    as $u$ tends to $\frac{n}{n+1}$ from above. 
For maximal cones 
    $\sigma \in \Sigma_D(r)$ and $\tau \in \Sigma_w(r_w)$ 
    let $\binom{\tau}{\sigma} \in \ZZ_{\geq 0}$ 
    denote the number of ordered 
    $r$-subsets 
    $E \subset \tau(1)$ such that $\pi E = \sigma(1)$ 
    (as unordered sets) 
    and the flag $\gamma$ defined by $E$ is in $Z_{\sigma}$.\footnote{Since the Jacobians of $\Sigma$ satisfy 
    the conjecture of \cite[Remark~3]{residue} 
    (cf.~\S\ref{sec:VerifyingCompat}), 
    the integer $\binom{\tau}{\sigma}$ 
    can more simply be defined as 
    the number of unordered $r$-subsets 
    $E \subset \tau(1)$ such that $\pi E = \sigma(1)$.} 
Say $\sigma = \sigma_{(j)}$. 
With the help of 
    Lemma~\ref{lemma:ResidualFunctionFormula} 
    and Proposition~\ref{prop:HelpingForMainThm} 
    one finds that 
\begin{equation} 
\sum_{\gamma \in Z_{\sigma}}
R_\gamma(\tfrac{n}{n+1}s_{AC})
=
    \left(\frac{o_j}{2n}\right)^{r_w-r}
    \sum_{\tau \in \Sigma_w(r_w)}
    \binom{\tau}{\sigma}
    \frac{[\ccl_w:\ZZ\tau(1)]}{[\ccl:\ZZ\sigma(1)]}.
\end{equation} 

The ratio of the Dirichlet series 
    \eqref{eqn:LeadingCoefficient} and $Z(u)$ 
    tends to $1$ as $u$ tends to $\frac{n}{n+1}$ from above, 
    so their summatory functions have the same main terms. 
Set 
\begin{equation} 
    \kappa_0=
    \res[\zeta_{E_2}]\cdots \res[\zeta_{E_r}]
    \prod_{p}
    \sum_{n \in \im \lt_w \cap \Delta }
    p^{\frac1{e_p} \varphi_{AC}(n)}
    \quad\text{and}\quad
    \kappa_\sigma=\left(\frac{o_j}{2n}\right)^{r_w-r}
    \sum_{\tau \in \Sigma_w(r_w)}
    \binom{\tau}{\sigma}
    \frac{[\ccl_w:\ZZ\tau(1)]}{[\ccl:\ZZ\sigma(1)]}.
\end{equation} 
From the asymptotic formula \eqref{eqn:Coefficient2} 
    for the coefficients of 
    $L(s-s_\sigma,\psi^{-1},-\sigma,D)$ 
    we conclude that the summatory function $A(X)$ 
    of $Z(us_{AC})$ satisfies 
\begin{equation} 
A(X)=
    \frac{|\ToricROU||C \cap \ker \gamma^\ast|}{|C||R|P}
    \kappa_0
    \left(
    \sum_{\sigma \in \Sigma_D(r)}
    \kappa_\sigma
    \right)
    \tfrac{1}{n}
    X^{\frac{n}{n+1}}
    (1+o(1)).
\end{equation} 

To complete the proof, we need to relate 
    $A(X)$ to the count $N(C_n,X)$ of normal $C_n$-polynomials. 
Recall the rational points of $T$ 
    are in bijection with equivalence classes 
    of pairs $(K/\QQ,x)$ 
    where $K = \mathcal O(T)$ is 
    the ring of functions on a $G$-torsor over $\Spec \QQ$ 
    and $x\in K$ is a normal element 
    (Proposition~\ref{prop:orbitParam}). 
The number of twists of $(K/\QQ,x)$ for a $C_n$-field $K$ 
    by an outer automorphism of $C_n$ is equal to $\varphi(n)$, 
    so there are $\varphi(n)$ such rational points 
    of $T$ giving rise to any $C_n$-polynomial 
    (Corollary~\ref{cor:cyclicTwists}). 

We bound the number of reducible rational points on $T$. 
Recall $\G$ is a product of simple tori 
    indexed by the $\QQ$-irreducible representations 
    of $C_n$. 
In particular, $\G = \GG_m \times \G_1$ 
where $\GG_m$ corresponds to the trivial representation. 
The monodromy group of 
    a rational point $(K/\QQ,x)$ 
    is contained in a given subgroup $D \subsetneq C_n$ 
    if and only if $(K/\QQ,x)$ is 
    the image of a rational point on 
    $\G/D$ via the natural quotient map $\G/D \to \G/G=T$. 
Let $\Upsilon=\Upsilon_D \subset T(\QQ)_{\text{$D$-integral}}
    =T(\QQ) \cap U(\ZZ)$ 
    be the subset of such points which are $D$-integral. 
By expressing this quotient map as 
$\GG_m \times \G_1/D \to \GG_m \times T_1$ 
where $T_1 = \G_1/G$ 
we see that the projection of $\Upsilon$ 
    to the $T_1$ factor of $T$ is itself 
    a thin subset $\Upsilon_{1} \subset T_1$ of type II. 
We make use of ideas from 
    \cite{browning-loughran-2019} 
    to show $\Upsilon$ has vanishing density. 
It suffices to show that it  
    has a vanishing contribution to the main term 
    \eqref{eqn:LargestDirCoeff} 
    of the Dirichlet coefficients of $Z(us_{AC})$. 
For a finite place $v$ let 
    $\Phi_v \colon T(\QQ_v)_{\text{$D$ integral}} 
    \to \{0,1\}$ 
    be the characteristic function defined by 
\begin{equation} 
    \Phi_{v}(x_v)
    =
    \begin{cases}
        1&\text{$x_v \in \Upsilon \Mod {p_v}$},\\
        0&\text{otherwise.}
    \end{cases}
\end{equation} 
Set $\Phi =\Phi_\infty \times \prod_{v \text{ finite}} \Phi_v$ 
where $\Phi_\infty$ is identically one on $T(\RR)$. 
The same computation we carried out for $Z(us_{AC})=Z(us_{AC},1_D)$ 
goes through for 
$Z(us_{AC},\Phi 1_D) 
= \sum_{t \in T(\QQ)}\frac{\Phi(t)1_D(t)}{H(t,s)}$ 
since $\Phi$ is also $\IntegralCompact{}$-invariant, 
so the Poisson formula for $Z(us_{AC},\Phi 1_D)$ 
involves the same integral over level $\IntegralCompact{}$ 
automorphic characters. 
From \eqref{eqn:LeadingDirCoeffSummed} 
we see the contribution to the leading term in 
    the Dirichlet coefficients of $Z(us_{AC})$ is 
\begin{equation}
    \sum_{M=1}^{N^2}
        1_{Sp_E}(M)
        \phi(n)^{\omega(M)}
    \prod_{k=2}^r
    \left(
    \sum_{\substack{I_k\subset O_{E_k}\\1 \leq N(I_k)\leq(N/\sqrt{M})^{o_k}}}
        \Phi(I_k)
    \right)
\end{equation} 
where $\Phi(I_k)$ is evaluated using the bijection 
between $n \in \ToricIdeals_1$ and 
    $(M,I_1,\ldots,I_r) \in Sp_E \times \mathrm{Ideals}(O_{E_1}) 
    \times \cdots \times \mathrm{Ideals}(O_{E_r})$. 
Note that \eqref{eqn:LeadingDirCoeffSummed} 
    only depends on the projection $\pi(n)$, 
    so we are really counting points in 
    the thin set $\Upsilon_1$ in $T_1$. 
We use the strategy from 
    the proof of \cite[Theorem 1.2]{browning-loughran-2019}. 
There is a constant $0<c<1$ and an extension $L/\QQ$ 
    only depending on $\Upsilon_1$ 
such that for all primes $p$ which split in $L$ 
    the number of points in $\Upsilon_1 \pmod p$ 
    is at most $cp^{\dim T_1} = cp^{n-1}$ 
    by \cite[Lemma~3.8]{browning-loughran-2019}. 
Thus \eqref{eqn:LeadingDirCoeffSummed} is bounded by 
\begin{equation}
    \sum_{M=1}^{N^2}
        1_{Sp_E}(M)
        \phi(n)^{\omega(M)}
        c^{\omega_L(N(I_2)\cdots N(I_r))}
        (N/\sqrt{M})^{n-1}
\end{equation} 
where $\omega_L(N)$ is 
    the number of distinct prime divisors of $N$ 
    which are split in $L$. 
Since $0<c<1$ this is $o(N^{n-1})$ which shows 
     these points have vanishing density. 
%

\begin{remark}\label{rmk:GeneralPolynms}
We expect the asymptotic formula of 
    Theorem~\ref{thm:mainthm} is also 
    the correct formula for the count of 
    all (monic integral) $C_n$-polynomials 
    since non-normal polynomials 
    form a thin set of $\PP/C_n$. 
Non-normal $C_n$-polynomials 
    such as $t^3 - 3t - 1$ 
    are the images of rational points 
    on the boundary locus $\PP/C_n - D$. 
This closed set is a union of toric divisors defined over $\QQ$ 
    so the distribution of its rational points is 
    amenable to toric methods in principle, 
    however these toric divisors may be geometrically reducible 
    so they are not themselves toric varieties. 
The geometrically irreducible toric divisors in $\PP/C_n - D$ 
    are toric varieties but defined over cyclotomic fields. 
We expect the methods of this paper, 
    suitably extended to more general ground fields, 
    would verify that these thin sets have vanishing density. 
We similarly expect that an analysis of 
    the height zeta function for 
    the toric orbit parametrization 
    from \S\ref{sec:DiophProblem} 
    would produce an asymptotic formula for $G$-polynomials 
    of bounded root height for any finite abelian group $G$. 
\end{remark}



\bibliography{draft}

\begin{thebibliography}{10}

\bibitem{zbMATH05322358}
A.~Barvinok.
\newblock {\em Integer points in polyhedra}.
\newblock Zur. Lect. Adv. Math. Z{\"u}rich: European Mathematical Society
  (EMS), 2008.

\bibitem{batyrev1995maninsconjecturetoricvarieties}
V.~V. Batyrev and Y.~Tschinkel.
\newblock Manin's conjecture for toric varieties, 1995.

\bibitem{IMRN}
V.~V. Batyrev and Y.~Tschinkel.
\newblock Rational points of bounded height on compactifications of anisotropic
  tori.
\newblock {\em Internat. Math. Res. Notices}, (12):591--635, 1995.

\bibitem{HZF}
V.~V. Batyrev and Y.~Tschinkel.
\newblock Height zeta functions of toric varieties.
\newblock {\em
  \url{https://math.nyu.edu/~tschinke/papers/yuri/96march/march96.pdf}}, 1996.

\bibitem{zbMATH01353487}
V.~V. Batyrev and Y.~Tschinkel.
\newblock Manin's conjecture for toric varieties.
\newblock {\em J. Algebr. Geom.}, 7(1):15--53, 1998.

\bibitem{CUBIC}
S.~Bhattacharya and A.~O'Desky.
\newblock On monic abelian trace-one cubic polynomials, 2023.

\bibitem{browning-loughran-2019}
T.~Browning and D.~Loughran.
\newblock Sieving rational points on varieties.
\newblock {\em Trans. Amer. Math. Soc.}, 371(8):5757--5785, 2019.

\bibitem{integral-chambert-tschinkel}
A.~Chambert-Loir and Y.~Tschinkel.
\newblock Integral points of bounded height on toric varieties, 2010.

\bibitem{zbMATH01765119}
H.~Cohen, F.~Diaz~y Diaz, and M.~Olivier.
\newblock On the density of discriminants of cyclic extensions of prime degree.
\newblock {\em J. Reine Angew. Math.}, 550:169--209, 2002.

\bibitem{MR291099}
P.~K.~J. Draxl.
\newblock {$L$}-{F}unktionen algebraischer {T}ori.
\newblock {\em J. Number Theory}, 3:444--467, 1971.

\bibitem{MR1397028}
G.~B. Folland.
\newblock {\em A course in abstract harmonic analysis}.
\newblock Studies in Advanced Mathematics. CRC Press, Boca Raton, FL, 1995.

\bibitem{MR1234037}
W.~Fulton.
\newblock {\em Introduction to toric varieties}, volume 131 of {\em Annals of
  Mathematics Studies}.
\newblock Princeton University Press, Princeton, NJ, 1993.
\newblock The William H. Roever Lectures in Geometry.

\bibitem{CF}
A.~O'Desky.
\newblock On trace-one generators of abelian cubic fields, 2024.

\bibitem{residue}
A.~O'Desky.
\newblock A residue formula for integrals with hyperplane singularities, 2024.

\bibitem{moge}
A.~O'Desky and J.~Rosen.
\newblock The moduli space of ${G}$-algebras.
\newblock {\em Advances in Mathematics}, 431:109240, 2023.

\bibitem{MR114817}
T.~Ono.
\newblock On some arithmetic properties of linear algebraic groups.
\newblock {\em Ann. of Math. (2)}, 70:266--290, 1959.

\bibitem{ono}
T.~Ono.
\newblock Arithmetic of algebraic tori.
\newblock {\em Ann. of Math. (2)}, 74:101--139, 1961.

\bibitem{santens}
T.~Santens.
\newblock Manin's conjecture for integral points on toric varieties, 2023.

\bibitem{shahidi}
F.~Shahidi.
\newblock {\em Eisenstein series and automorphic {{\(L\)}}-functions},
  volume~58 of {\em Colloq. Publ., Am. Math. Soc.}
\newblock Providence, RI: American Mathematical Society (AMS), 2010.

\bibitem{zbMATH00107779}
A.~K. Tsikh.
\newblock {\em Multidimensional residues and their applications. {Transl}. from
  the {Russian} by {E}. {J}. {F}. {Primrose}. {Transl}. edited by {S}.
  {Gelfand}}, volume 103 of {\em Transl. Math. Monogr.}
\newblock Providence, RI: American Mathematical Society, 1992.

\bibitem{zbMATH03401075}
V.~E. Voskresenskii.
\newblock Birational properties of linear algebraic groups.
\newblock {\em Math. USSR, Izv.}, 4:1--17, 1971.

\bibitem{zbMATH07814403}
F.~Wilsch.
\newblock Integral points of bounded height on a certain toric variety.
\newblock {\em Trans. Am. Math. Soc., Ser. B}, 11:567--599, 2024.

\end{thebibliography}
\bibliographystyle{abbrv}

\end{document}